\documentclass[12pt,a4paper,oneside,reqno,notitlepage]{amsart}
\usepackage{amssymb}

\usepackage{amsfonts}
\usepackage{amscd,amsmath}
\usepackage{euscript}
\usepackage[T2A]{fontenc}
\usepackage{comment}
\usepackage[arrow,matrix,curve]{xy}\SilentMatrices
\def\xyma{\xymatrix@M.7em}
\usepackage{float}

\usepackage[arrow,matrix,curve]{xy}\SilentMatrices
\def\xyma{\xymatrix@M.7em}

\def\lotimes{\buildrel{L}\over\otimes}

\numberwithin{equation}{section}

\newtheorem{cor}{Corollary}[section]
\newtheorem{defi}{Definition}[section]
\newtheorem{prop}{Proposition}[section]
\newtheorem{theorem}{Theorem}[section]
\newtheorem{lemma}{Lemma}[section]
\newtheorem{remark}{Remark}[section]
\newtheorem{example}{Example}[section]

\def\Le{\EuScript L}

\def\Z{{\mathbb{Z}}}

\def\la{\longrightarrow}
\def\ot{\otimes}

\def\bee{\begin{equation}}
\def\ee{\end{equation}}

\def\Tor{\mathrm{Tor}}

\begin{document}

\title{Homotopy theory of Lie functors}
\author{Roman Mikhailov}
 \maketitle

\section{Introduction}
Let $\sf Ab$ be the category of abelian groups. By the classical
Dold-Kan theorem, the Moore normalization functor
$$
N: \sf SAb\to Ch
$$
defines an equivalence between the category $\sf SAb$ of
simplicial abelian groups and the category $\sf Ch$ of
non-negatively graded chain complexes. Denote by $K: \sf Ch\to \sf
SAb$ the Dold-Kan transform inverse to the Moore normalization
functor. Also denote by $\sf DAb$ the derived category, obtained
from $\sf Ch$ by inverting the weak equivalences.

 Given a covariant functor $F: \sf
Ab\to Ab$ with $F(0)=0$, one can ask about the structure of its
derived functors, or generally speaking, about "the homotopy
theory of $F$". For an element $C\in \sf Ch$, one can ask how to
describe the graded abelian groups
$$
H_*(LF(C))=\pi_*(F(K(C))))
$$
These homotopy groups depend only on the class of the element $C$
in $\sf DAb$, hence we can ask the following: for an element $C\in
\sf DAb$, describe the groups $ H_*(LF(C)) $ as functors from $\sf
DAb$ to the category of graded abelian groups.

In the present paper, we construct the bigraded functors
\begin{align*}
& \EuScript E^m(-,n): \sf DAb\to \sf DAb\\
& \widetilde{\EuScript E}^m(-,n): \sf DAb\to \sf DAb
\end{align*}
such that, for all $C\in \sf DAb$, there are absract isomorphisms
of graded abelian groups
\begin{align}
& L_*\EuScript L^m(C[n])\simeq H_*(\EuScript E^m(C,n))\label{qui1}\\
& L_*\EuScript L_s^m(C[n])\simeq H_*(\widetilde{\EuScript
E}^m(C,n))\label{qui2}
\end{align}
Here $\EuScript L$ and $\EuScript L_s$ are graded Lie and
super-Lie functors with squares respectively (for the definition
see the next section). The homotopy groups of $\EuScript E^m(-,n)$
and $\widetilde{\EuScript E}^m(-,n)$ can be described. In
particular, this will give a way how to compute (abstractly) the
derived functors of Lie functors and super-Lie functors in all
dimensions. Moreover, we show that, in the case when $C$ is a free
abelian group, the isomorphisms (\ref{qui1}), (\ref{qui2}) are
natural. That is, this gives a complete description of all derived
functors of $\EuScript L$, $\EuScript L_s$ for free abelian
groups.

For certain degrees we will give a general functorial description
of derived functors of Lie and super-Lie functors (not only for
free abelian groups). It is shown in \cite{Schlesinger:66} that if
$p$ is an odd prime then the groups $L_{n+k} \EuScript L^p(\mathbb
Z, n)$ are $p$-torsion for all $k$, and in particular
\begin{equation}
L_{n+k}\EuScript L^p(\mathbb Z,n)=\begin{cases} \mathbb Z/p,\
k=2i(p-1)-1,\ i=1,2,\dots, [n/2]\\
0,\ \text{otherwise}\end{cases}
\end{equation}
For example, in the simplest case, our theory gives the following
functorial generalization of the above description (for $n\geq
0$):
$$
L_i\EuScript L^p(A,2n)=\begin{cases} \EuScript L^p(A)\oplus
\Tor(A,\mathbb Z/p),\ i=2np\\ L_j\EuScript L^p(A),\ i=2np+j,
j=1,\dots, p-1\\ A\otimes \mathbb Z/p,\ i=2n+2j(p-1)-1,\ j=1,2,\dots,n\\
\Tor(A,\mathbb Z/p),\ i=2n+2j(p-1),\ j=1,2,\dots,n-1\end{cases}
$$
$$
L_i\EuScript L^p(A,2n+1)=\begin{cases} L_j\EuScript L_s^p(A),\
i=2np+j,
j=0,\dots, p-1\\ A\otimes \mathbb Z/p,\ i=2n+2j(p-1),\ j=1,2,\dots,n\\
\Tor(A,\mathbb Z/p),\ i=2n+2j(p-1)+1,\ j=1,2,\dots, n\end{cases}
$$

In some cases the description of derived functors is very simple.
For example, for a free abelian group $A$, the derived functors
$L_i\EuScript L^6(A,1)$ and $L_i\EuScript L_s^6(A,1)$ are the
following:
$$
L_i\EuScript L^6(A,1)=\begin{cases} \EuScript L_s^6(A),\ i=6\\
\Gamma_2(A)\otimes \mathbb Z/3,\ i=5\\ \EuScript L_s^3(A)\otimes
\mathbb Z/2,\ i=4\\ 0,\ i\neq 4,5,6\end{cases}\ \ \ \
L_i\EuScript L_s^6(A,1)=\begin{cases} \EuScript L^6(A),\ i=6\\
\Lambda^2(A)\otimes \mathbb Z/3\oplus \EuScript L^3(A)\otimes \mathbb Z/2,\ i=5\\
\EuScript L^3(A)\otimes \mathbb Z/2,\ i=4\\ 0,\ i\neq 3,
5,6\end{cases}
$$
In a more general case, for a free abelian $A$, if $m$ is a square
free number, i.e. $m$ is a product of different primes, then there
is the following description of the derived functors:
$$
\bigoplus_{i=1}^{2nm}L_i\EuScript L^m(A,2n)[i]\simeq \EuScript
L^m(A)[2nm]\oplus\\
\bigoplus_{\substack{p\ \text{prime}\\ p|m}}\
\bigoplus_{i=1}^{\frac{mn}{p}}\EuScript L^{\frac{m}{p}}(A)\otimes
\mathbb Z/p\ [\frac{2mn}{p}+(2p-2)i-1]
$$
To describe the derived functors in a general case, we introduce
the collection of special functors $\EuScript N^{k;p},\EuScript
N^{k;p}: \sf Ab\to Ab$ $(p\ \text{prime},\ k\geq 1)$. We call the
collection of these functors "hierarchies" by the following
reason. For any prime $p$ and $k\geq 1$, there are chains of
natural epimorphisms
$$
\EuScript N^{kp^i;p}\twoheadrightarrow \dots\EuScript
N^{kp;p}\twoheadrightarrow\EuScript N^{k;p}
$$
which split abstractly, but not naturally, moreover, every natural
transformation of the type $\EuScript N^{k;p}\to \EuScript
N^{kp;p}$ is the zero map. These functors play a crucial role in
the description of derived functors of Lie and super-Lie functors.
For example, the functor $L_9\EuScript L^9(A,2)$ is naturally
isomorphic to the functor $\EuScript N^{3;3}(A)$ and can be
presented in the short exact sequence
$$
0\to \EuScript L^3(A)\otimes \mathbb Z/3\to L_9\EuScript
L^9(A,2)\to A\otimes \mathbb Z/3\to 0
$$
which splits as a sequence of abelian groups and does not split as
a sequence of functors. Moreover, every natural transormation
$A\otimes \mathbb Z/3\to L_9\EuScript L^9(A,2)$ is the zero map.

\vspace{.5cm}\noindent{\bf Notation.}\\ \\
$\Lambda^n: \sf Ab\to \sf Ab$ the $n$th exterior power;\\ \\
$SP^n: \sf Ab\to \sf Ab$ the $n$th symmetric power;\\ \\
$\Gamma_n: \sf Ab\to \sf Ab$ the $n$th divided power;\\ \\
$\EuScript L^n: \sf Ab\to Ab$ the $n$th graded Lie functor;\\ \\
$\EuScript L_s^n: \sf Ab\to \sf Ab$ the $n$th super-Lie functor with squares;\\ \\
$\mathbb L: \sf Ch\to \sf Ch$ the universal differential graded
Lie algebra with
squares (see (\ref{universaldgls}));\\ \\
$\EuScript N^{n;p}: \sf Ab\to Ab$ the special functors (see
(\ref{specfu}), (\ref{specfu1}));\\ \\
$J^n,Y^n:\sf Ab\to Ab$ the special Schur functors (see
(\ref{schur}), (\ref{schur1})) which coincide with metabelian
$n$th
Lie and super-Lie functors;\\ \\
$\EuScript E^m(-,n),\ \widetilde{\EuScript E}^m(-,n): \sf DAb\to
\sf DAb$ the $\EuScript E$-functors (see \ref{ecom});\\ \\

\section{Graded Lie rings}
The tensor algebra $\otimes A$ is endowed with  a $\mathbb Z$-Lie
algebra structure, for which the  bracket operation is defined by
$$[a,\,b]=a\otimes b-b\otimes a,\quad  a,b\in \otimes (A).$$
One   defines  $n$-fold brackets inductively by setting
\begin{equation}
\label{def:lie-br}
 [a_1, \ldots, a_n] := [[a_1\ldots, a_{n-1}],a_n]
\end{equation}
We will  denote $\ot A$, viewed as a  $\mathbb Z$-Lie algebra,  by
$\otimes(A)^{Lie}.$  Let  $\EuScript L(A)=\bigoplus_{n\geq
1}\EuScript L^n(A)$ be the sub-Lie ring of $\otimes(A)^{Lie}$
generated by $A$. Its  degree 2 and  3 components are  generated
by the expressions
\begin{equation}
\label{lie-3}a\otimes b-b\otimes a \qquad \text{and} \qquad  a
\otimes b \otimes c - b \otimes a \otimes c + c \otimes a \otimes
b - c \otimes b \otimes a
\end{equation}
where $a,b,c \in A$. $\EuScript L(A) $ is called the {\it free Lie
ring generated by the abelian group $A$}. It is  universal for
homomorphisms from $A$ to $\Z$-Lie algebras. The grading of
$\otimes A$ determines  a grading   on $\EuScript L(A)$, so that
we obtain a family  of endofunctors
 on the category of abelian groups:
$$
\EuScript  L^i: {\sf Ab}\to {\sf Ab},\ i\geq 1.
$$
The universal property of the Lie functor implies that there is a
natural transformation of the graded functors:
$$
\EuScript L\EuScript L\to \EuScript L
$$
which assigns to an abelian group $A$, the unique map
$$
\EuScript L\EuScript L(A)\to \EuScript L(A)
$$
which is the identity on $\EuScript L_1\EuScript L(A)=\EuScript
L(A)$.
\begin{defi} \cite{Leibowitz}
A {\it graded Lie ring with squares}  (GLRS for short) is a graded
abelian group $B=\bigoplus_{i=0}^\infty B_i$ with homomorphisms
\begin{align} & \{\ ,\ \}: B_i\otimes B_j\to B_{i+j},\ \label{superbracket}\\ & ^{[2]}: B_n\to
B_{2n}\ \text{for}\ n\ \text{odd}
\end{align}
such that the following conditions are satisfied (for elements
$x\in B_i,\ y\in B_j,\ z\in B_k$):
\begin{align}
& 1)\ \{x,y\}+(-1)^{ij}\{y,x\}=0 \label{santi}\\
& 2)\ \{x,x\}=0\ \ \text{for}\ i\ \text{even} \notag\\
& 3)\
(-1)^{ik}\{\{x,y\},z\}+(-1)^{ji}\{\{y,z\},x\}+(-1)^{kj}\{\{z,x\},y\}=0\label{sjacob}\\
& 4)\ \{x,x,x\}=0 \notag \\
& 5)\ (ax)^{[2]}=a^2x^{[2]}\ \text{for}\ \ i\ \text{odd},\ a\in
\mathbb Z  \notag   \\
& 6)\ (x+y)^{[2]}=x^{[2]}+y^{[2]}+\{x,y\}\ \text{for}\ \ i=j\
\text{odd}  \notag  \\
& 7)\ \{y,x^{[2]}\}=\{y,x,x\}\ \ \text{for}\ i\
\text{odd}.\label{yxx}
\end{align}
\end{defi}
For an abelian group $A$, define $\EuScript L_s(A)$ to be the
graded Lie ring with squares freely generated by $A$ in degree 1.
It may be defined as a GLRS together with a homomorphism of
abelian groups $l: A\to \EuScript L_s(A)$ such that for every map
$f: A\to B$ with $B$ a GLRS, there is a unique morphism of GLRS
$d: \EuScript L_s(A)\to B$ such that $f=d\circ l$. The abelian
group $\EuScript L_s(A)$ is naturally graded by  $\EuScript
L_s(A)=\bigoplus_{n=1}^\infty \EuScript L_s^n(A)$ and for any $x
\in \Le_s(A)$, we set $|x|= n$ whenever $x \in \Le^n_s(A)$.

\begin{theorem}\label{schles} {\bf (Schlesinger)} Let $A_1,A_2,\dots, A_n$ be free abelian groups.
There is a natural isomorphism
$$
\EuScript L(\bigoplus_{i=1}^n A_i)=\bigoplus_{i=1}^n \EuScript
L(A_i)\oplus \bigoplus_{J}\EuScript L(A_J),
$$
where $J=(j_1,\dots, j_k),\ j_1>j_2\leq \dots \leq j_k,\ k\geq 2$
and $A_J=A_{j_1}\otimes A_{j_2}\otimes \dots \otimes A_{j_k}$.
\end{theorem}
The map
$$
A_J\to \EuScript L(\bigoplus_{i=1}^nA_i),\ J=(j_1,\dots, j_k)
$$
is given by
$$
a_{j_1}\otimes \dots \otimes a_{j_k}\mapsto
[a_{j_1},a_{j_2},\dots,a_{j_k}],\ a_{j_i}\in A_{j_i}.
$$
For example, for $n=2$ and free abelian groups $A$, $B$, we have
the following decomposition:
\begin{multline*}
\EuScript L(A\oplus B)=\EuScript L(A)\oplus \EuScript L(B)\oplus
\EuScript L(B\otimes A)\oplus \EuScript L(B\otimes A\otimes
A)\oplus \EuScript L(B\otimes A\otimes B)\oplus\\ \EuScript
L(B\otimes A\otimes A\otimes A)\oplus \EuScript L(B\otimes
A\otimes A\otimes B)\oplus \EuScript L(B\otimes A\otimes B\otimes
B)\oplus \dots
\end{multline*}
This gives a way to compute all cross-effects of the graded
components of Lie functors:
\begin{equation}\label{crefe}
\EuScript L^m(A\oplus B)=\bigoplus_{d|m,\ 1\leq d\leq
m}\bigoplus_{C\in J_{m/d}}\EuScript L^d(C),
\end{equation}
where $J_{m/d}$ is the set of all basic tensor products of weight
$m/d$ in $A$ and $B$. For example, \begin{align*} & \EuScript
L^2(A\oplus
B)=\EuScript L^2(A)\oplus\EuScript L^2(B)\oplus A\otimes B\\
& \EuScript L^3(A\oplus B)=\EuScript L^3(A)\oplus \EuScript
L^3(B)\oplus (A\otimes B\otimes B)\oplus (A\otimes B\otimes A)\\
& \EuScript L^4(A\oplus B)=\EuScript L^4(A)\oplus \EuScript
L^4(B)\oplus \EuScript L^2(A\otimes B)\oplus (A\otimes B\otimes
B\otimes B)\oplus\\ & \ \ \ \ \ \ \ \ \ \ \ \ \ \ \ \ \ \ \ \ \ \
\ \ \ \ \ \ \ \ \ \ \ \ \ \ \ \ \ \ \ \ \ \ \ \ \ \ \ \ \ \
(A\otimes B\otimes B\otimes A)\oplus (A\otimes B\otimes A\otimes
A)
\end{align*}

The number of all basic products of $r$ modules $A_1,\dots, A_r$
of weight $m$ is given by the following formula:
$$
M_r(m)=\frac{1}{m}\sum_{d|m}\mu(d)r^{\frac{m}{d}}
$$
where $\mu(d)$ is the M\"obius function. The number of all basic
products of $r$ modules $A_1,\dots, A_r$ with $m_i$ entries in
$A_i$ (i.e. $m=m_1+\dots+m_r$) is given by the following formula:
$$
M(m_1,\dots,
m_r)=\frac{1}{m}\sum_{d|m_i}\mu(d)\frac{\left(\frac{m}{d}\right)!}{\left(\frac{m_1}{d}\right)!\dots
\left(\frac{m_r}{d}\right)!}
$$
By definition, we have
$$
|J_{m/d}|=M_2(m/d)
$$
in the formula (\ref{crefe}).

 The cross-effects of graded components of
super-Lie functors with squares are the following:
\begin{prop}
For free abelian $A$ and $B$, one has
\begin{equation}\label{screfe}\EuScript L_s^m(A\oplus
B)=\bigoplus_{\substack{d|m,\ 1\leq d\leq m\\ m/d\
\text{odd}}}\bigoplus_{C\in J_{m/d}}\EuScript L_s^d(C)\oplus
\bigoplus_{\substack{d|m,\ 1\leq d< m\\ m/d\
\text{even}}}\bigoplus_{C\in J_{m/d}}\EuScript L^d(C)
\end{equation} where $C$ runs over all basic tensor products of
weight $m/d$ in $A$ and $B$.
\end{prop}
For example, \begin{align*} & \EuScript L_s^2(A\oplus
B)=\EuScript L_s^2(A)\oplus\EuScript L_s^2(B)\oplus A\otimes B\\
& \EuScript L_s^3(A\oplus B)=\EuScript L_s^3(A)\oplus \EuScript
L_s^3(B)\oplus (A\otimes B\otimes B)\oplus (A\otimes B\otimes A)\\
& \EuScript L_s^4(A\oplus B)=\EuScript L_s^4(A)\oplus \EuScript
L_s^4(B)\oplus \EuScript L^2(A\otimes B)\oplus (A\otimes B\otimes
B\otimes B)\oplus\\ & \ \ \ \ \ \ \ \ \ \ \ \ \ \ \ \ \ \ \ \ \ \
\ \ \ \ \ \ \ \ \ \ \ \ \ \ \ \ \ \ \ \ \ \ \ \ \ \  (A\otimes
B\otimes B\otimes A)\oplus (A\otimes B\otimes A\otimes A)
\end{align*}

\subsection{Curtis decomposition}  Consider the Schur functors $$J^n, Y^n: {\sf Ab}\to {\sf Ab},\
n\geq 2$$ defined by \footnote{The functors $Y^n(A)$ are the
$\Z$-forms of the Schur functors $\mathbb{S}_{\lambda}(V)$
associated  to the partition $\lambda= (2, 1\ldots,1)$ of the set
$(n)$ (see \cite{ful-har} exercise 6.11). The functors $J^n(A)$
and the $\Z$- forms of the Schur functors $\mathbb{S}_\mu$
associated to the partition $\mu = (n-1,1)$ of $(n)$, which is the
conjugate
 partition of $\lambda$.}
\begin{align}
& J^n(A)=ker\{A\otimes SP^{n-1}(A)\to SP^n(A)\},\ n\geq 2,\label{schur}\\
& Y^n(A)=ker\{ A\otimes \Lambda^{n-1}(A)\to \Lambda^n(A)\},\ n\geq
2\label{schur1}
\end{align}

Curtis gave in  \cite{Curtis:63}  a decomposition of the functors
$\EuScript L^m(A)$ in terms of functors $SP^m$, $J^m$ and their
iterates. Analogous decomposition exists also in the super-Lie
case (see \cite{BreenMikhailov}). For a free abelian group $A$,
there are natural exact sequences
\begin{align}
& 0 \to \widetilde J^m(A)\to \EuScript L^m(A)\buildrel{p_m}\over\to J^m(A)\to 0 \label{cx1}\\
& 0 \to \widetilde Y^m(A)\to \EuScript L_s^m(A)\buildrel{\bar
p_m}\over\to Y^m(A)\to 0\label{cx2}
\end{align}
where \begin{align*} & p_m:\ [a_1,\dots, a_m]\mapsto a_1\otimes
a_2\dots a_m-a_2\otimes a_1a_3\dots a_m,\\
& \bar p_m: \{a_1,\dots,a_m\}\mapsto a_1\otimes a_2\wedge
a_3\wedge\dots \wedge a_m+a_2\otimes a_1\wedge a_3\wedge\dots
\wedge a_m\\
& \bar p_2: a_1^{[2]}\mapsto a_1\otimes a_1,\\
& \bar p_{2m}: \{a_1,\dots, a_m\}^{[2]}\mapsto 0,\ \text{if}\ m\
\text{is odd}
\end{align*}
for $a_i\in A$. Here $\widetilde J^m(A)$ and $\widetilde Y^m(A)$
are defined as kernels of natural projections $p_m$ and $\bar
p_m$. For low degrees the sequence (\ref{cx1}) is the following:
\begin{align*}
& \EuScript L^2(A)\buildrel{p_2}\over\simeq J^2(A)\\
& \EuScript L^3(A)\buildrel{p_3}\over\simeq J^3(A)\\ & 0\to
\Lambda^2\Lambda^2(A)\to
\EuScript L^4(A) \stackrel{p_4}{\to} J^4(A) \to 0\\
 & 0\to \Lambda^2(A)\otimes J^3(A)\to \EuScript L^5(A)    \stackrel{p_5}{\to}  J^5(A)
\to 0,
\end{align*}
where the left-hand arrows are respectively defined by
\begin{align*}
(a\wedge b) \wedge (c\wedge d) &\mapsto [[a,b],[c,d]] \\
(a\wedge b) \otimes (c,d,e) &\mapsto [[a,b],[c,d,e]]\,.
\end{align*}
The super-analogs of Curtis decomposition can be constructed
analogously (see \cite{BreenMikhailov}). In low degrees the
sequence (\ref{cx2}) is the following:
\begin{align*}
& \EuScript L_s^2(A)\buildrel{\bar p_2}\over\simeq Y^2(A)\\
& \EuScript L_s^3(A)\buildrel{\bar p_3}\over\simeq Y^3(A)\\
& 0\to \Lambda^2Y^2(A) \to \EuScript L_s^4(A) \buildrel{\bar
p_4}\over\to Y^4(A)\to 0\\
& 0\to Y^2(A)\otimes Y^3(A)\to \EuScript L_s^5(A)\buildrel{\bar
p_5}\over\to Y^5(A)\to 0
\end{align*}
\vspace{.5cm}
\section{Derived functors}
\vspace{.5cm}
Let $A$ be an abelian group, and $F$
an endofunctor on the category of abelian groups. Recall that for
every  $n\geq 0$ the derived functor of $F$
 in the sense
of Dold-Puppe  \cite{DoldPuppe} are defined by
$$
L_iF(A,n)=\pi_i(FKP_\ast[n]),\ i\geq 0
$$
where $P_\ast \to A$ is a projective resolution of $A$, and
 $K$ is  the Dold-Kan transform,  inverse to the Moore normalization  functor
\[
N:  \mathrm{Simpl}({\sf Ab}) \to C({\sf Ab})
\]
from simplicial abelian groups to chain complexes. We denote by
$LF(A,n)$ the object $FK(P_\ast[n])$ in the homotopy category of
simplicial abelian groups determined by $FK(P_\ast[n])$, so that
\[ L_iF(A,n) = \pi_i(LF(A,n))\,.\]
  We
set $LF(A) :=LF(A,0)$ and  $L_iF(A):= L_iF(A,0)$ for any $\ i\geq
0$.

As examples of these  constructions, observe that the simplicial
models $LF(L\la M)$  of $LFA$ and  $FK((L\la M) [1])$ of $LF(A,1)$
associated to the two-term flat resolution
\begin{equation}\label{res11} 0\to L \buildrel{f}\over\to M \to A\to 0\end{equation}
 of an abelian group $A$ are respectively of the following form in low
 degrees:
\bee \label{lowdeg0}
\begin{matrix}
F(s_0(L)\oplus s_1(L) \oplus s_1s_0(M))
\end{matrix}
\begin{matrix}\buildrel{\partial_0,\partial_1,\partial_2}\over\longrightarrow\\[-3.5mm]\longrightarrow\\[-3.5mm]\longrightarrow\\[-3.5mm]\longleftarrow\\[-3.5mm]
\longleftarrow
\end{matrix}\
F(L\oplus s_0(M))
\begin{matrix}\buildrel{\partial_0,\partial_1}\over\longrightarrow\\[-3mm]\longrightarrow\\[-3mm]\longleftarrow\end{matrix}\
 F(M)\ee
 where the component $F(M)$ is  in degree zero, and
\bee \label{lowdeg}
\begin{matrix}
F(s_0(L)\oplus s_1(L)\oplus s_2(L)\oplus\\
s_1s_0(M)\oplus s_2s_0(M)\oplus s_2s_1(M))
\end{matrix}
\begin{matrix}\buildrel{\partial_0,\dots,\partial_3}\over\longrightarrow\\[-3.5mm]\longrightarrow\\[-3.5mm]\longrightarrow\\[-3.5mm]\longrightarrow\\[-3.5mm]\longleftarrow\\[-3.5mm]
\longleftarrow\\[-3.5mm]\longleftarrow
\end{matrix}\
F(L\oplus s_1(M)\oplus s_0(M))
\begin{matrix}\buildrel{\partial_0,\partial_1,\partial_2}\over\longrightarrow\\[-3mm]\longrightarrow\\[-3mm]\longrightarrow\\[-3mm]\longleftarrow\\[-3mm]\longleftarrow\end{matrix}\
 F(M)\ee
where
 the component
 $F(M)$ is  in degree 1.

 \subsection{Homotopy operations} For a simplicial abelian group $X$, $i,k,l,q\geq 1$, the composition is given as a
 natural map
 \begin{equation}\label{cmaps}
L_i\EuScript L^k(\mathbb Z,q)\otimes \pi_q\EuScript L^l(X)\to
\pi_i\EuScript L^{kl}(X)
 \end{equation}
It is constructed as follows. Let $\alpha\in \pi_q\EuScript
L^l(X)$. Take a map $f: K(\mathbb Z,q)\to \EuScript L^l(X)$ which
represents $\alpha$ in the homotopy group $\pi_q$. This map
induces the composition map
$$
L\EuScript L^k(\mathbb Z,q)\to L(\EuScript L^k\circ\EuScript
L^l)(X)\to L\EuScript L^{kl}(X)
$$
Taking the $i$-th homotopy groups we obtain the map (\ref{cmaps}).
\vspace{.5cm}
\section{Allowable sets}
\vspace{.5cm}
The description of allowable sets is given in \cite{Kan}.

Let $n\geq 1$, $k\geq 1$. We call a
sequence $(i_1,\dots, i_k)$ allowable with respect to $2n$ if\\ \\
1) $i_1\leq 2n,\
i_{j+1}\leq 2i_j$ for all $j=1,\dots, k-1$\\ 2) $i_k$ is odd.\\

Denote the set of allowable sequences with respect to $2n$ of the
length $k$ by $\mathcal W_{2n,k}$ (or $\mathcal
W_{2n,k}^{(2)}:=\mathcal W_{2n,k}$). For $k\geq 2$, define the
filtration
$$
\mathcal W_{2n,k}^{(2)}=\mathcal W_{2n,k}^{(2)}(1)\supset \mathcal
W_{2n,k}^{(2)}(2)\supset\dots\supset \mathcal W_{2n,k}^{(2)}(k)
$$
as follows: the subset $ \mathcal W_{2n,k}^{(2)}(j)$ consists of
allowable sequences $(i_1,\dots, i_k)\in \mathcal W_{2n,k}^{(2)}$
with $i_1=2n, i_2=4n,\dots, i_{j-1}=2^{j-1}n$.

\noindent{\bf Example.} \begin{align*} \mathcal
W_{2,3}^{(2)}(1)=\{& (1,1,1),\ (2,1,1),\ (2,2,1),\ (1,2,3),\
(1,2,1)\\ & (2,3,1),\ (2,2,3),\ (2,3,3),\ (2,4,3),\\ & (2,3,5),\
(2,4,5),\ (2,4,7),\
(2,4,1)\}\\
\mathcal W_{2,3}^{(2)}(2)=\{& (2,1,1),\ (2,2,1),\ (2,3,1),\
(2,2,3),\ (2,3,3),\\ & (2,4,3),\ (2,3,5),\ (2,4,5),\ (2,4,7),\
(2,4,1)\}\\
\mathcal W_{2,3}^{(2)}(3)=\{& (2,4,1),\ (2,4,3),\ (2,4,5),\
(2,4,7)\}\\
\end{align*}

For $w=(i_1,\dots,i_k)\in \mathcal V_{2n,k}$, let $o(w)$ be the
number of odd elements in $w$ and $d(w)=i_1+\dots+i_k$.

Let $p$ be an odd prime. Define the set $\mathcal W_{2n,k}^{(p)}$
as follows. The set $\mathcal W_{2n,k}^{(p)}$ consists of
sequences $(\nu_{i_1}\dots \nu_{i_k})$ with $\nu_{i_j}$ equal
either $\lambda_{i_j}$ or $\nu_{i_j}=\mu_{i_j}$, and such that
$i_1\leq n,\ i_{j+1}\leq pi_j-1$ whenever
$\nu_{i_j}=\lambda_{i_j}$ and $i_{j+1}\leq pi_j,$ whenever
$\nu_{i_j}=\mu_{i_j}$ and $\nu_{i_k}=\lambda_{i_k}$. For $k\geq
2$, define the filtration
$$
\mathcal W_{2n,k}^{(p)}=\mathcal W_{2n,k}^{(p)}(1)\supset \mathcal
W_{2n,k}^{(p)}(2)\supset\dots\supset \mathcal W_{2n,k}^{(p)}(k)
$$
as follows: the subset $ \mathcal W_{2n,k}^{(p)}(j)$ consists of
allowable sequences $(\nu_{i_1},\dots, \nu_{i_k})\in \mathcal
W_{2n,k}^{(p)}$ with $\nu_{i_1}=\mu_{n}, \nu_{i_2}=\mu_{pn},\dots,
\nu_{i_{j-1}}=\mu_{p^{j-1}n}$.

\vspace{.35cm}\noindent{\bf Example.}
\begin{align*}
\mathcal W_{2,2}^{(3)}(1)=\{& (\lambda_1,\lambda_1),\
(\lambda_1,\lambda_2),\ (\mu_1,\lambda_1),\ (\mu_1,\lambda_2),\
(\mu_1,\lambda_3)\}\\
\mathcal W_{2,2}^{(3)}(2)=\{& (\mu_1,\lambda_1),\
(\mu_1,\lambda_2),\
(\mu_1,\lambda_3)\}\\
\mathcal W_{4,2}^{(3)}(1)=\{& (\lambda_1,\lambda_1),\
(\lambda_1,\lambda_2),\ (\mu_1,\lambda_1),\ (\mu_1,\lambda_2),\
(\mu_1,\lambda_3), (\lambda_2,\lambda_1),\\ & (\lambda_2,\lambda_2),\ (\lambda_2,\lambda_3),\ (\lambda_2,\lambda_4), (\lambda_2,\lambda_5),\ \\
& (\mu_2,\lambda_1),\ (\mu_2,\lambda_2),\ (\mu_2,\lambda_3),\ (\mu_2,\lambda_4),\ (\mu_2,\lambda_5),\ (\mu_2,\lambda_6)\}\\
\mathcal W_{4,2}^{(3)}(2)=\{& (\mu_2,\lambda_1),\ (\mu_2,\lambda_2),\ (\mu_2,\lambda_3),\ (\mu_2,\lambda_4),\ (\mu_2,\lambda_5),\ (\mu_2,\lambda_6)\}\\
\end{align*}

For $\nu\in \mathcal W_{2n,k}^{(p)},$ let $o(\nu)$ be the number
of entries of $\lambda_{i_j}$ in $\nu$ and
$$d(\nu)=(2p-2)(i_1+\dots+i_k)-o(\nu).$$

\subsection{} We will need also the following notation. Let $p$ be
an odd prime
$$
\widetilde {\mathcal W}_{2n,k}^{(p)}\subseteq \mathcal
W_{2n,k}^{(p)}
$$
consists of all sequences $(\nu_{i_1},\dots, \nu_{i_k})\in\mathcal
W_{2n,k}^{(p)}$ such that $i_k>\frac{p^{k-1}-1}{2}.$ For $p=2$, $
\widetilde {\mathcal W}_{2n,k}^{(2)}$ consists of all sequences
$(i_1, \dots, i_k)\in \mathcal W_{2n,k}^{(2)}$ such that $i_k\geq
2^{k-1}$. \vspace{.5cm}
\section{Derived functors $L_*\EuScript L(\mathbb
Z,m)$}\label{section4} \vspace{.5cm}
\begin{theorem}\label{bousfield}{\bf (Bousfield)} For $n>0$ and a pointed
simplicial set $Y$, the natural map
$$
\pi_*(\EuScript LK(\mathbb Z,2n)\otimes \mathbb Z[Y])\to
\pi_*\EuScript L(K(\mathbb Z,2n)\otimes \mathbb Z[Y])
$$
is a monomorphism onto a direct summand.
\end{theorem}
For the case of the $k$-dimensional simplicial sphere $Y=S^k,$ one
has the following:
\begin{cor}
For $n,k>0$, the $k$-fold suspension map $$L_*\EuScript L(\mathbb
Z,2n)\to L_{*+k}\EuScript L(\mathbb Z,2n+k)
$$
is a monomorphism onto a direct summand.
\end{cor}

\noindent{\it Proof of theorem \ref{bousfield}.} For $m\geq 1,$
consider the natural map of bisimplicial groups
$$
(\EuScript LK(\mathbb Z,m)\otimes \mathbb Z[Y])\to \EuScript
L(K(\mathbb Z,m)\otimes \mathbb Z[Y])
$$
By the Eilenberg-MacLane-Cartier theorem it is enough, for the
analysis of the induced map on homotopy groups, to consider the
corresponding map of diagonals of these bisimplicial groups:
\begin{center}
\begin{tabular}{ccccccccccc}
$\dots$ & $\begin{matrix}\longrightarrow\\[-3.5mm]\ldots\\[-2.5mm]\longrightarrow\\[-3.5mm]\longleftarrow\\[-3.5mm]
\ldots\\[-2.5mm]\longleftarrow
\end{matrix}$ & $\EuScript L(\mathbb Z,m)_3\otimes \mathbb Z[Y]_3$ & $\begin{matrix}\longrightarrow\\[-3.5mm]\longrightarrow\\[-3.5mm]\longrightarrow\\[-3.5mm]\longrightarrow\\[-3.5mm]\longleftarrow\\[-3.5mm]
\longleftarrow\\[-3.5mm]\longleftarrow
\end{matrix}$ & $\EuScript L(\mathbb Z,m)_2\otimes \mathbb Z[Y]_2$ & $\begin{matrix}\longrightarrow\\[-3.5mm] \longrightarrow\\[-3.5mm]\longrightarrow\\[-3.5mm]
\longleftarrow\\[-3.5mm]\longleftarrow \end{matrix}$ & $\EuScript L(\mathbb Z,m)_1\otimes \mathbb Z[Y]_1$\\
 & & $\downarrow$ & & $\downarrow$ & & $\downarrow$\\
$\dots$ & $\begin{matrix}\longrightarrow\\[-3.5mm]\ldots\\[-2.5mm]\longrightarrow\\[-3.5mm]\longleftarrow\\[-3.5mm]
\ldots\\[-2.5mm]\longleftarrow
\end{matrix}$ & $\EuScript L(K(\mathbb Z,m)_3\otimes \mathbb Z[Y]_3)$ & $\begin{matrix}\longrightarrow\\[-3.5mm]\longrightarrow\\[-3.5mm]\longrightarrow\\[-3.5mm]\longrightarrow\\[-3.5mm]\longleftarrow\\[-3.5mm]
\longleftarrow\\[-3.5mm]\longleftarrow
\end{matrix}$ & $\EuScript L(K(\mathbb Z,m)_2\otimes \mathbb Z[Y]_2)$ & $\begin{matrix}\longrightarrow\\[-3.5mm] \longrightarrow\\[-3.5mm]\longrightarrow\\[-3.5mm]
\longleftarrow\\[-3.5mm]\longleftarrow \end{matrix}$ & $\EuScript L(K(\mathbb Z,m)_1\otimes \mathbb Z[Y]_1)$
\end{tabular}
\end{center}
\begin{equation}\label{simpldiag}
\end{equation}
Theorem \ref{schles} implies that the maps
$$
\EuScript L(\mathbb Z,m)_k\otimes \mathbb Z[Y]_k\to \EuScript
L(K(\mathbb Z,m)_k\otimes \mathbb Z[Y]_k)
$$
are monomorphisms onto direct summands, since $\mathbb Z[Y]_k$ is
a free abelian group with a basis $Y_k$.

For $k\geq 1$, fix an order of elements of $Y_k$ and denote by
$\mathcal D_k$ the set of all sequences $(\sigma_1,\dots,
\sigma_l),\ \sigma_i\in Y_k,\ l\geq 2$ such that
$\sigma_1>\sigma_2\leq \dots \leq \sigma_l$. Theorem \ref{schles}
implies that
$$
\EuScript L(K(\mathbb Z,m)\otimes \mathbb Z[Y]_k)\simeq \EuScript
LK(\mathbb Z,m)\otimes \mathbb Z[Y]_k\oplus \EuScript
L(\bigoplus_{(\sigma_1,\dots,\sigma_l)\in \mathcal D_k}K(\mathbb
Z,m)^{\otimes l})
$$
For every $(\sigma_1,\dots, \sigma_l)\in \mathcal D_k$, consider
the simplicial abelian subgroup $$A_{(\sigma_1,\dots,
\sigma_l)}(\simeq K(\mathbb Z,lm))\hookrightarrow\EuScript
L(K(\mathbb Z,m)\otimes \mathbb Z[Y]_k) $$ generated by the
element $[[i_m\otimes \sigma_1,\dots, i_m\otimes \sigma_l]].$
There is a monomorphism which is a homotopy equivalence
$$
\EuScript L(\bigoplus_{(\sigma_1,\dots, \sigma_l)\in \mathcal
D_k}A_{(\sigma_1,\dots, \sigma_l)})\hookrightarrow \EuScript
L(\bigoplus_{(\sigma_1,\dots, \sigma_l)\in \mathcal D_k}K(\mathbb
Z,m)^{\otimes l})
$$
Hence, there is the following monomorphism which induces a
homotopy equivalence
$$
(\EuScript LK(\mathbb Z,m)\otimes \mathbb Z[Y]_k)\oplus \EuScript
L(\bigoplus_{(\sigma_1,\dots, \sigma_l)\in \mathcal
D_k}A_{(\sigma_1,\dots, \sigma_l)})\to \EuScript L(K(\mathbb
Z,m)\otimes \mathbb Z[Y]_k)
$$
Let $m=2n$. Then the simplicial abelian group $\EuScript
L(\bigoplus_{(\sigma_1,\dots, \sigma_l)\in \mathcal
D_k}A_{(\sigma_1,\dots, \sigma_l)})$ may be viewed as a free
simplicial Lie subring in $\EuScript L(K(\mathbb Z,m)\otimes
\mathbb Z[Y]_k)$ generated by all products of the form
$[[i_m\otimes \sigma_1,\dots, i_m\otimes \sigma_l]]$ of weight
$>1$.

\begin{remark}
The suspension $L_*\EuScript L(\mathbb Z,m)\to L_{*+1}\EuScript
L(\mathbb Z,m+1)$ is not a monomorphism for an odd $m$. For
example, we have the following:
$$
\xyma{L_2\Lambda^2(\mathbb Z,1) \ar@{=}[d] \ar@{->}[r] &
L_3\Lambda^2(\mathbb Z,2)\ar@{=}[d]\\ \mathbb Z\ar@{->>}[r] &
\mathbb Z/2}
$$
\end{remark}

Consider the simplicial circle $S^1=\Delta[1]/\partial\Delta[1]$:
$$
S_0^1=\{*\},\ S_1^1=\{*,\sigma\},\ S_2^1=\{*, s_0\sigma,
s_1\sigma\},\dots, S_n^1=\{*, x_0,\dots, x_n\},
$$
where $x_i=s_n\dots \hat s_i\dots s_0\sigma$. Take the standard
abelian simplicial model of $K(\mathbb Z,1)$:
\begin{center}
\begin{tabular}{ccccccccccc}
$\dots$ & $\begin{matrix}\longrightarrow\\[-3.5mm]\ldots\\[-2.5mm]\longrightarrow\\[-3.5mm]\longleftarrow\\[-3.5mm]
\ldots\\[-2.5mm]\longleftarrow
\end{matrix}$ & $\mathbb Z\oplus \mathbb Z\oplus \mathbb Z$ & $\begin{matrix}\longrightarrow\\[-3.5mm]\longrightarrow\\[-3.5mm]\longrightarrow\\[-3.5mm]\longrightarrow\\[-3.5mm]\longleftarrow\\[-3.5mm]
\longleftarrow\\[-3.5mm]\longleftarrow
\end{matrix}$ & $\mathbb Z\oplus \mathbb Z$ & $\begin{matrix}\longrightarrow\\[-3.5mm] \longrightarrow\\[-3.5mm]\longrightarrow\\[-3.5mm]
\longleftarrow\\[-3.5mm]\longleftarrow \end{matrix}$ & $\mathbb Z$
\end{tabular}
\end{center}
with free generators $y_i=s_n\dots \hat s_i\dots s_0(i_1),\
i=0,\dots, n+1$ in degree $n+1$, for a generator $i_1\in K(\mathbb
Z,1)$. Consider the map (\ref{simpldiag}) between diagonals of
correspondent bisimplicial groups:
\begin{center}
\begin{tabular}{ccccccccccc}
$\dots$ & $\begin{matrix}\longrightarrow\\[-3.5mm]\ldots\\[-2.5mm]\longrightarrow\\[-3.5mm]\longleftarrow\\[-3.5mm]
\ldots\\[-2.5mm]\longleftarrow
\end{matrix}$ & $\Lambda^2(\mathbb Z^{\oplus 3})\otimes \mathbb Z[*,x_0,x_1,x_2]$ & $\begin{matrix}\longrightarrow\\[-3.5mm]\longrightarrow\\[-3.5mm]\longrightarrow\\[-3.5mm]\longrightarrow\\[-3.5mm]\longleftarrow\\[-3.5mm]
\longleftarrow\\[-3.5mm]\longleftarrow
\end{matrix}$ & $\Lambda^2(\mathbb Z\oplus \mathbb Z)\otimes \mathbb Z[*,x_0,x_1]$ & $\begin{matrix}\longrightarrow\\[-3.5mm] \longrightarrow\\[-3.5mm]\longrightarrow\\[-3.5mm]
\longleftarrow\\[-3.5mm]\longleftarrow \end{matrix}$ & $\Lambda^2(\mathbb Z)\otimes \mathbb Z[*,\sigma]$\\
 & & $\downarrow_{\ v_3}$ & & $\downarrow_{\ v_2}$ & & $\downarrow_{\ v_1}$\\
$\dots$ & $\begin{matrix}\longrightarrow\\[-3.5mm]\ldots\\[-2.5mm]\longrightarrow\\[-3.5mm]\longleftarrow\\[-3.5mm]
\ldots\\[-2.5mm]\longleftarrow
\end{matrix}$ & $\Lambda^2(\mathbb Z^{\oplus 3}\otimes \mathbb Z[*,x_0,x_1,x_2])$ & $\begin{matrix}\longrightarrow\\[-3.5mm]\longrightarrow\\[-3.5mm]\longrightarrow\\[-3.5mm]\longrightarrow\\[-3.5mm]\longleftarrow\\[-3.5mm]
\longleftarrow\\[-3.5mm]\longleftarrow
\end{matrix}$ & $\Lambda^2((\mathbb Z\oplus \mathbb Z)\otimes \mathbb Z[*,x_0,x_1])$ & $\begin{matrix}\longrightarrow\\[-3.5mm] \longrightarrow\\[-3.5mm]\longrightarrow\\[-3.5mm]
\longleftarrow\\[-3.5mm]\longleftarrow \end{matrix}$ & $\Lambda^2(\mathbb Z\otimes \mathbb Z[*,\sigma])$
\end{tabular}
\end{center}
The generator of $\pi_3(\Lambda^2K(\mathbb Z,1)\otimes \mathbb
Z[S^1])=L_2\Lambda^2(\mathbb Z,1)$ is given by the element
$$
\alpha=[s_1s_0(i_1),s_2s_0(i_1)]\otimes
s_2s_1(\sigma)-[s_1s_0(i_1),s_2s_1(i_1)]\otimes
s_2s_0(\sigma)+[s_2s_0(i_1),s_2s_1(i_1)]\otimes s_1s_0(\sigma)
$$
Then \begin{multline*} v_3(\alpha)=[s_1s_0(i_1)\otimes
s_2s_1(\sigma),s_2s_0(i_1)\otimes s_2s_1(\sigma)]-\\
[s_1s_0(i_1)\otimes s_2s_0(\sigma),s_2s_1(i_1)\otimes
s_2s_0(\sigma)]+ [s_2s_0(i_1)\otimes
s_1s_0(\sigma),s_2s_1(i_1)\otimes s_1s_0(\sigma)]
\end{multline*}
\begin{theorem}\label{kan1}{\bf (Kan)} Let $n\geq 1$. If $r$ is odd, then the suspension map
$$
L_*\EuScript L^r(\mathbb Z,2n)\to L_{*+1}\EuScript L^r(\mathbb
Z,2n+1)
$$
is an isomorphism. Let $\alpha_{2n+1}\in
\pi_{4n+2}\Lambda^2(\mathbb Z,2n+1)=\mathbb Z$ and $i_{2n+1}\in
K(\mathbb Z,2n+1)_{2n+1}$ be generators, then the suspension and
the composition homomorphism $\alpha_{2n+1}$ induce an isomorphism
$$
L_*\EuScript L^{2r}(\mathbb Z,2n)\oplus L_{*+1}\EuScript
L^r(\mathbb Z,4n+2)\simeq L_{*+1}\EuScript L^{2r}(\mathbb Z,2n+1).
$$
\end{theorem}

\subsection{}\label{derivedkan}

\begin{theorem}
$L_i\EuScript L^{2^k}(\mathbb Z,2n)$ is a $\mathbb Z/2$-vector
space indexed by all sequences $(i_1,\dots, i_k)\in \mathcal
W_{2n,k}$ with $i=2n+i_1+\dots+i_k$.
\end{theorem}

Recall the construction of basis elements of $L_i\EuScript
L^{2^k}(\mathbb Z,2n)$. Consider the simplicial group $$ \EuScript
L^{2^k}(K(\mathbb Z,2n)\otimes \mathbb Z/2)\simeq (\EuScript
L^{2^k}K(\mathbb Z,2n))\otimes\mathbb Z/2
$$
For $k,l,i,n\geq 1$, the composition  $$ \pi_i(\EuScript
L^k(\mathbb Z,q)\otimes \mathbb Z/2)\otimes \pi_q(\EuScript
L^l(\mathbb Z,n)\otimes \mathbb Z/2)\to \pi_i(\EuScript
L^{kl}(\mathbb Z,n)\otimes \mathbb Z/2)
$$
 is defined analogously with (\ref{cmaps}).

\begin{prop}
Let $w=(i_1,\dots, i_k)$ be the sequence with the following
properties: 1) $i_1\leq 2n$;\ 2) $2i_j\geq i_{j+1}$. For
$i=1,2,\dots $ let $\beta_i\in \pi_{2i}(\Lambda^2K(\mathbb
Z,i)\otimes\mathbb Z/2)=\mathbb Z/2$ be a non-zero element.
Consider the following element
$$
\beta_w=\sigma^{2n-i_1}\beta_{i_1}\sigma^{2i_1-i_2}\beta_{i_2}\dots
\sigma^{2i_{k-1}-i_k}\beta_{i_k}\in
\pi_{2n+i_1+\dots+i_k}(\EuScript L^{2^k}(\mathbb Z,2n)\otimes
\mathbb Z/2)
$$
The group $\pi_i(\EuScript L^{2^k}(\mathbb Z,2n)\otimes \mathbb
Z/2)$ is a $\mathbb Z/2$-vector space with a basis consisting of
elements $\beta_w$ and $i=2n+i_1+\dots+i_k$.
\end{prop}

It follows that the natural projection $K(\mathbb Z,2n)\to
K(\mathbb Z,2n)\otimes \mathbb Z/2$ induces a monomorphism
$$
L_i\EuScript L^{2^k}(\mathbb Z,2n)\to \pi_i(\EuScript
L^{2^k}(K(\mathbb Z,2n)\otimes \mathbb Z/2))
$$
The image of this monomorphism is generated by elements $\beta_w$
for which $i_k$ is odd, i.e. for which $w\in \mathcal
W_{2n,k}^{(2)}$.

\vspace{.5cm} \begin{example}\label{examp1} \end{example} $$
L_i\EuScript L^8(\mathbb Z,2)=\begin{cases} \mathbb Z/2\oplus
\mathbb Z/2,\ i=6,8,9\\ \mathbb Z/2,\ i=5,7,10,11,12,13,15\\ 0\
\text{otherwise}
\end{cases}
$$

\subsection{}
\begin{theorem}\label{lieofz}
The group $L_i\EuScript L^{p^k}(\mathbb Z,2n)$ is a $\mathbb
Z/p$-vector space indexed by all sequences $(\nu_{i_1},\dots
,\nu_{i_k})\in \mathcal W_{2n,k}^{(p)}$ with
$i=2n+(2p-2)(i_1+\dots+i_j)-|\text{number of}\ \lambda_i\
\text{in}\ (\nu_{i_1}\dots \nu_{i_j})|$.
\end{theorem}
The elements from the sequences considered in the theorem
correspond to the following elements in the derived functors:
\begin{align*} & \mu_i\in \pi_{2pi}(\EuScript
L^p(\mathbb Z,2i)\otimes \mathbb Z/p)\simeq \mathbb Z/p,\\
& \lambda_i\in \pi_{2pi-1}(\EuScript L^p(\mathbb Z,2i)\otimes
\mathbb Z/p)\simeq \mathbb Z/p
\end{align*}

\begin{theorem}\label{kan2}{\bf (Kan)} $L_*\EuScript L^r(\mathbb Z,2n)=0$ if $r\neq p^j$ for any prime $p$ and $j\geq 1$.
\end{theorem}

\subsection{Generating function} Since there is a general
description of derived functors of Lie functors for $\mathbb Z$,
one can define a generating function with dimensions of values
derived functors as coefficients. The following result is due to
M. Tangora \cite{Tangora}. Let $H_n$ be the number of sequences of
positive integers $(i_1,\dots, i_l)$ with $i_1+\dots+i_l=n$
satisfying $i_{k+1}\leq di_k$ and $i_1\leq m$. The function
$$
H(q)=1+H_1q+H_2q^2+\dots+H_nq^n+\dots
$$
is
$$
H(q)=\frac{a(q)}{1-b(q)}
$$
where
\begin{align*}
& a(q)=\sum_{k\geq 0}\left(
\frac{q^{e(k)}(1-q^{me(k)})}{1-q^{e(k)}}\prod_{j=0}^{k-1}\frac{-q^{e(j)}}{1-q^{e(k)}}\right),\\
& b(q)=\sum_{k \geq
0}\left(\frac{q^{e(k)}}{1-q^{e(k)}}\prod_{j=0}^{k-1}\frac{-q^{e(k)}}{1-q^{e(k)}}\right),\\
& e(k)=\frac{d^{k+1}-1}{d-1}
\end{align*}

\vspace{.5cm}
\section{Derived functors $L_*\EuScript L_s(\mathbb Z,m)$}
\vspace{.5cm}

Analog of theorems \ref{kan1} and \ref{kan2} is the following

\begin{theorem}
1) $L_*\EuScript L_s^r(\mathbb Z,2n+1)=0$ if $r\neq p^j$ for some
prime $p$ and $j\geq 1$;\\ \\
2) if $r$ is odd, the suspension $L_*\EuScript L_s^r(\mathbb
Z,2n+1)\to L_{*+1}\EuScript L_s^r(\mathbb Z, 2n+2)$ is an
isomorphism;\\ \\
3) if $r$ is odd, there is an isomorphism
$$
L_{*+1}\EuScript L_s^{2r}(\mathbb Z,2n)\simeq L_{*+1}\EuScript
L_s^r(\mathbb Z,4n)\oplus L_*\EuScript L_s^{2r}(\mathbb Z,2n-1)
$$
\end{theorem}

Analog of theorem \ref{lieofz} is the following:
\begin{theorem}\label{slieofz} For an odd prime $p$, $k\geq 1$, the group $L_i\EuScript
L_s^{p^k}(\mathbb Z,2n)$ is a $\mathbb Z/p$-vector space indexed
by all sequences $(\nu_{i_1},\dots ,\nu_{i_k})\in \widetilde
{\mathcal W}_{2n,k}^{(p)}$ with
$i=2n+(2p-2)(i_1+\dots+i_j)-|\text{number of}\ \lambda_i\
\text{in}\ (\nu_{i_1}\dots \nu_{i_j})|$.
\end{theorem}

Above results show that the structures of derived functors of Lie
and super-Lie functors have similar points but they are not the
same. The derived functors of $\EuScript L_s^2=\Gamma_2$ of
$\mathbb Z$ are well-known and follow from the decalage
$$
L\Gamma_2(\mathbb Z,n)[2]\simeq L\Lambda^2(\mathbb Z,n+1)
$$
In this connection the following result looks surprising:
\begin{theorem}\label{strange2}
For $k\geq 2$, $n\geq 1$, one has an isomorphism
$$
L_*\EuScript L^{2^k}(\mathbb Z,n)\simeq L_*\EuScript
L_s^{2^k}(\mathbb Z,n)
$$
\end{theorem}
Theorem \ref{strange2} uses the specific structure of Lie and
super-Lie functors of degrees powers of 2 and can not be
generalized to odd prime powers.
\begin{example}
\end{example}
$$
L_i\EuScript L^9(\mathbb Z,2)=\begin{cases} \mathbb Z/3,\ i=8,9,12,13,17\\
0\ \text{otherwise}
\end{cases}\ \ \ \ \
L_i\EuScript L_s^9(\mathbb Z,2)=\begin{cases} \mathbb Z/3,\
i=4,5,9\\ 0\ \text{otherwise}
\end{cases}
$$
Generators of the 3-torsion of $L_*\EuScript L_s^9(\mathbb Z,2)$
in degrees $4,5,9$ correspond to the elements
$(\lambda_1,\lambda_2),$ $(\mu_1,\lambda_2),$ $(\mu_1,\lambda_3)$
from $\widetilde{\mathcal W}_{2,2}^{(3)}$ respectively.
\vspace{.5cm}
\section{Derived functors $L_*\EuScript L(\mathbb Z/p^k, m)$}
\vspace{.5cm}
A differential graded Lie ring with squares (shortly DGLS) is a
GLRS $B=\bigoplus_{i=0}^\infty B_i$ together with homomorphisms
$\partial: B_i\to B_{i-1},\ i=1,2,\dots$ such that
\begin{align}
& \partial\circ \partial=0\\
& \partial\{x,y\}=\{\partial(x),y\}+(-1)^i\{x,\partial(y)\},\ x\in
B_i\\
& \partial(x^{[2]})=\{\partial(x),x\},\ \text{for}\ i\ \text{odd}
\end{align}

The main example of DGLS is the following. Let $X$ be a simplicial
Lie ring. Define $\partial: X_n\to X_{n-1}$ by
$\partial=\sum_{i=0}^n(-1)^i\partial_i,$ the Lie bracket $$[[\ ,\
]]: X_i\otimes X_j\to X_{i+j}$$ define by
$$
[[x,y]]=\sum_{(a;b)\in
(i,j)-\text{shuffles}}sign(a,b)[s_a(x),s_b(y)], x\in X_j,\ y\in
X_i
$$
and
$$
x^{[2]}=\sum_{(a;b)\in
(n,n)-\text{shuffles}}sign(a;b)[s_a(x),s_b(x)],\ x\in X_n,\ n\
\text{is odd}
$$
Then $X$ is a DGLS.

For every free abelian chain complex $C=\{C_i,\ \partial:
C_{i+1}\to C_i\}$, there exists a DGLS $\mathbb L(C)$ and a
morphism of chain complexes \begin{equation}\label{universaldgls}
C\to \mathbb L(C)\end{equation} such that, for any DGLS $R$ and
chain map $f:C\to R$, there is a unique map of DGLS-s: $\bar f:
\mathbb L(C)\to R$ such that the diagram
$$
\xyma{C \ar@{->}[r] \ar@{->}[d]^f & \mathbb L(C)\ar@{-->}[ld]^{\bar f}\\
R}
$$
commutes (this is lemma 3.4 from \cite{Leibowitz}).

For $k\geq 0$, denote by $(n+1,n;k)$ the chain complex $(\mathbb
Z\buildrel{k}\over\to \mathbb Z)[n].$

\begin{theorem}\label{leib}{\bf (Leibowitz), \cite{Leibowitz}} Let $p$ be a prime, $k\geq 1,\ i\geq 1,\ r\geq 1$.
There is an isomorphism
$$
L_i\EuScript L^r(\mathbb Z/p^{k},n)\simeq\ _p\pi_i(L\EuScript
L^r(\mathbb Z[n]\oplus \mathbb Z[n+1]))\oplus\ _pH_i\mathbb
L^r(n+1,n;p^k)
$$
\end{theorem}

Recall the main steps of the proof of theorem \ref{leib} from
\cite{Leibowitz}. For a prime $p$ and $n\geq 1$, consider the
standard simplicial model of $K(\mathbb Z/p^k,n)$:
$$
A:=N^{-1}((\mathbb Z\buildrel{p^k}\over\to \mathbb Z)[n]).
$$
Define the filtration
$$
\EuScript L^r(A)\supset F^r\supset F^{r-1}\supset \dots \supset
F^1\supset F^0
$$
where $F^j$ is the simplicial subgroup of $\EuScript L^r(A)$
generated by all basic commutators with at most $j$ elements which
arise from the generator of degree $(n+1)$ of $A$. This filtration
defines a spectral sequence
\begin{equation}\label{lss}
E_{i,j}^1=\pi_{i+j}(F^i/ F^{i-1})\Rightarrow L_{i+j}\EuScript
L^r(\mathbb Z/p,n). \end{equation}
The following facts are proved in \cite{Leibowitz}:\\ \\
1) Consider the map of DGLS-s:
\begin{equation}\label{lmap}
\mathbb L^r(n+1,n; p^k)\to \EuScript L^r(A)
\end{equation}
induced by
\begin{align*}
& y_n\mapsto x_n\\
& y_{n+1}\mapsto x_{n+1}
\end{align*}
where $(y_n,y_{n+1})$ are generators of dimensions $n$ and $n+1$
of $(n+1,n;p^k)$ and $(x_n,x_{n+1})$ are generators of dimensions
$n$ and $n+1$ of $K(\mathbb Z/p^k,n)$.

 There is an isomorphism
$$
E_{i,rn}^2=H_{i+rn}\mathbb L^r(n+1,n; p^k);
$$
2) for $j>rn$, $E_{i,j}^1=0$;\\ \\
3) For $j\neq rn$, $E_{i,j}^2=E_{i,j}^1$;\\ \\ 4)
$E_{i,j}^\infty=\ _{(p)}E_{i,j}^2,$ where
$_{(p)}E_{i,j}^2$ is the $p$-primary component of $E_{i,j}^2$.\\ \\
4) $\sum_{j=0,\ j\neq i-rn}^r\ _{(p)}E_{j,i-j}^2=\
_p\pi_i(K(\mathbb Z,n)\oplus
K(\mathbb Z,n+1))$\\

Theorem \ref{leib} follows from these facts regarding the spectral
sequence (\ref{lss}).

Recall other results from \cite{Leibowitz}, which we will use
bellow.

\begin{prop}\label{dan2} (\cite{Leibowitz}, Prop. 4.5)
$$
H_k\mathbb L^r(n+1,n;p^f)=(\mathbb Z/p^{f+1})^{\oplus d_k}\oplus
(\mathbb Z/p^f)^{\oplus (M_k-d_k)}
$$
where \begin{align*} & M_k=\ _pRank\ B_k\mathbb L^r(n+1,n;1)\\
& d_k=\ _p Rank\ H_k\mathbb L^r(n+1,n;1)
\end{align*}
\end{prop}
The numbers from proposition \ref{dan2} can be computed by the
following formulas:
\begin{align}
&  Rank\ B_k\mathbb L^r(n+1,n;1)=Rank\ \mathbb
L_k^r(n+1,n;1)-Rank\ B_{k-1}\mathbb L^r(n+1,n;1)\\
& Rank\ \mathbb L_k^r(1,0;1)=M(k,r-k)\ \text{if}\ r\ \text{is\
odd\ or\ if}\ r\ \text{is even and}\ k\equiv 0\mod 4\\
& Rank\ \mathbb L_k^r(1,0;1)=M(k,r-k)+Rank\ \mathbb
L_{k/2}^{r/2}(1,0;1)\
\text{if}\ r\ \text{is even and}\ k\equiv 2\mod 4\\
& Rank\ \mathbb L_{k+r}^r(2,1;1)=M(k,r-k)\ \text{if}\ r\ \text{is
odd or if}\ r\ \text{is even and}\ k+r\neq 2\mod 4\\
& Rank\ \mathbb L_{k+r}^r(2,1;1)=M(k,r-k)+Rank\ \mathbb
L_{\frac{k+r}{2}}^{r/2}(2,1;1)\ \text{if}\ r\ \text{is even and}\
k+r\equiv 2\mod 4
\end{align}

\begin{prop}\label{dan3} (\cite{Leibowitz}, Prop. 4.7) For $n\geq
0$, $f\geq 0$, there are isomorphisms
\begin{align*}
& H_*\mathbb L^r(1,0;p^f)\to H_{*+rn}\mathbb L^r(n+1,n;p^f)\
\text{if}\ n\ \text{is even}\\
& H_*\mathbb L^r(2,1;p^f)\to H_{*+r(n-1)}\mathbb L^r(n+1,n;p^f)\
\text{if}\ n\ \text{is odd}
\end{align*}
\end{prop}

\vspace{.5cm}\noindent{\bf Example.} We collect the
low-dimensional derived functors for $A=\mathbb Z/k$ in the
following table:
\begin{table}[H]
\begin{align}\label{table1}
& \begin{tabular}{cccccccccccccccccc}
 $i$ & \vline & $L_i\EuScript L^2(A)$ & \vline & $L_i\EuScript L^3(A)$ & \vline & $L_i\EuScript L^4(A)$ & \vline & $L_i\EuScript L^5(A)$ & \vline
 & $L_i\EuScript L^6(A)$ & \vline & $L_i\EuScript L^7(A)$ & \vline\\
 \hline
$6$ & \vline & 0 & \vline & 0 & \vline & 0 & \vline & 0 & \vline & 0 & \vline & 0 & \vline\\
$5$ & \vline & 0 & \vline & $0$ & \vline & $0$ & \vline & 0 & \vline & $ \mathbb Z/(3,k)$& \vline & $\mathbb Z/k$ & \vline\\
$4$ & \vline & 0 & \vline & 0 & \vline & $0$ & \vline & 0 & \vline & $\mathbb Z/(3k,k^2)$ & \vline & $(\mathbb Z/k)^{\oplus 2}$ & \vline\\
$3$ & \vline & 0 & \vline & $0$ & \vline & $\mathbb Z/(2,k)$ & \vline & $\mathbb Z/k$ & \vline & $\mathbb Z/(2,k)\oplus \mathbb Z/k$ & \vline & $(\mathbb Z/k)^{\oplus 3}$ & \vline\\
$2$ & \vline & $0$ & \vline & $0$ & \vline & $\mathbb Z/(2k,k^2)$ & \vline & $\mathbb Z/k$ & \vline & $\mathbb Z/(2k,k^2)\oplus\mathbb Z/k^{\oplus 2}$ & \vline & $(\mathbb Z/k)^{\oplus 2}$ & \vline\\
$1$ & \vline & $\mathbb Z/k$ & \vline & $\mathbb Z/k$ & \vline &
$\mathbb Z/k$ & \vline & $\mathbb Z/k$ & \vline & $\mathbb Z/k$ &
\vline & $\mathbb Z/k$ & \vline\\
\end{tabular}
\end{align}
\caption{Values of derived functors $L_i\EuScript L^m(\mathbb
Z/k)$}
\end{table}

\noindent Observe that, for a prime $p$, one has
$$
L_i\EuScript L^p(\mathbb Z/k)=\mathbb Z/k^{\oplus
\sum_{j=1}^i(-1)^{i-j}\frac{1}{p}\binom{p}{j}}
$$

\subsection{Functorial generalization}
For $k\geq 2,$ recall that we denote by $J_k$ the set of basic
tensor products of weight $k$ in $A$ and $B$. Set
$$
J_k=J_k^e\sqcup J_k^o=\bar J_k^e\sqcup \bar J_k^o
$$
where
\begin{align*}
& J_k^e\ \text{is the set of basic products from}\ J_k\ \text{with
even number of
entrances of}\ A\\
& J_k^o\ \text{is the set of basic products from}\ J_k\ \text{with
odd number of
entrances of}\ A\\
& \bar J_k^e\ \text{is the set of basic products from}\ J_k\
\text{with even number of
entrances of}\ B\\
& \bar J_k^o\ \text{is the set of basic products from}\ J_k\
\text{with odd number
of entrances of}\ B\\
\end{align*}
The structure of the graded components of DGLS $\mathbb
L(B\buildrel{f}\over\to A)$ is the following:
$$
\mathbb L^m(B\buildrel{f}\over\to A)_l=\bigoplus_{d|m,\
d|l}\left(\bigoplus_{D\in \bar J_{(m-l)/d,l/d}^e} \EuScript
L^d(D)\oplus \bigoplus_{D\in \bar J_{(m-l)/d,l/d}^o} \EuScript
L_s^d(D)\right)
$$
The structure of the DGLS $\mathbb L(C[1])$ for the shifted
complex $C=(B\buildrel{f}\over\to A)$ is the following
$$
\mathbb L^m((B\buildrel{f}\over\to A)[1])_{l+m}=\bigoplus_{d|m,\
d|l}\left(\bigoplus_{D\in J_{l/d,(m-l)/d}^e} \EuScript
L^d(D)\oplus \bigoplus_{D\in J_{l/d,(m-l)/d}^o} \EuScript
L_s^d(D)\right)
$$
The following proposition follows immediately from the
construction of DGLS $\EuScript L(C)$:
\begin{prop}\label{ww1} For $n\geq 1,\ m\geq 2,$ one has natural isomorphisms of chain
complexes \begin{align*} & \mathbb L^m((B\buildrel{f}\over\to
A)[2n]))\simeq \mathbb L^m(B\buildrel{f}\over\to A)[2nm]\\
& \mathbb L^m((B\buildrel{f}\over\to A)[2n+1]))\simeq \mathbb
L^m((B\buildrel{f}\over\to A)[1])[2nm]
\end{align*}
\end{prop}
Observe that proposition \ref{dan3} is a simple consequence of
proposition \ref{ww1}.
\begin{example}
\end{example}
Consider the simplest example of DGLS $\EuScript L(C)$. Let
$C=\{B\buildrel{f}\over\longrightarrow A\},$ i.e. we set the
abelian group $A$ in degree 0 and the abelian group $B$ in degree
1. The DGLS $\EuScript L(C)$ is a graded object, with the
following terms in low dimensions:\\
$$\EuScript L^2(C):\ \ \ \ \Gamma_2(B)\to A\otimes B\to \Lambda^2(A)$$
This is quadratic Koszul complex associated to $C$. The cubical term is\\
$$
\EuScript L^3(C):\ \ \ \ \EuScript L_s^3(B)\to B\otimes A\otimes B
\buildrel{\delta}\over\to B\otimes A\otimes A\to \EuScript L^3(A)
$$
where
\begin{equation}\label{deltamap}
\delta: b_1\otimes a\otimes b_2\mapsto b_1\otimes a\otimes
f(b_2)+b_2\otimes f(b_1)\otimes a-b_1\otimes f(b_2)\otimes a,\ \
b_1,b_2\in B, a\in A.
\end{equation}
The fourth degree component of $\EuScript L(C)$ has the following
structure:
\begin{multline}
\EuScript L^4(C):\ \ \ \ \EuScript L_s^4(B)\to B\otimes A\otimes
B\otimes B\buildrel{\delta_2}\over\to B\otimes A\otimes A\otimes
B\oplus \Gamma_2(B\otimes A)\buildrel{\delta_1}\over\to\\ B\otimes
A\otimes A\otimes A\to \EuScript L^4(A)
\end{multline}
where
\begin{align*}
& \delta_2:\ b_1\otimes a\otimes b_2\otimes b_3\mapsto b_2\otimes a\otimes f(b_1)\otimes b_3-b_2\otimes f(b_1)\otimes a\otimes b_3-b_1\otimes a\otimes f(b_2)\otimes b_3+ \\
& \ \ \ \ \ \ \ \ \ \ \ \ \ \ \ \ \ \ \ \ \ b_1\otimes a\otimes f(b_3)\otimes b_2+(b_2\otimes f(b_3))(b_1\otimes a)\\
& \delta_1:\ b_1\otimes a_1\otimes a_2\otimes b_2\mapsto b_2\otimes a_2\otimes f(b_1)\otimes a_1-b_2\otimes a_2\otimes a_1\otimes f(b_1)-b_2\otimes f(b_1)\otimes a_1\otimes a_2+\\
& \ \ \ \ \ \ \ \ \ \ \ \ \ \ \ \ \ \ \ \ \ +b_2\otimes a_1\otimes
f(b_1)\otimes a_2-b_1\otimes a_1\otimes a_2\otimes f(b_2)\\ &
\delta_1: \gamma_2(b\otimes a)\mapsto b\otimes a\otimes
f(b)\otimes a-b\otimes f(b)\otimes a\otimes a
\end{align*}
with $b,b_1,b_2,b_3\in B,\ a,a_1,a_2\in A$. For example, one has
\begin{align*}
& \EuScript L^2(\mathbb Z\buildrel{k}\over\to \mathbb Z)=(\mathbb
Z\buildrel{k}\over\to \mathbb Z)[1]\\
& \EuScript L^3(\mathbb Z\buildrel{k}\over\to \mathbb Z)=(\mathbb
Z\buildrel{k}\over\to \mathbb Z)[1]\\
& \EuScript L^4(\mathbb Z\buildrel{k}\over\to \mathbb Z)=(\mathbb
Z\buildrel{(0,2k)}\over\to \mathbb
Z\oplus \mathbb Z\buildrel{(k,0)}\over\to \mathbb Z)[1]\\
\end{align*}

\begin{example}
\end{example}
Now consider the low-dimensional terms of DGLS $\EuScript L(C[1])$
where $C=\{B\buildrel{f}\over\longrightarrow A\}.$ The quadratic
term is the following:
$$
\EuScript L^2(C[1]):\ \ \ \ \ (\Lambda^2(B)\to B\otimes A\to
\Gamma_2(A))[2]
$$
This is the shifted quadratic dual de Rham complex associated to
$C$. The cubical term is the following:
$$
\EuScript L^3(C[1]):\ \ \ \ \ (\EuScript L^3(B)\to B\otimes
A\otimes B\buildrel{\delta'}\over\to B\otimes A\otimes A\to
\EuScript L_s^3(A))[3]
$$
where
$$
\delta': b_1\otimes a\otimes b_2\mapsto b_2\otimes a\otimes
f(b_1)+b_2\otimes f(b_1)\otimes a+b_1\otimes a\otimes f(b_2),\
a\in A,\ b_1,b_2\in B
$$
(compare with map $\delta$ in (\ref{deltamap})). For example, one
has
\begin{align*}
& \mathbb L^2((\mathbb Z\buildrel{k}\over\to \mathbb
Z)[1])=(\mathbb Z\buildrel{2k}\over\to \mathbb Z)[2]\\
& \mathbb L^3((\mathbb Z\buildrel{k}\over\to \mathbb
Z)[1])=(\mathbb Z\buildrel{3k}\over\to \mathbb Z)[3]
\end{align*}

Observe that, given a two-step complex of free abelian groups
$C=(B\buildrel{f}\over\to A)$ and $n\geq 0,$ we have a natural (in
the category of flat resolutions of abelian groups) map of DGLS-s
$$
\mathbb L((C)[n])\to \EuScript L(N^{-1}(C[n]))
$$
which generalizes the map (\ref{lmap}). For a fixed $m\geq 2$, the
map
$$
\EuScript L^m(C)\to \EuScript L^m(N^{-1}(C))
$$
goes through the cross-effect sequence in the following way:
\begin{equation}\label{wdi}{\small \xyma{\EuScript L^m(C)_{l+1} \ar@{->}[r] \ar@{>->}[d]
& \EuScript L^m(C)_{l} \ar@{>->}[d]
\ar@{->}[r] &\dots \ar@{->}[r] & \EuScript L^m(C)_0 \ar@{=}[d]\\
\EuScript L^m(\underbrace{B|\dots|B}_{l\ \text{terms}})\oplus
\EuScript L^m(\underbrace{A|B|\dots|B}_{l+1\ \text{terms}})
\ar@{->}[r] \ar@{>->}[d] & \EuScript
L^m(\underbrace{B|\dots|B}_{l-1\ \text{terms}})\oplus \EuScript
L^m(\underbrace{A|B|\dots|B}_{l\
\text{terms}}) \ar@{->}[r] \ar@{>->}[d] & \dots \ar@{->}[r] & \EuScript L^m(A)\ar@{=}[d]\\
\EuScript L^m(N^{-1}(C))_{l+1}\ar@{->}[r] & \EuScript
L^m(N^{-1}(C))_{l} \ar@{->}[r] & \dots \ar@{->}[r] & \EuScript
L^m(N^{-1}(C))_0}}
\end{equation}
For $m=2$ this diagram has the following form:
\begin{equation}\label{gamma2d}
\xyma{\Gamma_2(B) \ar@{->}[r] \ar@{>->}[d] & B\otimes
A\ar@{>->}[d] \ar@{->}[r] & \Lambda^2(A)\ar@{=}[d]\\ B\otimes
B\ar@{->}[r] \ar@{>->}[d] &
\Lambda^2(B)\oplus B\otimes A \ar@{->}[r] \ar@{>->}[d] & \Lambda^2(A)\ar@{=}[d]\\
\Lambda^2(A\oplus B\oplus B) \ar@{->}[r] & \Lambda^2(A\oplus
B)\ar@{->}[r] & \Lambda^2(A)}
\end{equation}
Computing the cokernels in the upper part of diagram
(\ref{gamma2d}) we obtain the following diagram with exact rows
and columns:
\begin{equation}
\xyma{\Gamma_2(B) \ar@{->}[r] \ar@{>->}[d] & B\otimes
A\ar@{>->}[d] \ar@{->}[r] & \Lambda^2(A)\ar@{=}[d]\\ B\otimes
B\ar@{->}[r] \ar@{->>}[d] &
\Lambda^2(B)\oplus B\otimes A \ar@{->}[r] \ar@{->>}[d] & \Lambda^2(A)\\
\Lambda^2(B) \ar@{=}[r] & \Lambda^2(B)}
\end{equation}
which shows that the natural map
$$
\Lambda^2(B\to A)\to \Lambda^2(N^{-1}(B\to A))
$$
is a homotopy equivalence. The following proposition shows that
this happens for all prime degrees.

\begin{prop}\label{primepro}
For a prime $p$, the map
$$
\EuScript L^p(B\buildrel{f}\over\longrightarrow A)\to \EuScript
L^p(N^{-1}(B\buildrel{f}\over\longrightarrow A))
$$
is a homotopy equivalence.
\end{prop}

\begin{cor}
For a prime $p$, abelian group $A$ and a flat resolution $0\to
P\to Q\to A\to 0$, the truncated complex $\EuScript L^p(P\to Q)$
$$
Q\otimes P^{\oplus p-1}\to \dots \to
\bigoplus_{J_{l,p-l}}P^{\otimes p-l}\otimes Q^{\otimes l}\to
\dots\to P\otimes Q^{\otimes p-1}\to \EuScript L^p(Q)
$$
represents the object $L\EuScript L^p(A)$ in the derived category.
\end{cor}

Clearly, proposition \ref{primepro} works only for prime degrees.
For example, for the fourth degree, the natural map
$$
\mathbb L^4(B\buildrel{f}\over\longrightarrow A)\to \EuScript
L^4(N^{-1}(B\buildrel{f}\over\longrightarrow A))
$$
is not a homotopy equivalence, in general.

Now consider the case of a shifted complex $C[1]$ with
$C=(B\buildrel{f}\over\to A)$. The super-analog of the diagram
(\ref{wdi}) is the following diagram
$${\small
\xyma{\EuScript L^m(C[1])_{l+m+1} \ar@{->}[r] \ar@{>->}[d] &
\EuScript L^m(C[1])_{l+m} \ar@{>->}[d]
\ar@{->}[r] &\dots \ar@{->}[r] & \EuScript L^m(C[1])_m \ar@{=}[d]\\
\EuScript L_s^m(\underbrace{B|\dots|B}_{l\ \text{terms}})\oplus
\EuScript L_s^m(\underbrace{A|B|\dots|B}_{l+1\ \text{terms}})
\ar@{->}[r] \ar@{>->}[d] & \EuScript
L_s^m(\underbrace{B|\dots|B}_{l-1\ \text{terms}})\oplus \EuScript
L_s^m(\underbrace{A|B|\dots|B}_{l\
\text{terms}}) \ar@{->}[r] \ar@{>->}[d] & \dots \ar@{->}[r] & \EuScript L_s^m(A)\ar@{=}[d]\\
\EuScript L_s^m(N^{-1}(C))_{l+1}\ar@{->}[r] & \EuScript
L_s^m(N^{-1}(C))_{l} \ar@{->}[r] & \dots \ar@{->}[r] & \EuScript
L_s^m(N^{-1}(C))_0}}
$$
For $m=2$ this diagram has the following form:
\begin{equation}\label{spr2}
\xyma{\Lambda^2(B) \ar@{->}[r] \ar@{>->}[d] & B\otimes
A\ar@{>->}[d] \ar@{->}[r] & \Gamma_2(A)\ar@{=}[d]\\ B\otimes
B\ar@{->}[r] \ar@{>->}[d] &
\Gamma_2(B)\oplus B\otimes A \ar@{->}[r] \ar@{>->}[d] & \Gamma_2(A)\ar@{=}[d]\\
\Gamma_2(A\oplus B\oplus B) \ar@{->}[r] & \Gamma_2(A\oplus
B)\ar@{->}[r] & \Gamma_2(A)}
\end{equation}
There is no an analog of proposition \ref{primepro} in the shifted
case. Observe that the upper map of complexes in (\ref{spr2}) is
not a homotopy equivalence.

\subsection{Spectral sequence} Now consider the functorial generalization of the spectral
sequence (\ref{lss}). Let $C=\{B\buildrel{f}\over\to A\}$ be a
complex of length two with free abelian groups $A$ and $B$. Let
$m\geq 2$, $n\geq 0$. Consider the standard model for $L\EuScript
L^m(C[n])$:
$$\EuScript L^m(N^{-1}(C[n])):\ \ \dots
\begin{matrix} \EuScript L^m(B^{\oplus n+2}\oplus A^{\oplus \binom{n+2}{2}})
\end{matrix}
\begin{matrix}\longrightarrow\\[-3.5mm]\dots\\[-2.5mm]\longrightarrow\\[-3.5mm]\longleftarrow\\[-3.5mm]\dots\\[-2.5mm]\longleftarrow
\end{matrix}\
\EuScript L^m(B\oplus A^{\oplus n+1})
\begin{matrix}\longrightarrow\\[-3.5mm]\dots\\[-2.5mm]\longrightarrow\\[-3.5mm]\longleftarrow\\[-3.5mm]\dots\\[-2.5mm]\longleftarrow
\end{matrix}\
 \EuScript L^m(A)
$$
Define the natural simplicial filtration
\begin{equation}\label{sfiltration}
\EuScript L^m(N^{-1}(C[n]))=I_{m}\supset I_{m-1}\supset \dots
\supset I_0\supset I_{-1}=\{1\}
\end{equation}
where $I_j$ is the simplicial subgroup of $\EuScript
L^m(N^{-1}(C[n]))$ generated at each dimension by basic
commutators with at most $j$ elements which arise from $B$. For
example, for $n=0$,
$$
I_0:\ \ \ \dots\ \begin{matrix} \EuScript L^m(A)
\end{matrix}
\begin{matrix}\longrightarrow\\[-3.5mm]\longrightarrow\\[-3.5mm]\longrightarrow\\[-3.5mm]\longleftarrow\\[-3.5mm]
\longleftarrow
\end{matrix}\
\EuScript L^m(A)
\begin{matrix}\longrightarrow\\[-3mm]\longrightarrow\\[-3mm]\longleftarrow\end{matrix}\
 \EuScript L^m(A)
$$
and
$$
I_1:\ \ \ \dots\ \begin{matrix} (B\oplus B)\otimes A^{\otimes
m-1}\oplus \EuScript L^m(A)
\end{matrix}
\begin{matrix}\longrightarrow\\[-3.5mm]\longrightarrow\\[-3.5mm]\longrightarrow\\[-3.5mm]\longleftarrow\\[-3.5mm]
\longleftarrow
\end{matrix}\
B\otimes A^{\otimes m-1}\oplus \EuScript L^m(A)
\begin{matrix}\longrightarrow\\[-3mm]\longrightarrow\\[-3mm]\longleftarrow\end{matrix}\
 \EuScript L^m(A)
$$
This filtration
gives rise to the spectral sequence
\begin{equation}\label{fl}
E_{i,j}^1(C[n])=\pi_{i+j}(I_{i}/I_{i-1})\Rightarrow
\pi_{i+j}L\EuScript L^m(C[n])
\end{equation}
with differentials
$$
d_{i,j}^k: E_{i,j}^k\to E_{i-k, j+k-1}^k
$$
Observe that, the general $E^1$-term can be described as follows:
\begin{equation}\label{gene}
E_{mn+m-l,j-mn-m+l}^1(C[n])=\bigoplus_{d|m,\ d|l}\bigoplus_{D\in
J_{\frac{l}{d},\frac{m-l}{d}}}L_j\EuScript L^d(D,\frac{mn+m-n}{d})
\end{equation} In particular,
$$
E_{mn+i,0}^1(C[n])=\EuScript L^m(C[n])_{i},\ i\geq 0
$$
As an example, consider initial terms of the spectral sequence for
the fourth degree Lie functor. The spectral sequence for
$L\EuScript L^4(C)$ has the following form:
$$
\xyma{\EuScript L^4(A) & A\otimes B^{\otimes 3}
\ar@{->}[l]_{d_{1,0}^1} & A\otimes B\otimes B\otimes A\oplus
\Gamma_2(A\otimes B) \ar@{->}[l]_{d_{2,0}^1\ \ \ \ \ \ \ \ \ \ \ }
& B\otimes A^{\otimes
3} \ar@{->}[l]_{\ \ \ \ \ \ \ \ \ \ \ \ \ \ d_{3,0}^1} & \EuScript L_s^4(B)\ar@{->}[l]_{\ \ \ \ \ d_{4,0}^1}\\
& & & & \Gamma_2(B)\otimes \mathbb Z/2\ar@{-->}[ull]^{d_{4,-1}^2}}
$$
where the dash arrow is defined on homology. The $E^1$-term of the
spectral sequence for $L\EuScript L^4(C[1])$ has the following
form:

\begin{table}[H]
\begin{tabular}{ccccccccccccccc}
 $q$ & \vline & $E_{4,q}^1$ & $E_{5,q}^1$ & $E_{6,q}^1$ & $E_{7,q}^1$ & $E_{8,q}^1$\\
 \hline
 0 & \vline & $\EuScript L_s^4(A)$ & $A^{\otimes 3}\otimes B$ & $A\otimes B^{\otimes 2}\otimes A\oplus \Gamma_2(A\otimes B)$ & $A\otimes B^{\otimes 3}$ & $\EuScript L^4(B)$\\
 -1 & \vline & $\Gamma_2(A)\otimes \mathbb Z/2$ & $0$ & $0$ & $0$ & $\Gamma_2(B)\otimes \mathbb Z/2$\\
-2 & \vline & 0 & $0$ & $A\otimes B\otimes \mathbb Z/2$ & $0$ & $0$\\
-3 & \vline & $0$ & $0$ & $0$ & $0$ & $\Gamma_2(B)\otimes \mathbb Z/2$\\
-4 & \vline & 0 & $0$ & $0$ & 0 & $B\otimes \mathbb Z/2$\\
\end{tabular}
\vspace{.5cm} \caption{The $E^1$-term of the spectral sequence
  for $\EuScript L^4((B\to A)[1])$}
\end{table}

\subsection{D\'ecalage} Proposition \ref{primepro} implies that, for an exact sequence of free abelian
groups $$ 0\to P\to Q\to A\to 0,
$$
there is a natural long exact sequence
\begin{multline}\label{psp}
0\to \EuScript L_s^p(P)\to Q\otimes P^{\oplus p-1}\to \dots \to
\bigoplus_{J_{l,p-l}}P^{\otimes p-l}\otimes Q^{\otimes l}\to\\
\dots\to P\otimes Q^{\otimes p-1}\to \EuScript L^p(Q)\to \EuScript
L^p(A)\to 0
\end{multline}
Given a complex $B$ of free abelian groups, we have a natural
short exact sequence
$$
0\to B\to C(B)\to B[1]\to 0
$$
where $C(B)$ is the cone of $B$, which is contractible. Applying
the sequence (\ref{psp}) to the cone sequence of an arbitrary
element of the derived category ${\sf DAb_{\leq 0}},$ we obtain
the following

\begin{theorem}
For a prime $p$ and $C\in {\sf DAb_{\leq 0}}$, there is a natural
isomorphism in the derived category
$$
L\EuScript L^p(C[1])\simeq L\EuScript L_s^p(C)[p].
$$
\end{theorem}

\vspace{.5cm}
\section{The super-analog of Leibowitz spectral sequence}
\vspace{.5cm}

Now consider the super-analog of the spectral sequence (\ref{fl}).
Let $C=\{B\buildrel{f}\over\to A\}$ be a complex of length two
with free abelian groups $A$ and $B$. Let $m\geq 2$, $n\geq 0$.
Consider the standard model for $L\EuScript L_s^m(C[n])$:
$$\EuScript L_s^m(N^{-1}(C[n])):\ \ \dots
\begin{matrix} \EuScript L_s^m(B^{\oplus n+2}\oplus A^{\oplus \binom{n+2}{2}})
\end{matrix}
\begin{matrix}\longrightarrow\\[-3.5mm]\dots\\[-2.5mm]\longrightarrow\\[-3.5mm]\longleftarrow\\[-3.5mm]\dots\\[-2.5mm]\longleftarrow
\end{matrix}\
\EuScript L_s^m(B\oplus A^{\oplus n+1})
\begin{matrix}\longrightarrow\\[-3.5mm]\dots\\[-2.5mm]\longrightarrow\\[-3.5mm]\longleftarrow\\[-3.5mm]\dots\\[-2.5mm]\longleftarrow
\end{matrix}\
 \EuScript L_s^m(A)
$$
Define the natural simplicial filtration
$$
\EuScript L_s^m(N^{-1}(C[n]))=\bar I_{m}\supset \bar
I_{m-1}\supset \dots \supset \bar I_0\supset \bar I_{-1}=\{1\}
$$
where $\bar I_j$ is the simplicial subgroup of $\EuScript
L_s^m(N^{-1}(C[n]))$ generated at each dimension by basic
commutators with at most $j$ elements which arise from $B$
together with elements of the type $x^{[2]}$, if $m\equiv 2 \mod
4$ and $x$ is a basic commutator of length $m/2$ with at most
$j/2$ elements which arise from $B$. For example, for $n=0$,
$$
\bar I_0:\ \ \ \dots\ \begin{matrix} \EuScript L_s^m(A)
\end{matrix}
\begin{matrix}\longrightarrow\\[-3.5mm]\longrightarrow\\[-3.5mm]\longrightarrow\\[-3.5mm]\longleftarrow\\[-3.5mm]
\longleftarrow
\end{matrix}\
\EuScript L_s^m(A)
\begin{matrix}\longrightarrow\\[-3mm]\longrightarrow\\[-3mm]\longleftarrow\end{matrix}\
 \EuScript L_s^m(A)
$$
and
$$
\bar I_1:\ \ \ \dots\ \begin{matrix} (B\oplus B)\otimes A^{\otimes
m-1}\oplus \EuScript L_s^m(A)
\end{matrix}
\begin{matrix}\longrightarrow\\[-3.5mm]\longrightarrow\\[-3.5mm]\longrightarrow\\[-3.5mm]\longleftarrow\\[-3.5mm]
\longleftarrow
\end{matrix}\
B\otimes A^{\otimes m-1}\oplus \EuScript L_s^m(A)
\begin{matrix}\longrightarrow\\[-3mm]\longrightarrow\\[-3mm]\longleftarrow\end{matrix}\
 \EuScript L_s^m(A)
$$
This filtration gives rise to the spectral sequence
$$
\bar E_{i,j}^1(C[n])=\pi_{i+j}(\bar I_{i}/\bar I_{i-1})\Rightarrow
\pi_{i+j}L\EuScript L_s^m(C[n])
$$
with differentials
$$
d_{i,j}^k: \bar E_{i,j}^k\to \bar E_{i-k, j+k-1}^k
$$
A natural analog of the formula (\ref{gene}) for the $E^1$-term is
the following:
\begin{equation}\label{gene1}
\bar E_{mn+m-l,j-mn-m+l}^1(C[n])=\bigoplus_{d|m,\
d|l}\bigoplus_{D\in J_{\frac{l}{d},\frac{m-l}{d}}}L_j\EuScript
L_s^d(D,\frac{mn+m-n}{d})
\end{equation}

\vspace{.5cm}\noindent{\bf Example.} Consider, for example, the
spectral sequence for the third super-Lie functor $\EuScript
L_s^3$:
$$
\xyma{\EuScript L_s^3(A) & A\otimes B\otimes A
\ar@{->}[l]_{d_{1,0}^1}
& A\otimes B\otimes B \ar@{->}[l]_{d_{2,0}^1} & \EuScript L^3(B)\ar@{->}[l]_{d_{3,0}^1}\\
& & & B\otimes \mathbb Z/3\ar@{-->}[ull]^{d_{3,-1}^2}}
$$
One can compare the spectral sequence for $\EuScript L^3$ and
$\EuScript L_s^3$ applied for the complex $\mathbb
Z\buildrel{3}\over\to\mathbb Z$:
\begin{table}[H]
\begin{tabular}{ccccccccccccccccccc}
$j=0$ & \vline & 0 & & $\mathbb Z$ &
$\buildrel{3}\over\leftarrow$& $\mathbb Z$ & & 0\\ \hline & \vline
\\ $j=-1$ & \vline & 0 & & 0 & & 0 & & 0
\end{tabular}\ \ \ \ \ \ \ \ \ \
\begin{tabular}{ccccccccccccccccccc}
$j=0$ & \vline & 0 & & $\mathbb Z$ &
$\buildrel{9}\over\leftarrow$& $\mathbb Z$ & & 0\\ \hline & \vline
\\ $j=-1$ & \vline & 0 & & 0 & & 0 & & $\mathbb Z/3$
\end{tabular}
\vspace{.5cm} \caption{The $E^1$-term of spectral sequences for
$L\EuScript L^3(\mathbb Z/3)$ and $L\EuScript L_s^3(\mathbb Z/3)$}
\end{table}

In particular, one has the following values of the derived
functors:
$$
L_i\EuScript L^3(\mathbb Z/3)=\begin{cases} \mathbb Z/3,\ i=1\\
0,\ i\neq 1\end{cases}\ \ \
L_i\EuScript L_s^3(\mathbb Z/3)=\begin{cases} \mathbb Z/9,\ i=1\\ \mathbb Z/3,\ i=2\\
0,\ i\neq 1,2\end{cases}
$$
Here is an example of the initial terms of the spectral sequence
for the case $\EuScript L_s^3(C[1])$:
\begin{table}[H]
\begin{tabular}{ccccccccccccccc}
 $q$ & \vline & $E_{3,q}^1$ & $E_{4,q}^1$ & $E_{5,q}^1$ & $E_{6,q}^1$\\
 \hline
 0 & \vline & $\EuScript L^3(A)$ & $A\otimes A\otimes B$ & $A\otimes B\otimes B$ & $\EuScript L_s^3(B)$\\
 -1 & \vline & $A\otimes \mathbb Z/3$ & $0$ & $0$ & $0$\\
-2 & \vline & 0 & $0$ & $0$ & $0$\\
-3 & \vline & $0$ & $0$ & $0$ & $B\otimes \mathbb Z/3$\\
\end{tabular}
\vspace{.5cm} \caption{The $E^1$-term of the spectral sequence
  for $\EuScript L_s^3((B\to A)[1])$}
\end{table}

Now we can formulate the super-analog of theorem \ref{leib}.

\begin{theorem}\label{superleib}
Let $A= \mathbb Z/p^k$, then
$$
L_i\EuScript L_s^m(A,n)\simeq\ _p\pi_i(\EuScript L_s^m(\mathbb
Z[n]\oplus \mathbb Z[n+1]))\oplus\ _pH_{i+m}\mathbb
L^m(n+2,n+1;p^k)
$$
\end{theorem}
In particular, for $A=\mathbb Z/p^k$, one has an isomorphism
$$
L_i\EuScript L_s^m(A)\simeq\ _p\pi_i(\EuScript L_s^m(\mathbb
Z\oplus \mathbb Z[1]))\oplus\ _pH_{i+m}\mathbb L^m(2,1;p^k)
$$

\vspace{.5cm}\noindent{\bf Example.} Analogously to table
(\ref{table1}) we will collect the low-dimensional derived
functors for $A=\mathbb Z/k$ in the following table:
\begin{table}[H]
\begin{align*}
& \begin{tabular}{cccccccccccccccccc}
 $i$ & \vline & $L_i\EuScript L_s^2(A)$ & \vline & $L_i\EuScript L_s^3(A)$ & \vline & $L_i\EuScript L_s^4(A)$ & \vline & $L_i\EuScript L_s^5(A)$ & \vline
 & $L_i\EuScript L_s^6(A)$ & \vline & $L_i\EuScript L_s^7(A)$ & \vline\\
 \hline
$6$ & \vline & 0 & \vline & 0 & \vline & 0 & \vline & 0 & \vline & 0 & \vline & $\mathbb Z/(7,k)$ & \vline\\
$5$ & \vline & 0 & \vline & $0$ & \vline & $0$ & \vline & 0 & \vline & $ \mathbb Z/(3,k)$& \vline & $\mathbb Z/(7k,k^2)$ & \vline\\
$4$ & \vline & 0 & \vline & 0 & \vline & $0$ & \vline & $\mathbb Z/(5,k)$ & \vline & $\mathbb Z/k$ & \vline & $(\mathbb Z/k)^{\oplus 2}$ & \vline\\
$3$ & \vline & 0 & \vline & $0$ & \vline & $\mathbb Z/(2,k)$ & \vline & $\mathbb Z/(5k,k^2)$ & \vline & $\mathbb Z/k^{\oplus 2}$ & \vline & $(\mathbb Z/k)^{\oplus 3}$ & \vline\\
$2$ & \vline & $0$ & \vline & $\mathbb Z/(3,k)$ & \vline & $\mathbb Z/(2k,k^2)$ & \vline & $\mathbb Z/k$ & \vline & $\mathbb Z/(2,k)\oplus\mathbb Z/k^{\oplus 2}$ & \vline & $(\mathbb Z/k)^{\oplus 2}$ & \vline\\
$1$ & \vline & $\mathbb Z/(2,k)$ & \vline & $\mathbb Z/(3k,k^2)$ &
\vline & $\mathbb Z/k$ & \vline & $\mathbb Z/k$ & \vline &
$\mathbb Z/(2,k)\oplus \mathbb Z/k$ & \vline & $\mathbb Z/k$ & \vline\\
$0$ & \vline & $\mathbb Z/(2k,k^2)$ & \vline & $0$ & \vline & $0$
& \vline & $0$ & \vline & $0$ & \vline & $0$ & \vline
\end{tabular}
\end{align*}
\caption{Values of derived functors $L_i\EuScript L_s^m(\mathbb
Z/k)$}
\end{table}

\vspace{.5cm}
\section{New functors and Main Conjecture}
\vspace{.5cm} \subsection{Hierarchies of special functors.} For a
prime $p$ and $n\geq 2,$ define the functors
\begin{align} & \EuScript N_s^{n;p}(A):=\ker\{\mathbb
L^n(A\buildrel{\sim}\over\to A)_{n-1}\otimes \mathbb Z/p\to
\mathbb L^n(A\buildrel{\sim}\over\to A)_{n-2}\otimes \mathbb
Z/p\} \label{specfu}\\
& \EuScript N^{n;p}(A):=\ker\{\mathbb
L^n((A\buildrel{\sim}\over\to A)[1])_{2n-1}\otimes \mathbb Z/p\to
\mathbb L^n((A\buildrel{\sim}\over\to A)[1])_{2n-2}\otimes \mathbb
Z/p\}\label{specfu1}
 \end{align}
It follows from definition that, for every abelian group $A$,
there are natural monomorphisms
\begin{align*}
& \EuScript N^{n;p}(A) \hookrightarrow \otimes^n(A)\otimes \mathbb Z/p\\
& \EuScript N_s^{n;p}(A) \hookrightarrow \otimes^n(A)\otimes
\mathbb Z/p
\end{align*}
For a free abelian group $A$, the spectral sequence (\ref{fl}) for
$C=\{A\buildrel{\sim}\over\to A\}$ implies that there is a natural
isomorphism
$$
L_{n-1}\EuScript L^n(A,2)\simeq H_{n-2}\mathbb L^n(C)
$$
The K\"unneth formula implies that there is the following natural
exact sequence:
\begin{equation}\label{mzq1}
0\to \EuScript L^{n}(A)\otimes \mathbb Z/p \to \EuScript
N^{n;p}(A)\to Tor(L_{n-1}\EuScript L^n(A,2), \mathbb Z/p)\to 0
\end{equation}
Its super-analog has the following form:
$$
0\to \EuScript L_s^n(A)\otimes \mathbb Z/p\to \EuScript
N_s^{n;p}(A)\to Tor(L_{n-1}\EuScript L^n(A,1),\mathbb Z/p)\to 0
$$
The sequence (\ref{mzq1}) splits as a sequence of abelian groups,
however, for every $n$, such that the $Tor$-term in (\ref{mzq1})
is not zero, the sequence (\ref{mzq1}) does not split as a
sequence of functors. This follows from the fact that the functor
$$
A\mapsto L_{n-1}\EuScript L^n(A,2)
$$
is of degree less than $n$, however, for every functor $F$ of
degree less than $n$, any natural transformation $F(A)\to
\otimes^n(A)\otimes \mathbb Z/p$ is zero. For any prime $p$ and
$k\geq 1$, there are chains of natural epimorphisms
$$
\EuScript N^{kp^i;p}\twoheadrightarrow \dots\EuScript
N^{kp;p}\twoheadrightarrow\EuScript N^{k;p}
$$
which split abstractly, but not naturally, moreover, every natural
transformation of the type $N^{k;p}\to N^{kp;p}$ is the zero map
(the same is true for the functors $N_s^{kp^i;p}$). This is the
reason why we call such chains of functors by {\it hierarchies.}
For example, for $n=p,p^2,p^3,p^4$, the sequences (\ref{mzq1}) can
be found in the following diagram of exact sequences which split
abstractly but not naturally:
$$
\xyma{& \EuScript L^{p^4}(A)\otimes \mathbb Z/p \ar@{>->}[d]\\ &
\EuScript N^{p^4;p}(A) \ar@{->>}[d] & \EuScript L^{p^2}(A)\otimes
\mathbb Z/p\ar@{>->}[d]\\ \EuScript L^{p^3}(A)\otimes \mathbb Z/p
\ar@{>->}[r] & \EuScript N^{p^3;p}(A) \ar@{->>}[r] & \EuScript
N^{p^2;p}(A)\ar@{->>}[d]\\ & \EuScript L^p(A)\otimes \mathbb Z/p
\ar@{>->}[r] & \EuScript N^{p;p}(A) \ar@{->>}[r] & A\otimes
\mathbb Z/p}
$$

The spectral sequence (\ref{fl}) implies that, for a free abelian
group $A$, there are natural isomorphisms
\begin{align}
& \EuScript N^{n,p}(A)\simeq \pi_{2n}\left(L\EuScript
L^n(A,2)\lotimes \mathbb Z/p\right)\label{nco1}\\
& \EuScript N_s^{n,p}(A)\simeq \pi_{n}\left(L\EuScript
L^n(A,1)\lotimes \mathbb Z/p\right)\label{nco2}
\end{align}
For $p=2$ the
functor $\EuScript N^{2;2}(A)$ is exactly $\Gamma_2(A)\otimes
\mathbb Z/2$, there is the following sequence
$$
0\to \Lambda^2(A)\otimes \mathbb Z/2\to \Gamma_2(A)\otimes \mathbb
Z/2\to A\otimes \mathbb Z/2\to 0
$$
The description of cross-effects of the functors $\EuScript
N^{n;p},\ \EuScript N_s^{n;p}$ follows directly from the
description of cross-effects of the Lie and super-Lie functors
(see (\ref{crefe}) and (\ref{screfe})). For free abelian $A,B$ one
has the following
$$
\EuScript N^{m;p}(A\oplus B)=\bigoplus_{d|m,\ 1\leq d\leq
m}\bigoplus_{C\in J_{m/d}}\EuScript N^{d;p}(C)$$ and
$$\EuScript N_s^{m;p}(A\oplus
B)=\bigoplus_{\substack{d|m,\ 1\leq d\leq m\\ m/d\
\text{odd}}}\bigoplus_{C\in J_{m/d}}\EuScript N_s^{d;p}(C)\oplus
\bigoplus_{\substack{d|m,\ 1\leq d< m\\ m/d\
\text{even}}}\bigoplus_{C\in J_{m/d}}\EuScript N^{d;p}(C)
$$

\subsection{Construction of $\EuScript E$-complexes}\label{ecom}
For $n,k\geq 1,$ a prime $p$, we will use the following notation:
$$
\overline{\mathcal W}_{n,k}^{(p)}:=\begin{cases}\mathcal
W_{n,k}^{(p)}(1)\setminus \mathcal W_{n,k}^{(p)}(2),\ \text{if}\
n\ \text{is even}\\ \mathcal W_{n-1,k}^{(p)},\ \text{if}\ n\
\text{is odd}\end{cases}
$$
For $C\in \sf DAb$, define the following objects of $\sf DAb$
$(n\geq 0)$
\begin{align*}
& \EuScript E^m(C,2n):=L\EuScript
L^{m}(C)[2mn]\oplus\bigoplus_{\substack{p\ \text{prime}\\
p^k|m\\ p^{k+1}\nmid m}}\bigoplus_{\substack{i=1,\dots, k\\
w\in \overline{\mathcal W}_{\frac{2nm}{p^{i}},i}^{(p)}}}
L\EuScript N^{\frac{m}{p^{i}};p}(C)[\frac{2nm}{p^{i}}+d(w)]\\
& \EuScript E^m(C,2n+1):=L\EuScript
L_s^{m}(C)[(2n+1)m]\oplus\\ & \ \ \ \ \ \ \ \ \ \ \ \ \ \ \bigoplus_{\substack{p\ \text{prime}\\
p^k|m\\ p^{k+1}\nmid m}}\bigoplus_{\substack{i=1,\dots,k\\ w\in
\overline{\mathcal W}_{\frac{(2n+1)m}{p^{i}},i}^{(p)}}} L\EuScript
N_s^{\frac{m}{p^{i}};p}(C)[\frac{(2n+1)m}{p^{i}}+d(w)]
\end{align*}
and their super-analogs $(n\geq 1)$
\begin{align*}
& \widetilde{\EuScript E}^m(C,2n-1):=L\EuScript
L^{m}(C)[(2n-1)m]\oplus\bigoplus_{\substack{p\ \text{prime}\\
p^k|m\\ p^{k+1}\nmid m}}\bigoplus_{\substack{i=1,\dots, k\\
w=(w_1,\dots,w_i)\in\overline{\mathcal W}_{\frac{2nm}{p^{i}},i}^{(p)}\\
w_i>u(p;k)}}
L\EuScript N^{\frac{m}{p^{i}};p}(C)[\frac{2nm}{p^{i}}+d(w)-m]\\
& \widetilde{\EuScript E}^m(C,2n):=L\EuScript
L_s^{m}(C)[2nm]\oplus\\ & \ \ \ \ \ \ \ \ \ \ \ \ \ \ \bigoplus_{\substack{p\ \text{prime}\\
p^k|m\\ p^{k+1}\nmid m}}\bigoplus_{\substack{i=1,\dots,k\\
w=(w_1,\dots,w_i)\in \overline{\mathcal
W}_{\frac{(2n+1)m}{p^{i}},i}^{(p)}\\ w_i>u(p;k)}} L\EuScript
N_s^{\frac{m}{p^{i}};p}(C)[\frac{(2n+1)m}{p^{i}}+d(w)-m]
\end{align*}
where $$u(p,k):=\begin{cases} \frac{p^{k-1}-1}{2},\ \text{if}\ p\
\text{is odd}\\ 2^{k-1},\ \text{if}\ p=2\end{cases}$$

\vspace{.5cm}\noindent{\bf Main Conjecture.} For $n,m,i\geq 1,$
$C\in \sf DAb$, there are natural isomorphisms
\begin{align}
& \pi_i(L\EuScript L^m(C[n]))\simeq \pi_i(\EuScript E^m(C,n))\\
& \pi_i(L\EuScript L_s^m(C[n]))\simeq \pi_i(\widetilde{\EuScript
E}^m(C,n)).
\end{align}

\vspace{.5cm} Here are the simplest examples of $\EuScript
E$-complexes in low
degrees:\\

\noindent{\bf Examples.} For $m=2$, $n\geq 1$, one has
\begin{align*}
& \EuScript E^2(C,2n)=L\Lambda^2(C)[4n]\oplus
\bigoplus_{i=1}^{n}C\lotimes \mathbb Z/2[2n+2i-1]\\
& \EuScript E^2(C,
2n+1)=L\Gamma_2(C)[4n+2]\oplus\bigoplus_{i=1}^nC\lotimes \mathbb
Z/2[2n+2i]\\
& \widetilde{\EuScript E}^2(C,2n)=L\Gamma_2(C)[4n]\oplus\bigoplus_{i=1}^{n}C\lotimes \mathbb Z/2[2n+2i-2]\\
& \widetilde{\EuScript E}^2(C,2n+1)=L\Lambda^2(C)[4n+2]\oplus
\bigoplus_{i=1}^{n+1}C\lotimes \mathbb Z/2[2n+2i-1]
\end{align*}
For $m=3$, $n\geq 1$, one has
\begin{align*}
& \EuScript E^3(C,2n)=L\EuScript L^3(C)[6n]\oplus
\bigoplus_{i=1}^{n}C\lotimes \mathbb Z/3[2n+4i-1]\\
& \EuScript E^3(C, 2n+1)=L\EuScript
L_s^3(C)[6n+3]\oplus\bigoplus_{i=1}^nC\lotimes \mathbb
Z/3[2n+4i]\\
& \widetilde{\EuScript E}^3(C,2n)=L\EuScript L_s^3(C)[6n]\oplus\bigoplus_{i=1}^{n}C\lotimes \mathbb Z/3[2n+4i-3]\\
& \widetilde{\EuScript E}^3(C,2n+1)=L\EuScript L^3(C)[6n+3]\oplus
\bigoplus_{i=1}^{n+1}C\lotimes \mathbb Z/3[2n+4i-2]
\end{align*}

\vspace{.5cm}
\section{The abstract isomorphism}
\vspace{.5cm}

\begin{prop}\label{abstr11}
For all $n,l\geq 1,$ there is a homotopy equivalence
$$
L\EuScript L^m(\mathbb Z,l)\sim \EuScript E^n(\mathbb Z,l).
$$
\end{prop}
\begin{proof}
The description of derived functors from \ref{derivedkan} and
(\ref{nco1}), (\ref{nco2}) imply that
\begin{equation}\label{lc1}
L\EuScript N^{m;p}(\mathbb Z)\simeq \begin{cases} \mathbb Z/p,\
\text{if}\
m=p^t,\ t\geq 0\\
0\ \text{otherwise}
\end{cases}
\end{equation}
and
\begin{equation}\label{lc2}
L\EuScript N_s^{m;p}(\mathbb Z)\simeq \begin{cases} \mathbb Z/p,\
\text{if}\
m=2p^t,\ t\geq 0\ \text{or}\ m=1\\
0\ \text{otherwise}
\end{cases}
\end{equation} The description (\ref{lc1}) implies that, if $m$ is not a power of prime, then
$\EuScript E^m(\mathbb Z,2n)\simeq 0,\ n\geq 1$. If $m$ is a power
of prime
$m=p^k$, then \begin{multline}\bigoplus_{\substack{i=1,\dots, k\\
w\in \overline{\mathcal W}_{\frac{2nm}{p^{i}},i}}}L\EuScript
N^{\frac{m}{p^{i}}}(\mathbb Z)[\frac{2nm}{p^{i}}+d(w)]\simeq
\bigoplus_{\substack{i=1,\dots,k\\
w\in \mathcal W_{\frac{2nm}{p^k},k}^{(p)}(i)\setminus \mathcal
W_{\frac{2nm}{p^k},k}^{(p)}(i+1)}} L\EuScript
N^{\frac{m}{p^{k-i+1}};p}(\mathbb Z)[2n+d(w)]\\ \simeq
\bigoplus_{w\in \mathcal W_{2n,k}^{(p)}}\mathbb Z/p\
[2n+d(w)]\end{multline} and we have the needed statement for all
even $l$. For $l=2n+1,\ n\geq 0$, we have
$$
\EuScript E^2(\mathbb Z,2n+1)\simeq \mathbb
Z[4n+2]\oplus\bigoplus_{i=1}^n\mathbb Z/2\ [2n+2i]\simeq
L\EuScript L^2(\mathbb Z,2n+1)
$$
The description (\ref{lc2}) implies that $\EuScript E^m(\mathbb Z,
2n+1)\simeq 0$ if $m$ is neither a power of prime nor a double
power of prime. Now assume that $m=p^k$ for an odd prime. In this
case, (\ref{lc2}) implies that
$$
\EuScript E^m(\mathbb Z,2n+1)\simeq \bigoplus_{w\in
\overline{\mathcal W}_{2n+1,k}^{(p)}}\mathbb Z/p[2n+1+d(w)]\simeq
\bigoplus_{w\in \mathcal W_{2n,k}^{(p)}}\mathbb
Z/p[2n+d(w)+1]\simeq L\EuScript L^m(\mathbb Z,2n+1)
$$
Now consider the case $m=2p^k$ for an odd $p$. In this case
\begin{multline*}
\EuScript E^{2p^k}(\mathbb Z,2n+1)\simeq \bigoplus_{\substack{i=1,\dots, k\\
w\in \overline{\mathcal W}_{2(2n+1)p^{k-i},i}}}\mathbb
Z/p[(2n+2)p^{k-i}+d(w)]\simeq\\ \bigoplus_{w\in \mathcal
W_{4n+2,k}^{(p)}}\mathbb Z/p[4n+2+d(w)]\simeq L\EuScript
L^{2p^k}(\mathbb Z,2n+1)
\end{multline*}
It remains to consider the case $m=2^k,\ k>1$. In this case,
\begin{multline*}
\EuScript E^{2^k}(\mathbb Z,2n+1)\simeq \bigoplus_{w\in
\overline{\mathcal W}_{2n+1,k}^{(2)}}\mathbb Z/2[2n+1+d(w)]\oplus
\bigoplus_ {\substack{i=1,\dots, k-1\\
w\in \overline{\mathcal W}_{2(2n+1)2^{k-i},i}}} \mathbb
Z/2[4n+2+d(w)]\simeq\\ L\EuScript L^{2^k}(\mathbb Z,2n)[1]\oplus
L\EuScript L^{2^{k-1}}(\mathbb Z,4n+2)\simeq L\EuScript
L^{2^k}(\mathbb Z,2n+1)
\end{multline*}
The statement is proved.
\end{proof}
\begin{prop}\label{tfree}
Let $C\in \sf DAb$ and homology of $C$ are torsion-free. For all
$n,m\geq 1$, there is a homotopy equivalence
$$
L\EuScript L^n(C[m])\sim \EuScript B^n(C,m)
$$
\end{prop}
\begin{proof}
Let $A,B\in \sf DAb$. Suppose that we are given the homotopy
equivalences
\begin{equation}\label{ca1}
L\EuScript L^d(C[2n])\sim \EuScript E^d(C,2n)
\end{equation}
for all $d<m$ and all $C\in \sf DAb$ with torsion-free homology
and
\begin{align} & L\EuScript L^m(A[2n])\sim \EuScript E^m(A,2n)\label{ca2}\\
& L\EuScript L^m(B[2n])\sim \EuScript
E^m(B,2n)\label{ca3}\end{align} It follows that
\begin{align} L\EuScript L^m(A\oplus B[2n])=& \bigoplus_{\substack{d|m,\ 1\leq d\leq
m\\ C\in J_{m/d}}}L\EuScript L^{d}(C[\frac{2nm}{d}])=\\
& L\EuScript L^m(A[2n])\oplus L\EuScript
L^m(B[2n])\oplus\bigoplus_{\substack{d|m,\ 1\leq d< m\\
C\in J_{m/d}}}L\EuScript L^{d}(C[\frac{2nm}{d}])\sim\\
& \EuScript E^m(A,2n)\oplus \EuScript
E^m(B,2n)\oplus\bigoplus_{\substack{d|m,\ 1\leq d< m\\
C\in J_{m/d}}}\EuScript E^{d}(C,\frac{2nm}{d})\sim\\
& \EuScript E^m(A,2n)\oplus \EuScript
E^m(B,2n)\oplus\bigoplus_{d|m,\ d<m,\ C\in J_{m/d}}L\EuScript L^d(C)[2nm]\oplus\\ & \bigoplus_{\substack{d|m,\ d< m,\ p^k|d\\
C\in J_{m/d},\ p^{k+1}\nmid d}}\bigoplus_{\substack{i=1,\dots,k\\
w\in \overline{\mathcal W}_{\frac{2nm}{p^{i}},i}^{(p)}}}
L\EuScript
N^{\frac{d}{p^{i}};p}(C)[\frac{2nm}{p^k}+d(w)]\label{compa1}
\end{align}
where each $C\in J_{m/d}$ is a basic tensor product
$A_{i_1}\lotimes \dots \lotimes A_{i_{m/d}}$, with $A_{i_j}\in
\{A,B\}.$ From the other hand,
\begin{align}
\EuScript E^m(A\oplus B,2n): & =\EuScript
E^{m}(A,2n)\oplus \EuScript E^m(B,2n)\oplus\bigoplus_{d|m,\ d<m,\ C\in J_{m/d}}L\EuScript L^d(C)[2nm]\oplus\\ & \bigoplus_{\substack{p\ \text{prime}\\
p^k|m\\ p^{k+1}\nmid m}}\bigoplus_{\substack{i=1,\dots,k\\ w\in
\overline{\mathcal W}_{\frac{2nm}{p^{i}},i}^{(p)}}}\bigoplus_{\substack{s|\frac{m}{p^i}\\
C\in J_{\frac{m}{sp^i}}}} L\EuScript
N^{s;p}(C)[\frac{2nm}{p^k}+d(w)]\label{compa2}
\end{align}
We will show that the terms (\ref{compa1}) and (\ref{compa2}) are
the same. Observe that (\ref{compa1}) and (\ref{compa2}) are
direct sums of the shifted terms like
$$
R_{s,t}(C):=L\EuScript N^{s;p}(C), C\in J_t
$$
for $s,t$ such that $s,t|m,\ m/st\ \text{is a power of p}.$ The
contributions of the term $R_{s,t}(C),\ C\in J_t$ in
(\ref{compa1}) and (\ref{compa2}) are the same, namely
$$
\bigoplus_{w\in\overline{\mathcal W}_{2nst,
log_p\frac{m}{st}}^{(p)}}R_{s,t}(C)[2nst+d(w)].
$$
Hence, the conditions (\ref{ca1}),(\ref{ca2}) and (\ref{ca3})
imply that
$$
L\EuScript L^m(A\oplus B[2n])\sim \EuScript E^m(A\oplus B,2m).
$$
Now consider the odd dimensions. Suppose that we have the
conditions
\begin{equation}\label{coa1}
L\EuScript L^d(C[2n+1])\sim \EuScript E^d(C,2n+1)
\end{equation}
for all $d<m$ and arbitrary complexes $C$ and
\begin{align} & L\EuScript L^m(A[2n+1])\sim \EuScript E^m(A,2n+1)\label{coa2}\\
& L\EuScript L^m(B[2n+1])\sim \EuScript
E^m(B,2n+1)\label{coa3}\end{align} We have
\begin{align} & L\EuScript L^m(A\oplus B[2n+1])=
L\EuScript L^m(A[2n+1])\oplus L\EuScript L^m(B[2n+1])\oplus\\
& \ \ \ \ \ \ \ \ \ \ \ \ \ \ \ \bigoplus_{\substack{d|m,\ 1\leq d<
m\\ C\in J_{m/d}}}L\EuScript L^{d}(C[\frac{(2n+1)m}{d}])\sim\\
& \EuScript E^m(A,2n+1)\oplus \EuScript
E^m(B,2n+1)\oplus\bigoplus_{\substack{d|m,\ 1\leq
d< m\\ C\in J_{m/d}}}\EuScript E^{d}(C,\frac{(2n+1)m}{d})\sim\\
& \EuScript E^m(A,2n+1)\oplus \EuScript
E^m(B,2n+1)\oplus\bigoplus_{\substack{d|m,\ 1\leq d< m\\ C\in
J_{m/d}}}\EuScript L^{d}(C)[(2n+1)m]\oplus\\ &
\bigoplus_{\substack{d|m,\ m/d\ \text{even},\ p^k|d\\
C\in J_{m/d},\ p^{k+1}\nmid d}}\bigoplus_{\substack{i=1,\dots,k\\
w\in \overline{\mathcal W}_{\frac{(2n+1)m}{p^i},i}^{(p)}}}
L\EuScript
N^{\frac{d}{p^i};p}(C)[\frac{(2n+1)m}{p^i}+d(w)]\oplus\\
& \bigoplus_{\substack{d|m,\ m/d\ \text{odd},\ p^k|d\\
C\in J_{m/d},\ p^{k+1}\nmid d}}\bigoplus_{\substack{i=1,\dots,k\\
w\in \overline{\mathcal W}_{\frac{(2n+1)m}{p^i},i}^{(p)}}}
L\EuScript
N_s^{\frac{d}{p^i};p}(C)[\frac{(2n+1)m}{p^i}+d(w)]\label{contr6}
\end{align}
where each $C\in J_{m/d}$ is a basic tensor product
$A_{i_1}\lotimes \dots \lotimes A_{i_{m/d}}$, with $A_{i_j}\in
\{A,B\}.$ From the other hand,
\begin{align}
\EuScript E^m(A\oplus B,2n+1): & =\EuScript
E^{m}(A,2n+1)\oplus \EuScript E^m(B,2n+1)\oplus\bigoplus_{d|m,\ d<m,\ C\in J_{m/d}}L\EuScript L^d(C)[(2n+1)m]\oplus\\
& \bigoplus_{\substack{p\ \text{prime}\\
p^k|m\\ p^{k+1}\nmid m}}\bigoplus_{\substack{i=1,\dots,k\\ w\in
\overline{\mathcal W}_{\frac{(2n+1)m}{p^i},i}^{(p)}}}(\bigoplus_{\substack{s|\frac{m}{p^i},\ \frac{m}{sp^i}\ \text{even}\\
C\in J_{\frac{m}{sp^i}}}} L\EuScript
N^{s;p}(C)[\frac{(2n+1)m}{p^i}+d(w)]\oplus\label{contr5}\\
& \bigoplus_{\substack{s|\frac{m}{p^i},\ \frac{m}{sp^i}\ \text{odd}\\
C\in J_{\frac{m}{sp^i}}}} L\EuScript
N_s^{s;p}(C)[\frac{(2n+1)m}{p^i}+d(w)])\label{contr4}
\end{align}
Consider the contribution of the terms like
$$
R_{s,t}(C):=L\EuScript N^{s;p}(C),\ \ \ R_{s,t}'(C):=L\EuScript
N_s^{s;2}(C),\  C\in J_t
$$
in (\ref{contr6}), (\ref{contr5}), (\ref{contr4}), for $s,t$ such
that $s,t|m,\ m/st\ \text{is a power of p}.$ For an even $t$, the
contributions of the term $R_{s,t}(C),\ C\in J_t$ in
(\ref{contr6}) and (\ref{contr5}) are the same, namely
$$
\bigoplus_{w\in\overline{\mathcal W}_{(2n+1)st,
log_p\frac{m}{st}}^{(p)}}R_{s,t}(C)[(2n+1)st+d(w)]
$$
For an odd $t$, the contributions of the term $R_{s,t}'(C),\ C\in
J_t$ in (\ref{contr6}) and (\ref{contr5}) are the same as well,
namely
$$
\bigoplus_{w\in\overline{\mathcal W}_{(2n+1)st,
log_p\frac{m}{st}}^{(p)}}R_{s,t}'(C)[(2n+1)st+d(w)]
$$
This comparison shows that there is a homotopy equivalence
$$
\EuScript L^m(A\oplus B[2n+1])\sim \EuScript E^m(A\oplus B, 2n+1).
$$
Since any element of $\sf DAb$ is homotopy equivalent to a direct
sum of elements like $\mathbb Z[l]$ for different $l$, the
statement follows from proposition \ref{abstr11}.
\end{proof}
\begin{prop}\label{wert1}
Let $m,n,l\geq 1,$ $p$ a prime. There is a homotopy equivalence
$$
\EuScript E^m(\mathbb Z/p^l,n)\sim L\EuScript L^m(\mathbb
Z/p^l,n).
$$
\end{prop}
\begin{proof}
Observe that, for all $s\geq 1,$ there are homotopy equivalences
$$
L\EuScript N^{s;p}(\mathbb Z/p^l)\sim L\EuScript N^{s;p}(\mathbb
Z\oplus \mathbb Z[1]).
$$
$$
L\EuScript N_s^{s;p}(\mathbb Z/p^l)\sim L\EuScript
N_s^{s;p}(\mathbb Z\oplus \mathbb Z[1]).
$$
We have
\begin{align*}
\EuScript E^m(\mathbb Z/p^l,2n)=& L\EuScript L^m(\mathbb Z/p^l)[2nm]\bigoplus_{\substack{p\ \text{prime}\\
p^k|m\\ p^{k+1}\nmid m}}\bigoplus_{\substack{i=1,\dots, k\\
\overline{\mathcal W}_{\frac{2nm}{p^{i}},i}^{(p)}}} L\EuScript
N^{\frac{m}{p^{i}};p}(\mathbb Z/p^l)[\frac{2nm}{p^{i}}+d(w)]\sim\\
& L\EuScript L^m(\mathbb Z/p^l)[2nm]\bigoplus_{\substack{p\ \text{prime}\\
p^k|m\\ p^{k+1}\nmid m}}\bigoplus_{\substack{i=1,\dots, k\\
\overline{\mathcal W}_{\frac{2nm}{p^{i}},i}^{(p)}}} L\EuScript
N^{\frac{m}{p^{i}};p}(\mathbb Z\oplus\mathbb
Z[1])[\frac{2nm}{p^{i}}+d(w)]
\end{align*}
We have
\begin{multline}
\pi_i(L\EuScript L^m(\mathbb Z/p^l)[2mn])=\ _p\pi_i(L\EuScript
L^m(\mathbb Z\oplus \mathbb Z[1])[2mn])\oplus\ _pH_{i-2mn}\mathbb L^m(1,0;p^l)=\\
_p\pi_i(L\EuScript L^m(\mathbb Z\oplus\mathbb Z[1])[2mn])\oplus\
_pH_{i}\mathbb L^m(2n+1,2n;p^l)
\end{multline}
by proposition \ref{dan3}. Since
$$
\EuScript E^m(\mathbb Z\oplus \mathbb Z[1],2n)\sim L\EuScript
L^m((\mathbb Z\oplus \mathbb Z[1])[2n])
$$
by proposition \ref{tfree}, we obtain the following
$$
\pi_i\EuScript E^m(\mathbb Z/p^l,2n)\simeq\ _p\pi_i(L\EuScript
L^m((\mathbb Z\oplus \mathbb Z[1])[2n]))\oplus\ _pH_i\mathbb
L^m(2n+1,2n;p^l)=L_i\EuScript L^m(\mathbb Z/p^l,2n)
$$
by theorem \ref{leib}.

We have \begin{align*} \EuScript E^m(\mathbb Z/p^l,2n+1)=&
L\EuScript
L_s^{m}(\mathbb Z/p^l)[(2n+1)m]\oplus\\ & \ \ \ \ \ \ \ \ \ \ \ \ \ \ \bigoplus_{\substack{p\ \text{prime}\\
p^k|m\\ p^{k+1}\nmid m}}\bigoplus_{\substack{i=1,\dots,k\\ w\in
\overline{\mathcal W}_{\frac{(2n+1)m}{p^{i}},i}^{(p)}}} L\EuScript
N_s^{\frac{m}{p^{i}};p}(\mathbb
Z/p^l)[\frac{(2n+1)m}{p^{i}}+d(w)]\\& L\EuScript
L_s^{m}(\mathbb Z/p^l)[(2n+1)m]\oplus\\ & \ \ \ \ \ \ \ \ \ \ \ \ \ \ \bigoplus_{\substack{p\ \text{prime}\\
p^k|m\\ p^{k+1}\nmid m}}\bigoplus_{\substack{i=1,\dots,k\\ w\in
\overline{\mathcal W}_{\frac{(2n+1)m}{p^{i}},i}^{(p)}}} L\EuScript
N_s^{\frac{m}{p^{i}};p}(\mathbb Z\oplus\mathbb Z[1])[\frac{(2n+1)m}{p^{i}}+d(w)]\\
\end{align*}
Theorem \ref{superleib} implies that
\begin{multline}
\pi_i(L\EuScript L_s^m(\mathbb Z/p^l)[(2n+1)m])\simeq\
_p\pi_i(L\EuScript
L_s^m(\mathbb Z\oplus\mathbb Z[1])[(2n+1)m])\oplus\ _pH_{i-2nm}\mathbb L^m(2,1;p^l)\simeq\\
_p\pi_i(L\EuScript L_s^m(\mathbb Z\oplus\mathbb
Z[1])[(2n+1)m])\oplus\ _pH_{i}\mathbb L^m(2n+2,2n+1;p^l)
\end{multline}
Since
$$
\EuScript E^m(\mathbb Z\oplus\mathbb Z[1],2n+1)\sim L\EuScript
L^m((\mathbb Z\oplus \mathbb Z[1])[2n+1])
$$
by proposition \ref{tfree}, we obtain the following
$$
\pi_i\EuScript E^m(\mathbb Z/p^l,2n+1)\simeq\ _p\pi_i(L\EuScript
L^m((\mathbb Z\oplus\mathbb Z[1])[2n+1]))\oplus\ _pH_i\mathbb
L^m(2n+2,2n+1;p^l)\simeq L_i\EuScript L^m(\mathbb Z/p^l,2n+1)
$$
and the needed statement follows.
\end{proof}
Proposition \ref{wert1} gives a possibility to follow the proof of
the proposition \ref{tfree},not only for elements of $\sf DAb$
with torsion-free homology but for all elements of $\sf DAb$ and
prove the following
\begin{theorem}\label{abstractiso}
For every $n,m\geq 1,\ C\in \sf DAb,$ there is an (unnatural)
homotopy equivalence
$$
\EuScript E^m(C,n)\sim L\EuScript L^m(C[n]).
$$
\end{theorem}

\vspace{.5cm}
\section{Semi-d\'ecalage} \vspace{.5cm}
\subsection{Semi-d\'ecalage}For any  pair of free abelian groups
$A$ and $B$, and $m\geq 2,$ we  define a pair of  morphisms (see
\cite{BreenMikhailov}, 7.6)
\begin{align}
& \chi_m: \EuScript L_s^m(A)\otimes \Lambda^m(B)\to \EuScript
L^m(A\otimes B)\label{defchin}\\
& \bar \chi_m: \EuScript L^m(A)\otimes \Lambda^m(B)\to \EuScript
L_s^m(A\otimes B)\label{defchin8}
\end{align}
for  $a_1,\dots, a_m\in A$ and  $b_1,\dots, b_m\in B$, by
\begin{align*}
\chi_m: \{a_1,\dots, a_m\}\otimes b_1\wedge \dots \wedge b_m
&\mapsto \sum_{\sigma\in \Sigma_m}\text{sign}(\sigma)[a_1\otimes
b_{\sigma_1},\dots, a_m\otimes b_{\sigma_m}]
\\
\bar\chi_m: [a_1,\dots, a_m]\otimes b_1\wedge \dots \wedge
b_m&\mapsto \sum_{\sigma\in
\Sigma_m}\text{sign}(\sigma)\{a_1\otimes b_{\sigma_1},\dots,
a_m\otimes b_{\sigma_m}\} \,.
\end{align*}
For $m=2k$ with $k$ odd, we set
$$
\chi_m: \{a_1,\dots, a_k\}^{[2]}\otimes b_1\wedge\dots \wedge
b_n\mapsto \sum_{\sigma\in A_m}[[a_1\otimes b_{\sigma_1},\dots,
a_k\otimes b_{\sigma_k}],[a_1\otimes b_{\sigma_{k+1}},\dots,
a_k\otimes b_{\sigma_{2k}}]],
$$
Taking $X\in \sf DAb$, $B=K(\mathbb Z,1)$ we obtain the natural
maps
\begin{align}
& L\EuScript L^m(X)\lotimes L\Lambda^m(\mathbb Z,1)=L\EuScript
L^m(X)[m]\to L\EuScript L_s^m(X\lotimes K(\mathbb Z,1))=L\EuScript
L_s^m(X[1]) \label{zm1}\\
& L\EuScript L_s^m(X)\lotimes L\Lambda^m(\mathbb Z,1)=L\EuScript
L_s^m(X)[m]\to L\EuScript L^m(X\lotimes K(\mathbb Z,1))=L\EuScript
L^m(X[1]) \label{zm2} \end{align} Iterating these constructions,
we obtain the natural maps \begin{align*} & L\EuScript
L^m(X)[2nm]\to L\EuScript L^m(X[2n])\\ & L\EuScript
L^m(X)[(2n+1)m]\to L\EuScript
L_s^m(X[2n+1])\\
& L\EuScript L_s^m(X)[2nm]\to L\EuScript L_s^m(X[2n])\\
& L\EuScript L_s^m(X)[(2n+1)m]\to L\EuScript L^m(X[2n+1])
\end{align*}

\subsection{Bousfield's pension maps} For a functor $T: \sf Ab\to
Ab$, such that $T(0)=0$, there is a natural map
$$
E: \mathbb Z[M]\otimes T(N)\to T(M\otimes N),\ M,N\in \sf Ab
$$
which induces a paring
$$
E_*: \tilde H_*K(\mathbb Z,n)\otimes \pi_*(LT(X))\to
\pi_*(LT(X[2])),\ n\geq 1,\ X\in \sf DAb
$$
These constructions are from \cite{Bou}. For $n=2$, taking the
generator $\epsilon_r\in \tilde H_{2r}K(\mathbb Z,2)=\mathbb Z$,
we obtain so-called pension maps
$$
\epsilon_r: \pi_i(LT(X))\to \pi_{i+2r}(LT(X[2])),\ i\geq 0.
$$
Observe that, the composition of the maps (\ref{zm1}) and
(\ref{zm2}) induces the composition map
$$
\pi_i(L\EuScript L^m(X))\to \pi_{i+m}(L\EuScript L_s^m(X[1]))\to
\pi_{i+2m}(L\EuScript L^m(X[2]))
$$
which is exactly the pension map $\epsilon_r$ for the functor
$\EuScript L^m$.

The following statement is proved in (\cite{Bou}, Theorem 3.1):
\begin{theorem}
Let $T: \sf Ab\to Ab$ a polynomial functor of degree $\leq r\
(r\geq 1)$ and let $X\in DAb,$ such that $H_i(X)=0,\ i>n$ for some
$n\geq 0$. Then
$$
\epsilon_r: \pi_i(LT(X))\to \pi_{i+2r}(LT(X[2]))
$$
is an isomorphism for $i>n(r-1)+1$ and a monomorphism for
$i=n(r-1)+1$.
\end{theorem}

\begin{cor}\label{corla}
Let $p$ be an odd prime. For $i\geq 1$, let
$\beta_i,\mu_i,\lambda_i$ be nontrivial elements
\begin{align*}
& \beta_i\in \pi_{2i}(\Lambda^2(K(\mathbb Z,i))\otimes \mathbb
Z/2)=\mathbb Z/2\\
& \mu_i\in \pi_{2pi}(\EuScript L^p(K(\mathbb Z,2i))\otimes \mathbb
Z/p)=\mathbb Z/p\\
& \lambda_i\in \pi_{2pi-1}(\EuScript L^p(K(\mathbb Z,2i))\otimes
\mathbb Z/p)=\mathbb Z/p.
\end{align*}
Then $\epsilon_2(\beta_i)\neq 0$, $\epsilon_p(\mu_i)\neq 0$,
$\epsilon_p(\lambda_i)\neq 0$.
\end{cor}

There is a map
$$
\beta_d: \EuScript L^d(B)\otimes SP^d(A)\to \EuScript L^d(B\otimes
A)
$$
given by
$$
\beta_d: [b_1,\dots, b_d]\otimes a_1\dots a_d\mapsto
\sum_{\sigma\in \Sigma_d}[b_1\otimes a_{\sigma_1},\dots,
b_d\otimes a_{\sigma_d}]
$$
For $d|n$, there is a natural map
$$
SP^n(A) \to \otimes^{n/d}(SP^d(A))
$$
Construct the map
$$
\beta_{d,n}: \EuScript L^d(B)\otimes SP^n(A)\to \EuScript
L^d(B\otimes (\otimes^{n/d}(A)))
$$
as a composition
$$
\EuScript L^d(B)\otimes SP^n(A)\to \EuScript L^d(B)\otimes
(\otimes^{n/d}(SP^d(A)))\to \EuScript L^d(B\otimes
(\otimes^{n/d}(A)))
$$
Here the last map is the $n/d$-th iteration of the map $\beta_d$.
The following two lemmas follow straightforwardly from definition
of pension maps and composition maps in derived functors of Lie
functors (\ref{cmaps}).

\begin{lemma} For $d|n$ and free abelian groups $A,X,Y$, $X_{i_j}\in\{X,Y\}$, the
following diagram
$$
\xyma{ \EuScript L^d(X_{i_1}\otimes \dots\otimes
X_{i_{n/d}})\otimes SP^n(A) \ar@{->}[r]^{\beta_{d,n}}
\ar@{->}[d]  & \EuScript L^d((X_{i_1}\otimes A)\otimes \dots \otimes (X_{i_{n/d}}\otimes A))\ar@{>->}[d] \\
\EuScript L^n(X\oplus Y)\otimes SP^n(A) \ar@{->}[r]^{\beta_n} &
\EuScript L^n((X\oplus Y)\otimes A)}
$$
is commutative.
\end{lemma}

\begin{cor}\label{creffl}
For $d|n$, $C_1,C_2\in \sf DAb$, $D_{i_j}\in \{C_1,C_2\},$ the
following diagram
$$
\xyma{ \pi_i\left(L\EuScript L^d(D_{i_1}\lotimes \dots \lotimes
D_{i_{n/d}})\right) \ar@{->}[r]^{\epsilon_d^{n/d}} \ar@{>->}[d] &
\pi_{i+2n}\left(L\EuScript
L^d(D_{i_1}\lotimes \dots\lotimes D_{i_{n/d}}[\frac{2n}{d}])\right)\ar@{>->}[d]\\
\pi_i(L\EuScript L^n(C_1\oplus C_2))\ar@{->}[r]^{\epsilon_n} &
\pi_{i+2n}(L\EuScript L^n((C_1\oplus C_2)[2])) }
$$
is commutative for all $i\geq 0$.
\end{cor}

\begin{lemma}\label{mloa} For $i,k,l,q\geq 0$ and $X\in \sf DAb$, the following diagram
$$
\xyma{L_i\EuScript L^k(\mathbb Z,q)\otimes \pi_q(L\EuScript
L^l(X))\ar@{->}[r] \ar@{->}[d]^{\epsilon_{k}^l\otimes \epsilon_l}
& \pi_i(L\EuScript L^{kl}(X)))\ar@{->}[d]^{\epsilon_{kl}}\\
L_{i+2kl}\EuScript L^k(\mathbb Z,q+2l)\otimes
\pi_{q+2l}(L\EuScript L^l(X[2]))\ar@{->}[r] &
\pi_{i+2kl}(L\EuScript L^{kl}(X[2]))}
$$
is commutative.
\end{lemma}
Now we are ready to prove the following
\begin{theorem}
For $n\geq 1$, $C\in \sf DAb$, the map
$$
\epsilon_n: \pi_i(L\EuScript L^n(C))\to \pi_{i+2n}(L\EuScript
L^n(C[2])))
$$
is injective for all $i\geq 0$.
\end{theorem}
\begin{proof}
First observe that the statement follows for $n=1$, in this case,
the considered map is an isomorphism. We will proceed by induction
on $n$. Assume that the statement follows for all Lie powers less
than $n$.

Theorems \ref{kan1}, \ref{kan2}, Lemma \ref{mloa} and Corollary
\ref{corla} together with the fact that pension maps commute with
suspensions (see 2.3 \cite{Bou}) imply that the map $\epsilon_n$
is injective for $C=\mathbb Z[l]$ for every $l\geq 0$. Corollary
\ref{creffl} imply that $\epsilon_n$ is injective for all $C$ with
torsion-free homology. The spectral sequence (\ref{lss}) and
Theorem \ref{leib} imply that there is the following diagram
$$
\xyma{_p\pi_i(L\EuScript L^n(\mathbb Z[m]\oplus \mathbb Z[m+1])
\ar@{>->}[r] \ar@{->}[d]^{\epsilon_n} & L_i\EuScript L^n(\mathbb
Z/p^l, m) \ar@{->}[d]^{\epsilon_n}
\ar@{->>}[r] & _pH_i\EuScript L^n(n+1,n;p^l)\ar@{->}[d]^\simeq \\
_p\pi_{i+2n}(L\EuScript L^n(\mathbb Z[m+2]\oplus \mathbb Z[m+3])
\ar@{>->}[r] & L_{i+2n}\EuScript L^n(\mathbb Z/p^l, m+2)
\ar@{->}[r] & _pH_{i+2n}\EuScript L^n(n+1,n;p^l) }
$$
for a prime $p$ and $l\geq 1$. The righthand map is an isomorphism
by Proposition \ref{dan3} (it can be shown straightforwardly that
the maps in Proposition \ref{dan3} are induced by the pension
maps). Now we see that the middle map $\epsilon_n$ is injective
and the needed inductive step follows from Corollary \ref{corla}.
\end{proof}
\section{Prime Lie powers}

\subsection{}
Recall the definition of the graded functor
$$\Gamma_*=\bigoplus_{n\geq 0}\Gamma_n: \sf Ab\to \sf Ab.$$ The
graded abelian group  $\Gamma_\ast(A)$ is generated by symbols
$\gamma_i(x)$ of degree $i\geq 0$ satisfying the following
relations for all $x,y \in A$: \begin{align*} & 1)\ \gamma_0(x) =
1 \\ & 2)\ \gamma_1(x)=x\\ & 3)\
\gamma_s(x)\gamma_t(x)=\binom{s+t}{s}\gamma_{s+t}(x) \\
& 4)\ \gamma_n(x+y)=\sum_{s+t=n}\gamma_s(x)\gamma_t(y),\ n\geq 1\\
& 5)\ \gamma_n(-x)=(-1)^n\gamma_n(x),\ n\geq 1.
\end{align*}

For $n\geq 2$, define the functor $$\widetilde\Gamma_n(A): \sf
Ab\to Ab
$$
by setting
$$
\widetilde\Gamma_n(A):=im\{A\otimes
\Gamma_{n-1}(A)\buildrel{l_n}\over\to \Gamma_n(A)\}
$$
where $l_n$ is the natural map. For example, one has $$\widetilde
\Gamma_2(A)=SP^2(A)$$ and a natural exact sequence
$$
0\to \EuScript L^3(A)\to \Gamma_2(A)\otimes A\to \widetilde
\Gamma_3(A)\to 0
$$
For a prime $p$ and $k\geq 1$, one has the natural exact sequence
$$
0\to \widetilde \Gamma_{p^k}(A)\to \Gamma_{p^k}(A)\to A\otimes
\mathbb Z/p\to 0
$$
This implies that, for a complex $C\in \sf D\sf Ab_{\leq 0}$, one
has a triangle
\begin{equation}\label{tri}
L\widetilde \Gamma_{p^k}(C)\to L\Gamma_{p^k}(C)\to C\lotimes
\mathbb Z/p\to L\widetilde \Gamma_{p^k}(C)[1] \end{equation}

Let $A$ be an abelian group. For $n\geq 1$, let $C_*^n(A)$ be the
complex of abelian groups defined by
$$
C_i^n(A)=\Lambda^i(A)\otimes \Gamma_{n-i}(A),\ 0\leq i\leq n,
$$
where the differentials $d_i: C_i^n(A)\to C_{i-1}^n(A)$ are:
$$
d_i(b_1\wedge \dots \wedge b_i\otimes X)=\sum_{k=1}^i(-1)^k
b_1\wedge \dots \wedge\hat b_k\wedge \dots \wedge b_i\otimes b_kX
$$
for any $X\in \Gamma_{n-i}(A)$.  The complexes  $C^n(A)$ are
called dual de Rham complexes, they were considered in
\cite{Jean}.

\begin{theorem}\label{the1}
For a prime $p$ and $C\in \sf D\sf Ab_{\leq 0},$ there are natural
isomorphisms \begin{equation}\label{isogamma1}
\pi_i(L\Gamma_{p}(C[1]))\simeq \pi_i(L\widetilde
\Gamma_{p}(C[1]))\oplus \pi_i\left(C\lotimes \mathbb Z/p[1]\right)
\end{equation}
for all $i\geq 0$.
\end{theorem}

The $p$-torsion terms $C\lotimes \mathbb Z/p$ in theorem are basic
tools for the general $p$-torsion terms in the derived functors of
Lie and super-Lie functors from the functorial point of view.

\begin{lemma}\label{lem1}
For $C\in \sf DAb_{\leq 0},$ such that $H_0(C)=0,$ one has
$\pi_1(L\widetilde \Gamma_p(C))=0.$  If $H_i(C)=0$ for $i\leq m\
(m\geq 1)$, then
\begin{equation}\label{woc}
\pi_i(L\widetilde \Gamma_p(C))=0,\ \ i\leq m+2
\end{equation}
\end{lemma}
\begin{proof}
For $p=2$ this follows from (Satz 12.1 \cite{DoldPuppe}). Since,
for a free abelian group $A$, $H_iC^p(A)=0,\ i>0$ (see
\cite{Franjou}, \cite{Jean}), there is a natural exact sequence (a
truncated part of the dual de Rham complex)
\begin{equation}\label{exx}
0\to \Lambda^p(A)\to \Lambda^{p-1}(A)\otimes A\to\dots\to
\Lambda^2(A)\otimes \Gamma_{p-1}(A)\to \widetilde\Gamma_p(A)\to 0
\end{equation} Observe that $\pi_1(L\Lambda^n(C))=0,\ n\geq 2,\
\pi_1\left(C\lotimes L\Gamma_{p-1}(C)\right)=0$ for $C\in \sf
DAb_{\leq 0},$ such that $H_0(C)=0$, and the result follows from
the exactness of the sequence (\ref{exx}). In the same way one can
get (\ref{woc}) starting with a complex $C$ with $H_i(C)=0,\ i\leq
m$, just observing that, for $n\geq 2,$ $\pi_i(L\Lambda^n(C))=0,\
i<m+1$.
\end{proof}

\begin{lemma}\label{lem2}
For every prime $p$ and $C\in \sf DAb_{\leq 0},$ the suspension
homomorphism
$$
\pi_1(L\widetilde \Gamma_{p}(C))\to \pi_2(L\widetilde
\Gamma_{p}(C[1]))
$$
is the zero map.
\end{lemma}
\begin{proof}
First consider the case $p=2$. We have the following natural
diagram
\begin{equation}\label{diagm}
\xyma{\pi_1(LSP^2(C))\ar@{->}[r]^\simeq \ar@{->}[d]^{\text{susp}}
& L_1SP^2(H_0(C))\ar@{->}[d] \\ \pi_2(LSP^2(C[1]))
\ar@{->}[r]^\simeq & \Lambda^2(H_0(C)) } \end{equation} The
right-hand vertical map is zero by (Corollary 6.6,
\cite{DoldPuppe}). Another way to see why this map is trivial is
to write the cross-effect spectral sequence for
$\pi_*(LSP^2(C[1])$ from \cite{DoldPuppe}. The first page of this
spectral sequence implies that there is an exact sequence
\begin{multline*}
0\to L_1\Lambda^2(H_0(C))\to Tor(H_0(C),H_0(C))\to L_1SP^2(H_0(C))\to\\
\Lambda^2(H_0(C))\to H_0(C)\otimes H_0(C)\to SP^2(H_0(C))\to 0
\end{multline*} where the middle map is exactly the map from
(\ref{diagm}) and it is zero map since the natural transformation
$\Lambda^2(H_0(C))\to H_0(C)\otimes H_0(C)$ is injective.

For $p=3$, $\widetilde \Gamma_3=\EuScript L^3$ (see
\cite{BreenMikhailov}) and one has
$\pi_2(L\widetilde\Gamma_3(C[1]))=0$. Now consider the case $p>3$.
In this case,
$$
\pi_2\left(C[1]\lotimes L\Gamma_{p-1}(C[1])\right)=0,
$$
hence
$$
\pi_2(L\widetilde\Gamma_p(C[1]))\simeq \pi_1(LE^p(C[1])).
$$
Since $H_i(C^p(A))=0,\ i>0$ for every free abelian group $A$, and
$L_i\Lambda^n(C[1])=0,\ i\leq n,$ we conclude that
$$
\pi_1(LE^p(C[1]))=0
$$
and hence the result.
\end{proof}

\vspace{.5cm}\noindent{\it Proof of theorem \ref{the1}} The proof
is by induction on $i$. Lemma \ref{lem1} implies that there is a
natural isomorphism
$$
\pi_1(L\Gamma_{p}(C[1]))\simeq \pi_1(C\lotimes \mathbb Z/p[1])
$$
which is induced by the map $L\Gamma_{p}(C[1])\to C\lotimes
\mathbb Z/p[1]$ from (\ref{tri}).

Consider separately the case $i=2$. The needed statement follows
from the suspension diagram
$$
\xyma{\pi_3\left(C\lotimes\mathbb Z/p[1]\right)\ar@{=}[d]
\ar@{->}[r] & \pi_2(L\widetilde \Gamma_{p^k}(C[1]))\ar@{->}[r] &
\pi_2(L\Gamma_{p^k}(C[1]))
\ar@{->>}[r] & \pi_2\left(C\lotimes \mathbb Z/p[1]\right) \ar@{=}[d]\\
\pi_2\left(C\lotimes \mathbb Z/p\right) \ar@{->}[r] &
\pi_1(L\widetilde \Gamma_{p^k}(C))\ar@{->}[u]^0 \ar@{->}[r] &
\pi_1(L\Gamma_{p}(C)) \ar@{->}[u] \ar@{->>}[r] &
\pi_1\left(C\lotimes \mathbb Z/p\right)\ar@{->}[lu]}
$$
where the left hand vertical homomorphism is zero by lemma
\ref{lem2}.

Now assume by induction, that for all $i\leq j\ (\text{for some}\
j\geq 2)$, there is a natural isomorphism (\ref{isogamma1}), which
is induced by (\ref{tri}). Presenting the complex $C$ as $\dots
\to C_i\buildrel{\partial_i}\over\to C_{i+1}\to \dots$, consider
the subcomplex $Z_j(C)$ of $C$ defined as
\begin{align*}
& (Z_j(C))_i=C_i,\ i\geq j-1,\\ & (Z_j(C))_{j-2}=im(\partial_{i-1}),\\
& (Z_j(C))_i=0,\ i<j-2
\end{align*}
The complex $Z_j(C)$ has the following properties:\\ \\
1) the natural map $Z_j(C)\to C$ induces isomorphisms
$$
\pi_i\left(Z_j(C)\lotimes \mathbb Z/p\right)\simeq
\pi_i\left(C\lotimes \mathbb Z/p\right),\ i\geq j;
$$
2) $H_i(Z_j(C))=0,\ i\leq j-2.$

\vspace{.25cm}

Consider the natural diagram
\begin{equation}\label{dia2}{\footnotesize
\xyma{\pi_{j+2}\left(C\lotimes \mathbb
Z/p[1]\right)\ar@{->}[r]\ar@{=}[d] & \pi_{j+1}(L\widetilde
\Gamma_{p}(C[1]))\ar@{->}[r] & \pi_{j+1}(L\Gamma_{p}(C[1]))
\ar@{->>}[r] & \pi_{j+1}\left(C\lotimes \mathbb Z/p[1]\right)\\
\pi_{j+2}\left(Z_j(C)\lotimes \mathbb Z/p[1]\right) \ar@{->}[r] &
\pi_{j+1}(L\widetilde \Gamma_{p}(Z_j(C)[1]))\ar@{->}[r]
\ar@{->}[u]& \pi_{j+1}(L\Gamma_{p}(Z_j(C)[1]))
\ar@{->>}[r]\ar@{->}[u] & \pi_{j+1}\left(Z_j(C)\lotimes \mathbb
Z/p[1]\right)\ar@{=}[u]}}
\end{equation}
Lemma \ref{lem1} implies that
$\pi_{j+1}(L\widetilde\Gamma_{p}(Z_j(C)[1]))=0$. The needed
splitting now follows from diagram (\ref{dia2}). The inductive
step is complete and the splitting (\ref{isogamma1}) proved for
all $i$.\ $\Box$

\begin{prop}
The sequence
$$
LSP^2(C[1])\to L\Gamma_2(C[1])\to C\lotimes \mathbb Z/2[1]
$$
does not split in the category $\sf DAb_{\leq 0}$.
\end{prop}
\begin{proof}
We will prove the statement for the simplest case, when $C$ is a
free abelian group. Suppose that $L\Gamma_2(C[1])\simeq
LSP^2(C[1])\oplus C\lotimes \mathbb Z/2[1]$. Then
\begin{multline*}
\pi_2\left(L\Gamma_2(C[1])\lotimes \mathbb Z/2\right)\simeq
\pi_2\left(LSP^2(C[1])\lotimes \mathbb Z/2\oplus C\lotimes\mathbb
Z/2\lotimes \mathbb Z/2[1]\right)\simeq\\ \Lambda^2(C)\otimes
\mathbb Z/2\oplus C\otimes \mathbb Z/2
\end{multline*}
However, $L\Gamma_2(C[1])$ can be presented as complex
$$
(C\otimes C\otimes \mathbb Z/2\to \Gamma_2(C)\otimes\mathbb
Z/2)[1]
$$
in the category $\sf DAb_{\leq 0}$ and
$$
\pi_2\left(L\Gamma_2(C[1])\lotimes \mathbb
Z/2\right)=\ker\{C\otimes C\otimes \mathbb Z/2\to
\Gamma_2(C)\otimes\mathbb Z/2\}.
$$
Any natural transformation $C\otimes \mathbb Z/2\to C\otimes
C\otimes \mathbb Z/2$ is zero. Hence, the assumed splitting  not
possible.
\end{proof}

For every $m\geq 2,$ there is the following d\'ecalage isomorphism
in the derived category
$$
L\Gamma_m(C)[2m]\simeq L\Lambda^m(C[1])[m]\simeq LSP^m(C[2])
$$
which implies that the isomorphisms (\ref{isogamma1}) can be
written in the following form. For a prime $p$ and $C\in \sf D\sf
Ab_{\leq 0},$ there are natural isomorphisms
\begin{equation}\label{isosp1} \pi_i(L\Lambda^{p}(C[2]))\simeq
\pi_i(L\widetilde \Gamma_{p}(C[1])[p])\oplus \pi_i\left(C\lotimes
\mathbb Z/p[p+1]\right)
\end{equation}
for all $i\geq 0$.

Recall that, for a free abelian group $A$, there is the following
long exact sequence which is called Koszul complex:
\begin{equation}\label{kosz1}
0\to \Lambda^p(A)\to \Lambda^{p-1}(A)\otimes A\to \dots \to
A\otimes SP^{p-1}(A)\to SP^p(A)\to 0
\end{equation}
For $k=1,\dots, p-1$, define the functor
$$
V_{p,k}: \sf Ab\to Ab
$$
as a kernel of a map in the Koszul complex
$$
V_{p,k}(A)=ker\{\Lambda^{p-k}(A)\otimes SP^{k}(A)\to
\Lambda^{p-k+1}(A)\otimes SP^{k+1}\}.
$$
Clearly, $V_{p,p-1}(A)=J_p(A),\ V_{p,1}(A)=\Lambda^p(A)$.
\begin{lemma}\label{ele}
For $k=1,\dots, p-1$, and $C\in \sf DAb_{\leq 0},$ there are
natural splitting monomorphisms
\begin{equation}\label{nmono}\pi_i\left(C\lotimes \mathbb Z/p\
[p+k]\right)\hookrightarrow \pi_i(LV_{p,k}(C[2])))
\end{equation}
for all $i\geq 0$.
\end{lemma}
\begin{proof}
For $k=1$ the needed splitting monomorphisms are given by
(\ref{isosp1}). Assume that we have splitting monomorphisms
(\ref{isosp1}) for a fixed $k<p-1$. Since the sequence
(\ref{kosz1}) is exact, for $k=1,\dots, p-1$ and a free abelian
$A$, there is a natural short exact sequence
$$
0\to V_{p,k}(A)\to \Lambda^{p-k}(A)\otimes SP^k(A)\to
V_{p,k+1}(A)\to 0
$$
Observe that, for $C\in \sf DAb_{\leq 0},$ such that $H_i(C)=0,\
i\leq m$,
$$
\pi_i\left(L\Lambda^k(C[2])\lotimes LSP^{p-k}(C[2])\right)=0,\
i<2m+2p-k+2.
$$
In particular, there are natural isomorphsism
$$
\pi_{i+1}(LV_{p,k+1}(C[2]))\simeq\pi_i(LV_{p,k}(C[2])
$$
for $i<2p-k+2$. This shows that the needed splitting monomorphisms
$$
\pi_{i+1}(C\lotimes \mathbb Z/p\ [p+k+1])\simeq\pi_i\left
(C\lotimes \mathbb Z/p\ [p+k]\right)\hookrightarrow
\pi_i(LV_{p,k}(C[2]))\simeq \pi_{i+1}(LV_{p,k+1}(C[2]))
$$
exist for $i<2p-k+2$. For a fixed $j\geq 2,$ consider the complex
$Z_j(C)$ from the proof of theorem \ref{the1} with a natural map
$Z_j(C)\to C$. One has the following natural diagram
$$
\xyma{\pi_{j+1}(LV_{p,k+1}(Z_j(C)[2]))\ar@{->}[dd]\ar@{->}[r]^\simeq
& \pi_j(LV_{p,k}(Z_j(C)[2]))\ar@{->}[dd]\ar@{->>}[rd]\\ & & \pi_j\left(Z_j(C)\lotimes \mathbb Z/p[p+k]\right)\ar@{=}[dd]\ar@/_15pt/[lu]\\
\pi_{j+1}(LV_{p,k+1}(C[2]))\ar@{->}[r] &
\pi_j(LV_{p,k}(C[2]))\ar@{->>}[rd]\\ & & \pi_j\left(C\lotimes
\mathbb Z/p[p+k]\right)\ar@/_15pt/[lu]}
$$
Here the diagonal maps are natural monomorphisms and epimorphisms
(\ref{nmono}). The needed splitting monomorphism $$
\pi_{j+1}(C\lotimes \mathbb Z/p\
[p+k+1])\hookrightarrow\pi_{j+1}(LV_{p,k+1}(C[2]))
$$
now follows from the above diagram.
\end{proof}

\begin{theorem}\label{primelieth}
For a prime $p$ and $C\in \sf DAb_{\leq 0},$ there are natural
splitting monomorphisms
\begin{equation}\label{monomo1}
\pi_i\left(C\lotimes \mathbb Z/p\ [2p-1]\right)\hookrightarrow
\pi_i(L\EuScript L^p(C[2]))
\end{equation}
for all $i\geq 0$.
\end{theorem}
\begin{proof}
Recall that, for a free abelian group $A$, there is a natural
short exact sequence (see \ref{cx1}) \begin{equation}\label{xc1}
0\to \widetilde J^p(A)\to \EuScript L^p(A)\to J^p(A)\to 0
\end{equation}
It follows from \cite{Curtis:65} that the functor $\widetilde J^p$
can be decomposed as a sequence of functors of the type
$F_1\otimes\dots \otimes F_k$ for some $k\geq 2,$ such that each
$F_j,\ i=1,\dots, k$ is a composition of symmetric powers and
functors $J^l,\ l<p$.

Recall that, for $C\in \sf DAb_{\leq 0}$, such that $H_i(C)=0,\
i<k$, by \cite{DoldPuppe}, Satz 12.1
\begin{equation}\label{dpuppe}
\pi_i(LSP^n(C))=0,\begin{cases} \text{for}\ i<n,\ \text{when}\
k=1,\\
\text{for}\ i<k+2n-2,\ \text{provided}\ k>1.
\end{cases}
\end{equation}
and
\begin{equation}\label{dpuppe1}
\pi_i(LJ^n(C))=0,\begin{cases} \text{for}\ i<n+1,\ \text{when}\
k=1,\\
\text{for}\ i<k+2n-1,\ \text{provided}\ k>1.
\end{cases}
\end{equation}
Curtis decomposition of the functor $\widetilde J^p$ together with
(\ref{dpuppe}) and (\ref{dpuppe1}) imply that, for $C\in \sf
DAb_{\leq 0},$ such that $H_i(C)=0,\ i<k, (k>1)$
$$
\pi_i(L\widetilde J^p(C))=0,\ i<2p+2k-2.
$$
Now we construct the needed splitting monomorphism (\ref{monomo1})
by induction on $i$. The argument is the same as in the proof of
theorem \ref{the1} and lemma (\ref{ele}). For a fixed $i$, we
consider the natural map $Z_i(C)\to C$ and compare the sequences
of functors (\ref{xc1}) for the complexes $Z_i(C)[2]$ and $C[2]$.
\end{proof}

For a prime $p$, $i,m\geq 1$, and $C\in \sf DAb_{\leq 0}$ consider
the following map
\begin{multline*}
\pi_i\left(L\EuScript L^m(C)\lotimes \mathbb Z/p\
[2p+2m-3]\right)\to \pi_i\left(L\EuScript
L^m(C[2])[2p-3]\lotimes \mathbb Z/p\right)\to\\
\pi_i\left(L\EuScript L^p\circ \EuScript L^m(C[2])\right)\to
\pi_i\left(L\EuScript L^{pm}(C[2])\right)
\end{multline*}
where the last map is induced by the natural transformation
$\EuScript L^p\circ \EuScript L^m\to \EuScript L^{pm}$. Denote
this map by $w_{i,p,m}$.

\begin{lemma} If $(m,p)=1$ and homology of $C$ are torsion-free,
then $w_{i,p,m}$ is a monomorphism for all $i\geq 1$.
\end{lemma}

\begin{proof} Consider a prime decomposition of $m$: $m=p_1^{k_1}\dots
p_s^{k_s}$, where $p_j$ are primes and $k_j\geq 0.$ We will prove
the statement by induction on $k(m)=k_1+\dots+k_s$. If $k(m)=0$,
then the statement follows from theorem \ref{primelieth}.

\vspace{.5cm}\noindent{\it Step 1.} First we show that $w_{p,m}$
is injective if $C=\mathbb Z[l]$ for some $l$. The description of
the derived functors $L_*\EuScript L^m(\mathbb Z,l)$ given in
section \ref{section4} implies that
$$
\pi_i\left(L\EuScript L^m(\mathbb Z,l)\lotimes \mathbb Z/p
\right)=0,\ m\neq 2
$$
Only the case which we have to consider here is $m=2$ and an odd
$l$. In this case
$$
L\EuScript L^m(\mathbb Z,l)\lotimes \mathbb Z/p\simeq \mathbb Z/p\
[2l]
$$
The map
$$
\EuScript L^m(\mathbb Z,l)[2m]\lotimes \mathbb Z/p\to \EuScript
L^m(\mathbb Z,l+2)\lotimes \mathbb Z/p
$$
is an equivalence in $\sf DAb_{\leq 0}$. Now observe that, by
theorem \ref{kan1}, the natural map
$$
\pi_i\left(L\EuScript L^p\circ \Lambda^2(\mathbb
Z,l+2\right)\simeq L_i\EuScript L^p(\mathbb Z,2l+4)\to
L_i\EuScript L^{2p}(\mathbb Z,l+2)
$$
is an isomorphism. Therefore, $w_{i,p,m}$ is a monomorphism for
$C=\mathbb Z[l]$.

\vspace{.5cm}\noindent{\it Step 2.} Now we assume that, for
complexes $C_1$ and $C_2$, the maps $w_{p,m}$ are injective.
 Consider the cross-effects of the functors which
appear in the definition of $w_{p,m}$. We have
$$
\pi_i\left(L\EuScript L^m(C_1[2]\ |\ C_2[2])\lotimes\mathbb
Z/p\right)\simeq \bigoplus_{d|m,\ 1\leq d< m,\ D\in
J_{m/d}}\pi_i\left(L\EuScript L^d(D_{i_1}\lotimes \dots
D_{i_{m/d}}[\frac{2m}{d}])\lotimes\mathbb Z/p\right)
$$
where $\{D_1,\dots,D_{i_{m/d}}\}=\{C_1,C_2\}.$ For every $d$, and
$D\in J_{m/d}$,  we have the following commutative diagram

$$
\xyma{ \pi_i\left(L\EuScript L^d(D_1\lotimes\dots\lotimes
D_{i_{m/d}}[\frac{2m}{d}])\lotimes \mathbb Z/p\
[2p-3]\right)\ar@{->}[d]^{w_{i,p,d}} \ar@{->}[r] &
\pi_i(L\EuScript L^p\circ
L^m(C_1[2]\ |\ C_2[2])) \ar@{->}[d] \\
\pi_i\left(L\EuScript L^{pd}(D_1\lotimes \dots\lotimes
D_{i_{m/d}}[\frac{2m}{d}])\right)\ar@{>->}[r] & \pi_i(L\EuScript
L^{pm}(C_1[2]\ |\ C_2[2]))}
$$
The map $w_{i,p,d}$ is a monomorphism by induction, therefore, the
map $w_{i,p,m}$ induce the monomorphism
$$ \pi_i\left(L\EuScript L^m(C_1[2]\ |\ C_2[2])\lotimes \mathbb
Z/p\ [2p-3]\right)\hookrightarrow \pi_i\left(L\EuScript L^{pm}
(C_1[2]\ |\ C_2[2]))\right)
$$

Now the statement follows by induction from Step 1 and Step 2,
since $C$ is unnaturally equivalent to a direct sum of its
homology considered in the corresponding dimensions.
\end{proof}

\begin{prop}
Let $X$ be a free abelian simplicial group, then, for a prime $p$
and $n\geq 1$, there are natural isomorphisms
\begin{equation}\label{prekappa}
\pi_i(\EuScript L^p(\Sigma^{2n}(X))\otimes \mathbb Z/p)\simeq
\bigoplus_{w\in \overline{\mathcal V}_{2n,1}^{(p)}}\pi_i(X\otimes
\mathbb Z/p\ [2n+d(w)])
\end{equation}
for $i<2np$.
\end{prop}
\vspace{.5cm}

\section{Map $\kappa$}
For an abelian $A$, $m,n\geq 1$, denote the graded abelian groups
\begin{align*}
& \Theta_{m}(A,2n):=\bigoplus_{\substack{p\ \text{prime}\\
p^k|m\\ p^{k+1}\nmid m}}\bigoplus_{\substack{i=1,\dots, k\\
w\in \overline{\mathcal W}_{\frac{2nm}{p^{i}},i}^{(p)}}} \EuScript
N^{\frac{m}{p^{i}};p}(A)[\frac{2nm}{p^{i}}+d(w)]\\
& \Theta_m(A,2n+1):=\bigoplus_{\substack{p\ \text{prime}\\
p^k|m\\ p^{k+1}\nmid m}}\bigoplus_{\substack{i=1,\dots,k\\ w\in
\overline{\mathcal W}_{\frac{(2n+1)m}{p^{i}},i}^{(p)}}} \EuScript
N_s^{\frac{m}{p^{i}};p}(A)[\frac{(2n+1)m}{p^{i}}+d(w)]
\end{align*}
and their analogs where the indexes are from $\overline{\mathcal
V}$-sets:
\begin{align*}
& \widetilde\Theta_m(A,2n):=\bigoplus_{\substack{p\ \text{prime}\\
p^k|m\\ p^{k+1}\nmid m}}\bigoplus_{\substack{i=1,\dots, k\\
w\in \overline{\mathcal V}_{\frac{2nm}{p^{i}},i}^{(p)}}} \EuScript
N^{\frac{m}{p^{i}};p}(A)[\frac{2nm}{p^{i}}+d(w)]\\
& \widetilde\Theta_m(A,2n+1):=\bigoplus_{\substack{p\ \text{prime}\\
p^k|m\\ p^{k+1}\nmid m}}\bigoplus_{\substack{i=1,\dots,k\\ w\in
\overline{\mathcal V}_{\frac{(2n+1)m}{p^{i}},i}^{(p)}}} \EuScript
N_s^{\frac{m}{p^{i}};p}(A)[\frac{(2n+1)m}{p^{i}}+d(w)]
\end{align*}
\begin{theorem}
For a free abelian group $A$ and $m,n\geq 1$, there is a natural
isomorphism of graded abelian groups
$$
\bigoplus_{i=1}^{nm-1}L_i\EuScript L^m(A,n)[i]\simeq
\Theta_{m}(A,n)
$$
\end{theorem}
\begin{proof}
{\it Case I: $n$ even.} In this case we can write the dimension
$2n$ instead of $n$. First we prove that there is a natural
isomorphism of graded abelian groups
\begin{equation}
\bigoplus_{p|m\ p\ prime} \bigoplus_{i=1}^{2mn-1} \pi_i(\EuScript
L^{m}N^{-1}(A[2n])\otimes \mathbb Z/p)[i]\simeq
\widetilde\Theta_{m}(A,2n)
\end{equation}

For a prime $p$, define $\mathcal V_{*,0}^{(p)}=\overline{\mathcal
V}_{*,0}^{(p)}=\{0\}$ and $d(0)=0$. For $w\in \overline{\mathcal
V}_{2nd,i}^{(p)}\ (i\geq 0),$ define the map
$$
\kappa_{d,w}: \EuScript N^{d;p}(A)\to \pi_{2nd+d(w)}(\EuScript
L^{dp^i}N^{-1}(A[2n])\otimes \mathbb Z/p)
$$
as follows. For $i=0$, this map is a natural isomorphism
$$
\kappa_{d,0}: \EuScript
N^{d;p}(A)\buildrel{\simeq}\over\longrightarrow
\pi_{2nd}(\EuScript L^{d}N^{-1}(A[2n])\otimes \mathbb Z/p).
$$
Suppose the map $\kappa_{d,w}$ is defined for all $w\in
\overline{\mathcal V}_{2nd,l}^{(p)}$ for $l<i$. Let
$w=(\nu_{j_1},\dots,\nu_{j_i})$ and $v:=(\lambda_{j_1},\dots,
\lambda_{j_{i-1}})\in \overline{\mathcal V}_{2nd,i-1}^{(p)}$. Let
$$
d(\nu_{j_i}):=\begin{cases} (2p-2)j_i-1,\ \text{if}\
\nu_{j_i}=\lambda_{j_i}\\ (2p-2)j_i,\ \text{if}\
\nu_{j_i}=\mu_{j_i}\end{cases}
$$
Clearly, $d(w)=d(w')+d(\nu_{j_i})$. Let $C:=\EuScript
L^{dp^{i-1}}N^{-1}(A[2n])\otimes \mathbb Z/p\langle
2nd+d(w')\rangle.$ Now we define the map $\kappa_{d,w}$ as the
following composition map
$$
\xyma{\EuScript N^{d;p}(A)\ar@{->}[d]^{\kappa_{d,w'}}\\
\pi_{2nd+d(w')}({\EuScript L}^{dp^{i-1}}N^{-1}(A[2n])\otimes
\mathbb Z/p) \ar@{->}[r]^{\ \ \ \ \ \ \ \ \ \ \ \simeq} &
\pi_{2nd+d(w')}(C)\ar@{>->}[d]\\ &
\pi_{2nd+d(w')+d(\nu_{j_i})}(\bigoplus_{w\in \overline{\mathcal
V}_{2nd+d(w'),1}^{(p)}}C[2nd+d(w)] ) \ar@{->}[d]^{\simeq}\\
\pi_{2nd+d(w)}(\EuScript L^p\circ \EuScript
L^{dp^{i-1}}N^{-1}(A[2n])\otimes \mathbb Z/p) \ar@{->}[d]&
\pi_{2nd+d(w)}(\EuScript L^p(C))\ar@{->}[l]\\
\pi_{2nd+d(w)}(\EuScript L^{dp^i}N^{-1}(A[2n])\otimes \mathbb Z/p)
}
$$

Consider the case $A=\mathbb Z$. It follows from theorem
\ref{kan2}, that it is enough to consider the case $m=p^k$ for a
prime $p$ and $k\geq 1$. In this case, we have a map of graded
abelian groups
\begin{multline}\label{kappapk}
\bigoplus_{\substack{j=0,\dots,k-1\\ w\in \overline{\mathcal
V}_{2np^j,k-j}^{(p)}}}\EuScript N^{p^j;p}(\mathbb
Z)[2np^j+d(w)]\simeq \bigoplus_{\substack{j=0,\dots,k-1\\ w\in
\overline{\mathcal V}_{2np^j,k-j}^{(p)}}}\mathbb Z/p\ [2np^j+d(w)]\to\\
\pi_*(\EuScript L^{p^k}N^{-1}(\mathbb Z[2n])\otimes \mathbb
Z/p)\simeq \bigoplus_{v\in \overline{\mathcal
V}_{2n,k}^{(p)}}\mathbb Z/p\ [2n+d(w)]
\end{multline}
defined as a sum of $\kappa_{p^j;p}$-maps for $j=1,\dots,k$. It
follows immediately from the construction of the map
$\kappa_{p^j;p}$ that, for $j\geq 1$, the summand $\mathbb Z/p$
which corresponds to  $w=(w_1,\dots,w_{k-j})\in \overline{\mathcal
V}_{2np^j,k-j}^{(p)},$ goes to the term $\mathbb Z/p$ which
corresponds to
$$
(\mu_n,\dots, \mu_{p^{j-1}n}, w_1,\dots, w_{k-j})\in \mathcal
V_{2n,k}
$$
We have
$$
2n+d(\mu_n,\dots, \mu_{p^{j-1}n}, w_1,\dots,
w_{k-j})=2np^j+d(w_1,\dots,w_{k-j})
$$
Therefore, the map (\ref{kappapk}) is an isomorphism.

Now we compare cross-effects. Observe that, for $l|d$, $w\in
\overline{\mathcal V}_{2nd,i}^{(p)}$, and $C_{i_1}\otimes
\dots\otimes C_{i_{d/l}}\in J_{d/l},\ C_i\in \{A,B\}$ we have the
following natural diagram
\begin{equation}\label{nicediagram}
\xyma{\EuScript N^{l;p}(C_{i_1}\otimes \dots\otimes
C_{i_{d/l}})\ar@{>->}[d] \ar@{->}[r]^{\kappa_{l;w}\ \ \ \ \ \ \ \
\ \ \ \ \ \ \ \ \ \ \  } &
\pi_{2nd+d(w)}(\EuScript L^{lp^i}N^{-1}(C_{i_1}\otimes \dots\otimes C_{i_{d/l}}[\frac{2nd}{l}])\otimes \mathbb Z/p)\ar@{>->}[d]\\
\EuScript N^{d;p}(A\oplus B) \ar@{->}[r]^{\kappa_{d;w}\ \ \ \ \ \
\ \ \ \ \ \ } & \pi_{2nd+d(w)}(\EuScript L^{dp^i}N^{-1}(A\oplus
B[2n])\otimes \mathbb Z/p)}
\end{equation}
By theorem \ref{abstractiso}, the graded abelian groups
\begin{equation}\label{compa11}
\bigoplus_{p|m\ p\ prime} \bigoplus_{i=1}^{2mn-1} \pi_i(\EuScript
L^{m}N^{-1}(A[2n])\otimes \mathbb Z/p)[i]\ \ \text{and}\ \
\widetilde\Theta_{m}(A,2n)
\end{equation}
are abstractly isomorphic. The above observations show that the
direct sum of maps $\kappa_{d,w}$ defines a natural surjective map
from the right hand side of (\ref{compa11}) to the left hand side.
Components in both graded functors in (\ref{compa11}) are
finitely-generated for finitely-generated $A$ and these functors
commute with direct limits. Therefore, there is a natural
isomorphism of graded abelian groups (\ref{compa11}).

Observe that, for $i\geq 1$, there is a natural exact sequence
\begin{equation}\label{silence}
0\to L_i\EuScript L^m(A,2n)\otimes \mathbb Z/p\to \pi_i(\EuScript
L^{m}N^{-1}(A[2n])\otimes \mathbb Z/p)\to Tor(L_{i-1}\EuScript
L^m(A,2n),\mathbb Z/p)\to 0
\end{equation}
For $i<2nm$, $L_i\EuScript L^m(A,2n)\otimes \mathbb Z/p$ is
naturally isomorphic to the $p$-torsion component of $L_i\EuScript
L^m(A,2n)$. This follows from the abstract description of the
derived functors $L_i\EuScript L^m(A,2n)$ $(i\leq 2mn)$: for a
free $A$ these are direct sums of $\mathbb Z/p$-vector spaces for
different primes $p$.

Now consider the obvious natural embedding
$$
p_{m,n}:\Theta_{m}(A,2n)\subset
\widetilde\Theta_{m}(A,2n)\simeq\bigoplus_{p|m\ p\ prime}
\bigoplus_{i=1}^{2mn-1} \pi_i(\EuScript L^{m}N^{-1}(A[2n])\otimes
\mathbb Z/p)[i]
$$
induced by the obvious embeddings of the indexed sets in the
definition of $\Theta_{m}(A,2n)$ and $\widetilde
\Theta_{m}(A,2n)$. Now observe that the image of the left hand map
in (\ref{silence}) lies in the image of $p_{m,2n}$. This follows
from theorem \ref{kan2} and the diagram (\ref{nicediagram}). Since
the graded abelian groups
$$
\bigoplus_{p|m\ p\ prime}\bigoplus_{i=1}^{2mn-1}L_i\EuScript
L^m(A,2n)\otimes \mathbb Z/p[i]\ \ \text{and}\ \ \Theta_{m}(A,2n)
$$
are abstractly isomorphic, they are naturally isomorphic as well,
since the left hand map in (\ref{silence}) defines a natural
monomorphism between them and they are isomorphic direct sums of
$\mathbb Z/p$-vector spaces.

\vspace{.5cm}\noindent{\it Case II: $n$ odd.} This case is
analogous to the Case I. First we prove that the graded abelian
groups
\begin{equation}\label{graded22}
\bigoplus_{p|m\ p\ prime} \bigoplus_{i=1}^{2mn-1} \pi_i(\EuScript
L^{m}N^{-1}(A[2n+1])\otimes \mathbb Z/p)[i]\ \ \text{and}\ \
\widetilde\Theta_{m}(A,2n+1)
\end{equation}
are naturally isomorphic. These graded abelian groups are
abstractly isomorphic by theorem \ref{abstractiso}, hence, as
above, it is enough to construct a natural surjective map between
them. For a prime $p$, $w\in \overline{\mathcal
V}_{(2n+1)d,i}^{(p)},$ the construction of the map
$$
\kappa_{d,w}': \EuScript N_s^{d;p}(A)\to
\pi_{(2n+1)d+d(w)}(\EuScript L^{dp^i}N^{-1}(A[2n+1])\otimes
\mathbb Z/p)
$$
is analogous to the construction of the map $\kappa_{d,w}$.
Repeating the proof from the Case I, we get a natural isomorphism
of the graded abelian groups (\ref{graded22}). A transition from
the homotopy groups of  $\EuScript L^mN^{-1}(A[2n+1])\otimes
\mathbb Z/p$ to $\EuScript L^{m}(A,2n+1)$ is analogous to the Case
I.
\end{proof}

\noindent{\bf Remark.} Another way, how to define a natural map is
to consider a composition map
\begin{multline}
\EuScript N^{d;p}(A)\hookrightarrow \bigoplus_{\mathcal
V_{2nd,i}^{(p)}}\EuScript N^{d;p}(A)=\\ \pi_{2nd}(\EuScript
L^{d}N^{-1}(A[2n])\otimes \mathbb Z/p)\otimes
\pi_{2nd+d(w)}(\EuScript L^{p^i}K(\mathbb
Z,2nd)\otimes \mathbb Z/p)\to \\
\pi_{2nd+d(w)}(\EuScript L^{dp^i}N^{-1}(A[2n])\otimes \mathbb Z/p)
\end{multline}
Here the first inclusion is induced by the inclusion
$\overline{\mathcal V}_{2nd,i}^{(p)}\subset \mathcal
V_{2nd,i}^{(p)}$. However, it is not straightforward from the
definition that this map is a homomorphism on $A$.

\section{Homotopy groups}
\vspace{.5cm}
Given a simplicial set $K$ with simply connected
geometric realization $|K|$, the Kan loop group construction $GX$
has the following property: there is an equivalence of fiber
sequences:
$$
\xyma{|[GK,GK]| \ar@{->}[r] \ar@{->}[d]^{\simeq} & |GK|
\ar@{->}[r] \ar@{->}[d]^{\simeq} &
|(GK)_{ab}|\ar@{->}[d]^{\simeq}\\
\Omega \Gamma |K| \ar@{->}[r] & \Omega |K| \ar@{->}[r] & \Omega
SP^\infty |K|}
$$
In particular, the Hurewicz homomorphism
$$
\pi_n(|K|)\to H_n(|K|),\ n\geq 2
$$
is given as $\pi_{n-1}(GK)\to \pi_{n-1}((GK)_{ab}).$

The lower central series filtration of $GK$ gives rise to the long
exact sequence
\begin{equation*} \dots\to \pi_{i+1}(GK/\gamma_r(GK))\to
\pi_i(\gamma_r(GK)/\gamma_{r+1}(GK))\to
\pi_i(GK/\gamma_{r+1}(GK))\to \pi_i(GK/\gamma_r(GK))\to \dots
\end{equation*}
This exact sequence defines a graded exact couple, which gives
rise to a natural spectral sequence $E(K)$ with the initial terms
\begin{align*}
& E_{r,q}^1(K)=\pi_q(\gamma_{r}(GK)/\gamma_{r+1}(GK))
\end{align*}
and  differentials 
\begin{align}
\label{dinm} & d^i_{p,q}: E_{r,q}^i(K)\to E_{r+i,\, q-1}^i(K).
\end{align}

This spectral sequence $E^i(K)$ converges to $E^\infty(K)$ and
$\oplus_rE_{p,q}^\infty$ is the graded group associated to the
filtration on $\pi_q(GK)=\pi_{q+1}(|K|)$. The groups $E^1(K)$ are
homology invariants of $K$. By the Magnus-Witt isomorphism, the
spectral sequence  can be rewritten as \bee \label{curtisgen1}
 E_{r,q}^1(K) = \pi_q(\Le^r(\widetilde{\Z}K,-1)) \Longrightarrow \pi_{q+1}(|K|).
\ee
 since  the abelianization $  GK_{\mathrm{ab}}:= GK/\gamma_2(GK)$ of
 $GK$ corresponds to
    the reduced chains
 $\widetilde{\Z}K$  on $K$, with degree shifted by 1. When $K = M(A,n)$,   $\widetilde{\Z}K$ corresponds to an
 Eilenberg-Mac Lane space $K(A,n)$ so that the spectral sequence is
 simply  of
 the form
\bee \label{curtisgen2} E^1_{r,q} = L_q \Le^r(A, n-1)
\Longrightarrow \pi_{q+1}(M(A,n))\,. \ee In particular,
\[E^1_{1,q} = \pi_q(K(A,n-1)) = \begin{cases} A, \ q= n-1\\0, \ q \neq
  n-1
\end{cases}
\]

\begin{table}[ht]
\begin{tabular}{ccccccccccccccc}
 $q$ & \vline & $E_{1,q}^1$ & $E_{2,q}^1$ & $E_{3,q}^1$ & $E_{4,q}^1$ & $E_{5,q}^1$ & $E_{6,q}^1$ & $E_{7,q}^1$ & $E_{8,q}^1$ & $E_{12,q}^1$ & $E_{16,q}^1$ & $E_{32,q}^1$\\
 \hline

7 & \vline & 0 & 0 & 0 & $\mathbb Z/2$ & $\mathbb Z/2^2$ & $\mathbb Z/2$ & $\mathbb Z/2$ & $\mathbb Z/2^4$ & $\mathbb Z/2$ & $\mathbb Z/2^7$ & $\mathbb Z/2^5$\\
6 & \vline & 0 & 0 & 0 & $\mathbb Z/4$ & $\mathbb Z/2$ & $\mathbb Z/2$ & 0 & $\mathbb Z/2^4$ & 0 & $\mathbb Z/2^4$ & $\mathbb Z/2$\\
5 & \vline & 0 & 0 & 0 & $\mathbb Z/2^2$ & 0 & $\mathbb Z/2$ & 0 & $\mathbb Z/2^3$ & 0 & $\mathbb Z/2$ & 0\\
4 & \vline & 0 & 0 & $\mathbb Z/2$ & $\mathbb Z/2^2$ & 0 & 0 & 0 & $\mathbb Z/2$ & 0 & 0 & 0\\
3 & \vline & 0 & $\mathbb Z/2$ & 0 & $\mathbb Z/2$ & 0 & 0 & 0 & 0 & 0 & 0 & 0\\
2 & \vline & 0 & $\mathbb Z/4$ & 0 & 0 & 0 & 0 & 0 & 0 & 0 & 0 & 0\\
1 & \vline & $\mathbb Z/2$ & 0 & 0 & 0 & 0 & 0 & 0 & 0 & 0 & 0 & 0\\
\end{tabular}
\vspace{.5cm} \caption{The spectral sequence for $\Sigma \mathbb
RP^2=M(\mathbb Z/2,2)$}
\end{table}

\begin{table}[ht]
\begin{tabular}{ccccccccccccccc}
 $q$ & \vline & $E_{1,q}^1$ & $E_{2,q}^1$ & $E_{3,q}^1$ & $E_{4,q}^1$ & $E_{5,q}^1$ & $E_{6,q}^1$ & $E_{7,q}^1$\\
 \hline

10 & \vline & 0 & 0 & 0 & 0 & $\EuScript L^5(A)$ & 0 & 0\\
9 & \vline & 0 & 0 & 0 & 0 & $A\otimes \mathbb Z/5$ & $\EuScript L^3(A)\otimes \mathbb Z/2$ & 0\\
8 & \vline & 0 & 0 & 0 & $\EuScript L^4(A)$ & 0 & 0 & 0\\
7 & \vline & 0 & 0 & 0 & $\Gamma_2(A)\otimes \mathbb Z/2$ & 0 & $\Lambda^2(A)\otimes \mathbb Z/3\oplus \EuScript L^3(A)\otimes \mathbb Z/2$ & 0\\
6 & \vline & 0 & 0 & $\EuScript L^3(A)$ & 0 & 0 & 0 & 0\\
5 & \vline & 0 & 0 & $A\otimes \mathbb Z/3$ & $\Gamma_2(A)\otimes \mathbb Z/2$ & 0 & 0 & 0\\
4 & \vline & 0 & $\Lambda^2(A)$ & 0 & $A\otimes\mathbb Z/2$ & 0 & 0 & 0\\
3 & \vline & 0 & $A\otimes \mathbb Z/2$ & 0 & 0 & 0 & 0 & 0\\
2 & \vline & $A$ & 0 & 0 & 0 & 0 & 0 & 0\\
\end{tabular}
\vspace{.5cm} \caption{The spectral sequence for $M(A,3)\ (A\
\text{free})$}
\end{table}
{\small
\begin{table}[ht]
\begin{tabular}{ccccccccccccccc}
 $q$ & \vline & $E_{8,q}^1$ & $E_{9,q}^1$ & $E_{12,q}^1$ & $E_{16,q}^1$ & $E_{24,q}^1$\\
 \hline

10 & \vline & $\Gamma_2(A)\otimes \mathbb Z/2$ & 0 & $\EuScript L^3(A)\otimes \mathbb Z/2$ & $(\Gamma_2(A)^{\oplus 3}\oplus A^{\oplus 2})\otimes \mathbb Z/2$ & $(\EuScript L^3(A)\otimes \mathbb Z/2)^{\oplus 2}$\\
& \vline & & & & $\oplus \EuScript N^{4;2}(A)$ & & \\
9 & \vline & $\Gamma_2(A)\otimes \mathbb Z/2\oplus \EuScript N^{4;2}(A)$ & $\EuScript N^{3;3}(A)$ & $\EuScript L^3(A)\otimes \mathbb Z/2$ & $(\Gamma_2(A)^{\oplus 2}\oplus A^{\oplus 2})\otimes \mathbb Z/2$ & $\EuScript L^3(A)\otimes \mathbb Z/2$\\
8 & \vline & $(\Gamma_2(A)\oplus A)\otimes \mathbb Z/2$ & $A\otimes \mathbb Z/3$ & $\EuScript L^3(A)\otimes \mathbb Z/2$ & $(\Gamma_2(A)^{\oplus 2}\oplus A)\otimes \mathbb Z/2$ & 0\\
7 & \vline & $\Gamma_2(A)\otimes \mathbb Z/2$ & 0 & 0 & $(\Gamma_2(A)\oplus A^{\oplus 2})\otimes \mathbb Z/2$ & 0\\
6 & \vline & $(\Gamma_2(A)\oplus A)\otimes \mathbb Z/2$ & 0 & 0 & $A\otimes \mathbb Z/2$ & 0\\
5 & \vline & $A\otimes \mathbb Z/2$ & 0 & 0 & 0 & 0\\
4 & \vline & 0 & 0 & 0 & 0 & 0\\
\end{tabular}
\end{table}
\begin{table}[ht]
\begin{tabular}{ccccccccccccccc}
$q$ & \vline & $E_{32,q}^1$ & $E_{48,q}^1$ & $E_{64,q}^1$ & $E_{128,q}^1$ & $E_{256,q}^1$\\
 \hline

10 & \vline & $(\Gamma_2(A)^{\oplus 4}\oplus A^{\oplus 4})\otimes \mathbb Z/2$ & $\EuScript L^3(A)\otimes \mathbb Z/2$ & $(\Gamma_2(A)^{\oplus 4}\oplus A^{\oplus 6})\otimes \mathbb Z/2$ & $(\Gamma_2(A)\oplus A^{\oplus 5})\otimes \mathbb Z/2$ & $A\otimes \mathbb Z/2$\\
9 & \vline & $(\Gamma_2(A)^{\oplus 3}\oplus A^{\oplus 3})\otimes \mathbb Z/2$ & 0  & $(\Gamma_2(A)\oplus A^{\oplus 4})\otimes \mathbb Z/2$ & $A\otimes \mathbb Z/2$ & 0\\
8 & \vline & $(\Gamma_2(A)\oplus A^{\oplus 3})\otimes \mathbb Z/2$ & 0 & $A\otimes \mathbb Z/2$ & 0 & 0\\
7 & \vline & $A\otimes \mathbb Z/2$ & 0 & 0 & 0 & 0\\
6 & \vline & 0 & 0 & 0 & 0 & 0\\

\end{tabular}
\end{table}
}

\subsection{Generalized spectral sequence and bifunctors}
Let $X$ be a simply connected simplicial set and $Y$ a finite
dimensional simplicial set. There is a spectral sequence with
initial term
$$
E_{r,q}^1=\bigoplus_i H^i(\Sigma^qY,
\pi_i(\gamma_r(GX)/\gamma_{r+1}(GX)))
$$
which depends only on homology of $X,Y$ and converges to the
graded group associated with filtration of the group
$[\Sigma^{q+1}Y,X]$. Here $GX$ is the Kan loop group of X. The
case when $Y$ is a sphere is the classical Curtis spectral
sequence, which converges to homotopy groups of $X$.

Consider a simple space $Y$ from the point of view of homology,
say $Y=M(A,1)$. Then
$$
H^i(\Sigma^qM(A,1),
\pi_i(\gamma_r(GX)/\gamma_{r+1}(GX)))=\begin{cases} Hom(A,
\pi_q(\gamma_r(GX),\gamma_{r+1}(GX))),\ i=q-1\\ Ext(A,
\pi_q(\gamma_r(GX),\gamma_{r+1}(GX))),\ i=q
\end{cases}
$$
Take $X=M(B,2)$ for an abelian group $B$. Then
$$
H^i(\Sigma^qM(A,1),
\pi_i(\gamma_r(GX)/\gamma_{r+1}(GX)))=\begin{cases}
Hom(A, L_q\EuScript L^r(B,1))\ i=q-1\\
Ext(A, L_q\EuScript L^r(B,1))\ i=q
\end{cases}
$$

The initial terms are the following:

\begin{tabular}{ccccccccccc}
 $q$ & \vline & $E_{1,q}^1$ & \vline & $E_{2,q}^1$ & \vline & $E_{3,q}^1$ & \vline & $E_{4,q}^1$\\
 \hline
$4$ & \vline & 0 & \vline & 0 & \vline & $Hom(A,L_1\EuScript L_s^3(A))$ & \vline & $Hom(A,L_4\EuScript L^4(B,1))$ \\
$3$ & \vline & 0 & \vline & $Hom(A,L_1\Gamma_2(B))$ & \vline & $Hom(A,\EuScript L_s^3(B))$ & \vline & $Ext(A,L_4\EuScript L^4(B,1)$ \\
$2$ & \vline & 0 & \vline & $Hom(A,\Gamma_2(B))$ & \vline & $Ext(A,\EuScript L_s^3(B))$ & \vline & 0 \\
$1$ & \vline & $Hom(A,B)$ & \vline & $Ext(A, \Gamma_2(B))$ &
\vline & 0 & \vline & 0
\end{tabular}

\vspace{.5cm}

We have immediately the sequence of Barratt:
$$
0\to Ext(A,\Gamma_2(B))\to [M(A,2),M(B,2)]\to Hom(A,B)\to 0
$$
which is not split. At the next step, we have an exact sequence
\begin{multline*}
[M(A,4),M(B,2)]\to Hom(A,L_1\Gamma_2(B))\to Ext(A,\EuScript
L_s^3(B))\to\\ [M(A,3),M(B,2)]\to Hom(A,\Gamma_2(B))\to 0
\end{multline*}

\begin{table}
{\small
\begin{tabular}{cccccccccccccccccccccc}
$k$ & \vline & $_3\pi_{k+3}S^3$ & \vline & $_3\pi_{k+4}S^4$ & \vline & $_3\pi_{k+5}S^5$ & \vline\\
 \hline
12 & \vline & 0 & \vline & $0$ & \vline & 0 & \vline\\
11 & \vline & $\mathbb Z/3$ & \vline & $\mathbb Z/3$ & \vline & $\mathbb Z/9$ & \vline\\
10 & \vline & $\mathbb Z/3$ & \vline & $\mathbb Z/3\oplus \mathbb Z/3$ & \vline & $\mathbb Z/9$ & \vline \\
9 & \vline & $0$ &
\vline & $0$ & \vline & 0 & \vline\\
8 & \vline & 0 & \vline & 0 & \vline & 0 & \vline\\
7 & \vline & $\mathbb Z/3$ & \vline & $\mathbb Z/3$ & \vline & $\mathbb Z/3$ & \vline\\
6 & \vline & $\mathbb Z/3$ & \vline & $\mathbb Z/3\oplus \mathbb Z/3$ & \vline & 0 & \vline\\
5 & \vline & 0 & \vline & 0 & \vline & 0 & \vline\\
4 & \vline & 0 & \vline & 0 & \vline & 0 & \vline
\\ 3 & \vline & $\mathbb Z/3$ & \vline &
$\mathbb Z/3$ & \vline & $\mathbb Z/3$ & \vline
\\ 2 & \vline & 0 & \vline & 0 & \vline & 0 & \vline\\
1 & \vline & 0 & \vline & 0 & \vline & 0 & \vline
\end{tabular}
} \vspace{.5cm} \caption{3-torsion in $\pi_{n+k}S^n$, $n=3,4,5$}
\end{table}

\begin{table}
{\small
\begin{tabular}{cccccccccccccccccccccc}
$k$ & \vline & $_3\pi_kM(A,3)$ & \vline & $_3\pi_kM(A,4)$ & \vline & $_3\pi_kM(A,5)$ & \vline\\
 \hline
12 & \vline & 0 & \vline & $\EuScript L^4(A)\otimes \mathbb Z/3$ & \vline & 0 & \vline\\
11 & \vline & $\EuScript N^{3;3}(A)\oplus \EuScript L^5(A)\otimes \mathbb Z/3$ & \vline & $A\otimes \mathbb Z/3$ & \vline & $\Lambda^2(A)\otimes\mathbb Z/3\oplus A\otimes \mathbb Z/9$ & \vline\\
10 & \vline & $A\otimes \mathbb Z/3$ & \vline & $(\Gamma_2(A)\oplus A)\otimes\mathbb Z/3$ & \vline & $A\otimes \mathbb Z/9$ & \vline \\
9 & \vline & $(\EuScript L^4(A)\oplus \Lambda^2(A))\otimes \mathbb
Z/3$ &
\vline & $\EuScript L_s^3(A)\otimes \mathbb Z/3$ & \vline & 0 & \vline\\
8 & \vline & 0 & \vline & 0 & \vline & 0 & \vline\\
7 & \vline & $\EuScript N^{3;3}(A)$ & \vline & $A\otimes \mathbb Z/3$ & \vline & $(\Lambda^2(A)\oplus A)\otimes \mathbb Z/3$ & \vline\\
6 & \vline & $A\otimes \mathbb Z/3$ & \vline & $(\Gamma_2(A)\oplus A)\otimes \mathbb Z/3$ & \vline & 0 & \vline\\
5 & \vline & $\Lambda^2(A)\otimes \mathbb Z/3$ & \vline & 0 & \vline & 0 & \vline\\
4 & \vline & 0 & \vline & 0 & \vline & 0 & \vline
\\ 3 & \vline & $A\otimes \mathbb Z/3$ & \vline &
$A\otimes \mathbb Z/3$ & \vline & $A\otimes \mathbb Z/3$ & \vline
\\ 2 & \vline & 0 & \vline & 0 & \vline & 0 & \vline\\
1 & \vline & 0 & \vline & 0 & \vline & 0 & \vline
\end{tabular}
} \vspace{.5cm} \caption{3-torsion in $\pi_{n+k}M(A,n)$ for $A$
free}
\end{table}

\section{Appendix A: Tables of derived functors}
\vspace{.5cm} 1. Derived functors $L_i\EuScript L^4(A,2),
L_i\EuScript L^8(A,2)$ for $A$ free, together with sets of
allowable words: \vspace{.4cm}
\begin{center}
 \renewcommand{\arraystretch}{1.2}
\begin{tabular}{|r||c|c|c|c|c|c|c|c}
\hline  & $L_i\EuScript L^4(A,2)$&$\mathcal W_{2,2}$&$L_i\EuScript
L^8(A,2)$&$\mathcal W_{2,3}$
\\ \hline 16 & 0 & 0 & $\EuScript L^8$ & 0
\\ \hline 15 & 0 & 0 & $\EuScript N^{4;2}$ & $(2,4,7)$
\\ \hline 14 & 0 & 0 & 0 & 0
\\ \hline 13 & 0 & 0 & $\EuScript N^{4;2}$ & $(2,4,5)$
\\ \hline 12 & 0 & 0 & $\Gamma_2\otimes \mathbb Z/2$ & $(2,3,5)$
\\ \hline 11 & 0 & 0 & $\EuScript N^{4;2}$ & $(2,4,3)$
\\ \hline 10 & 0 & 0 & $\Gamma_2\otimes \mathbb Z/2$ & $(2,3,3)$
\\ \hline 9 & 0 & 0 & $\Gamma_2\otimes \mathbb Z/2\oplus \EuScript N^{4;2}$ & $(2,2,3),(2,4,1)$
\\ \hline 8 & $\EuScript L^4$ & 0 & $\Gamma_2\otimes \mathbb Z/2\oplus A\otimes
\mathbb Z/2$ & $(2,3,1), (1,2,3)$
\\ \hline 7 & $\Gamma_2\otimes \mathbb Z/2$ & $(2,3)$ & $\Gamma_2\otimes \mathbb Z/2$ & $(2,2,1)$
\\ \hline 6 & 0 & 0 & $\Gamma_2\otimes \mathbb Z/2\oplus A\otimes \mathbb Z/2$ & $(2,1,1),(1,2,1)$
\\ \hline 5 & $\Gamma_2\otimes \mathbb Z/2$ & $(2,1)$ & $A\otimes \mathbb Z/2$ & $(1,1,1)$
\\ \hline i=4 & $A\otimes \mathbb Z/2$ & $(1,1)$ & 0 & 0
\\ \hline
\end{tabular}
\end{center}

\vspace{.4cm} 2. 3-torsion in derived functors $L_i\EuScript
L^n(A,2),$ for $i\leq 21, n\leq 27$ for $A$ free: \vspace{.4cm}

  {\tiny
 \renewcommand{\arraystretch}{1.8}
\begin{tabular}{|r||c|c|c|c|c|c|c|c|c}
\hline  & $n=6$&9&12&15&18&21&24&27
\\\hline \hline 21 & 0 & 0& 0 & $\EuScript L^5\otimes \mathbb Z/3$ & 0
& $\EuScript L^7\otimes\mathbb Z/3$ & 0 & $\EuScript N^{9;3}\oplus
(A/3A)^{\oplus 2}$
\\ \hline 20 & 0 & 0 & 0 & 0 & 0 & 0 & 0 & ${\EuScript N^{3;3}}^{\oplus 3}\oplus A/3A$
\\ \hline 19 & 0 & 0 & $\EuScript L^4\otimes \mathbb Z/3$ & 0 & $\EuScript N^{6;3}\oplus \Lambda^2\otimes \mathbb Z/3$ & 0& $\EuScript L^8\otimes \mathbb Z/3$ & $A/3A$
\\ \hline 18 & 0 & 0 & 0 & 0 & $\Lambda^2\otimes \mathbb Z/3$ & 0 &
0& 0
\\ \hline 17 & 0 & $\EuScript N^{3;3}$ & 0 & $\EuScript L^5\otimes \mathbb Z/3$ & 0 & $\EuScript L^7\otimes \mathbb Z/3$ & 0 & ${\EuScript N^{3;3}}^{\oplus 2}$
\\ \hline 16 & 0 & 0 & 0 & 0 & 0 & 0 & 0 & ${\EuScript N^{3;3}}^{\oplus 2}\oplus (A/3A)^{\oplus 2}$
\\ \hline 15 & 0 & 0 & $\EuScript L^4\otimes \mathbb Z/3$ & 0 & $\EuScript N^{6;3}\oplus \Lambda^2\otimes \mathbb Z/3$ & 0 & 0 & $(A/3A)^{\oplus 2}$
\\ \hline 14 & 0 & 0 & 0 & 0 & $(\Lambda^2\otimes \mathbb Z/3)^{\oplus 2}$ &
0& 0 & 0
\\ \hline 13 & 0 & $\EuScript N^{3;3}$ & 0 & $\EuScript L^5\otimes \mathbb Z/3$ & 0  & 0 & 0 & $\EuScript N^{3;3}$
\\ \hline 12 & 0 & $A/3A$ & 0 & 0 & 0 & 0 & 0 & $\EuScript N^{3;3}\oplus A/3A$
\\ \hline 11 & $\Lambda^2\otimes \mathbb Z/3$ & 0 &  $\EuScript L^4\otimes \mathbb Z/3$ & 0 & $\Lambda^2\otimes \mathbb Z/3$ & 0 & 0 & $A/3A$
\\ \hline 10 & 0 & 0 & 0 & 0 & $\Lambda^2\otimes \mathbb Z/3$ & 0 & 0
& 0
\\ \hline 9 & 0 & $\EuScript N^{3;3}$ & 0 & 0 & 0 & 0 & 0 & 0
\\ \hline 8 & 0 & $A/3A$ & 0 & 0 & 0 & 0 & 0 & 0
\\ \hline 7 & $\Lambda^2\otimes \mathbb Z/3$ & 0 & 0 & 0 & 0 & 0 & 0 & 0
\\ \hline
\end{tabular}
}

\vspace{.5cm} 3. 3-torsion in derived functors $L_i\EuScript
L^n(A,2),$ for $i\leq 21, n\leq 135$:

\vspace{.4cm} {\tiny
 \renewcommand{\arraystretch}{1.8}
\begin{tabular}{|r||c|c|c|c|c|c|c|c}
\hline  & $n=36$&45&54&63&81&108&135
\\\hline \hline 21 & 0 & $(\EuScript L^5\otimes \mathbb Z/3)^{\oplus 2}$& $(\Lambda^2\otimes \mathbb Z/3)^{\oplus 3}$ & $\EuScript L^7\otimes \mathbb Z/3$ &
${\EuScript N^{3;3}}^{\oplus 3}$ & $(\EuScript L^4\otimes \mathbb
Z/3)^{\oplus 3}$ & $\EuScript L^5\otimes \mathbb Z/3$
\\ \hline 20 & 0 & $(\EuScript L^5\otimes \mathbb Z/3)^{\oplus 2}$ & 0 & $\EuScript L^7\otimes \mathbb Z/3$ &
${\EuScript N^{3;3}}^{\oplus 6}\oplus (A/3A)^{\oplus 3}$ & 0 &
$(\EuScript L^5\otimes \mathbb Z/3)^{\oplus 2}$
\\ \hline 19 & $(\EuScript L^4\otimes \mathbb Z/3)^{\oplus 2}$ & 0 & $\EuScript N^{6;3}\oplus (\Lambda^2\otimes \mathbb Z/3)^{\oplus 2}$ & 0 & ${\EuScript N^{3;3}}^{\oplus 4}\oplus (A/3A)^{\oplus 5}$ &
$\EuScript L^4\otimes \mathbb Z/3$ & $\EuScript L^5\otimes \mathbb
Z/3$
\\ \hline 18 & $(\EuScript L^4\otimes \mathbb Z/3)^{\oplus 2}$ & 0 & $\EuScript N^{6;3}\oplus (\Lambda^2\otimes \mathbb Z/3)^{\oplus 5}$ &
0 & $(A/3A)^{\oplus 3}$ & $(\EuScript L^4\otimes \mathbb
Z/3)^{\oplus 2}$ & 0
\\ \hline 17 & 0 & $\EuScript L^5\otimes \mathbb Z/3$ & $(\Lambda^2\otimes\mathbb Z/3)^{\oplus 3}$ & 0 & $\EuScript N^{3;3}$ &
$\EuScript L^4\otimes \mathbb Z/3$ & 0
\\ \hline 16 & 0 & $\EuScript L^5\otimes \mathbb Z/3$ & 0 & 0 & ${\EuScript N^{3;3}}{\oplus 2}\oplus A/3A$ & 0 & 0
\\ \hline 15 & $\EuScript L^4\otimes \mathbb Z/3$ & 0 & $\Lambda^2\otimes \mathbb Z/3$ & 0 & $\EuScript N^{3;3}\oplus (A/3A)^{\oplus 2}$ & 0 & 0
\\ \hline 14 & $\EuScript L^4\otimes \mathbb Z/3$ & 0 & $(\Lambda^2\otimes \mathbb Z/3)^{\oplus 2}$ & 0 & $A/3A$ & 0 & 0
\\ \hline 13 & 0 & 0 & $\Lambda^2\otimes \mathbb Z/3$ & 0 & 0 & 0 & 0
\\ \hline
\end{tabular}
}

\vspace{.5cm} 4. 3-torsion in derived functors $L_i\EuScript
L^n(A,2),$ for $i\leq 21, n> 135$:

\vspace{.3cm}

{\tiny
 \renewcommand{\arraystretch}{1.8}
\begin{tabular}{|r||c|c|c|c|c|c|c|c}
\hline  & $n=162$&243&324&486&729
\\\hline \hline 21 & $\EuScript N^{6;3}\oplus (\Lambda^2\otimes \mathbb Z/3)^{\oplus 11}$ & $\EuScript N^{3;3}\oplus (A/3A)^{\oplus 4}$ & $(\EuScript L^4\otimes \mathbb Z/3)^{\oplus 3}$ &
$(\Lambda^2\otimes \mathbb Z/3)^{\oplus 6}$ & $\EuScript
N^{3;3}\oplus (A/3A)^{\oplus 4}$
\\ \hline 20 & $(\Lambda^2\otimes \mathbb Z/3)^{\oplus 4}$ & ${\EuScript N^{3;3}}^{\oplus 3}\oplus A/3A$ & $\EuScript L^4\otimes \mathbb Z/3$ &
$(\Lambda^2\otimes \mathbb Z/3)^{\oplus 4}$& $A/3A$
\\ \hline 19 & $\Lambda^2\otimes \mathbb Z/3$ & ${\EuScript N^{3;3}}^{\oplus 3}\oplus (A/3A)^{\oplus 3}$ & 0 & $\Lambda^2\otimes \mathbb Z/3$ & 0
\\ \hline 18 & $(\Lambda^2\otimes \mathbb Z/3)^{\oplus 3}$ & $\EuScript N^{3;3}\oplus (A/3A)^{\oplus 3}$ & 0 & 0 & 0
\\ \hline 17 & $(\Lambda^2\otimes \mathbb Z/3)^{\oplus 3}$ & $A/3A$ & 0 & 0 & 0
\\ \hline 16 & $\Lambda^2\otimes \mathbb Z/3$ & 0 & 0 & 0 & 0
\\ \hline
\end{tabular}
}

\end{document}